\documentclass[letterpaper,10pt]{article}
\usepackage{amsfonts}
\usepackage{amssymb, amsmath, amsthm,  graphicx}
\usepackage[margin=1in]{geometry}
\usepackage[bottom]{footmisc}
\usepackage{commath}
\usepackage[utf8]{inputenc}
\usepackage{natbib}
\usepackage{booktabs, multirow}
\usepackage{url}

\renewcommand{\baselinestretch}{1.5}

\DeclareMathOperator*{\argmin}{argmin}
\DeclareMathOperator*{\diam}{diam}
\newtheorem{thm}{Theorem}

\newtheorem{lemma}{Lemma}
\newtheorem{corollary}{Corollary}

\begin{document}

\title{Asymptotically Honest Confidence Regions for High Dimensional Parameters by the Desparsified Conservative Lasso
}
\author{\textsc{Mehmet Caner\thanks{%
Ohio State University, 452 Arps Hall, Department of Economics, Translational Data Analytics, Department of Statistics, OH 43210. Email:caner.12@osu.edu.}} \and \textsc{Anders Bredahl Kock\thanks{%
Oxford University, Aarhus University and CREATES, Department of Economics, Manor Road, Oxford, OX1 3UQ, UK. Email: anders.kock@economics.ox.ac.uk.
We would like to thank Victor Chernozhukov and Andrea Montanari for pointing us to relevant related research. The paper has also benefited tremendously from insightful comments by the co- and associate editor as well as the referees. Financial support from the Danish National Research Foundation is gratefully acknowledged by the second author (grant DNRF78). First version: October 2014.
}} }
\date{\today}

\maketitle

\begin{abstract}
In this paper we consider the conservative Lasso which we argue penalizes more correctly than the Lasso and show how it may be desparsified in the sense of \cite{van2014} in order to construct asymptotically honest (uniform) confidence bands. In particular, we develop an oracle inequality for the conservative Lasso only assuming the existence of a certain number of moments. This is done by means of the Marcinkiewicz-Zygmund inequality. We allow for heteroskedastic non-subgaussian error terms and covariates. Next, we desparsify the conservative Lasso estimator and derive the asymptotic distribution of tests involving an increasing number of parameters. Our simulations reveal that the desparsified conservative Lasso estimates the parameters more precisely than the desparsified Lasso, has better size properties and produces confidence bands with superior coverage rates. 


\noindent \textit{Keywords and phrases}: conservative Lasso, honest inference, high-dimensional data, uniform inference, confidence intervals, tests.

\noindent\textit{JEL codes}: C12, C13, C21.

\end{abstract}

\section{Introduction}
In recent years we have seen a burgeoning literature on high-dimensional problems where the number of parameters is much greater than the sample size. Statistical inference in the sense of constructing tests and confidence bands in the high-dimensional linear regression model were considered in a seminal series of papers by \cite{belloni2010sparse,vicecta2012, vicrestud2012, belloni2014inference,belloni2011inference}. These authors showed how a cleverly constructed (double) post selection estimator can be used to construct uniformly valid confidence intervals for the parameter of interest in instrumental variable and treatment effect models allowing for imperfect model selection in the first step. Also \cite{fanliao15} show how to set up test statistics in high dimensions with power enhancing components  against sparse alternatives.
 \cite{nickl2013confidence} consider honest adaptive inference when $p > n$. This can be obtained as long as the rate of sparse estimation does not exceed $n^{-1/4}$. 
\cite{hoffmann2011adaptive} consider the existence of honest adaptive confidence bands for an unknown
density function. They show that this is possible if the non-parametric hypotheses for the
null and alternative are asymptotically consistently distinguishable. \cite{berk2013valid} propose
a conservative post selection inference method. The idea is simultaneous inference in all
models' submodels and this results in very wide confidence intervals. \cite{taylor2015statistical} discuss a practical way of taking into account the model selection's 
effect on post selection inference. \cite{tibshirani2011regression} provides a nice summary of developments in the literature while \cite{lockhart2014significance} provide a computation based significance test for Lasso estimators. Also \cite{zouli08} and \cite{fanfan14} used adaptive weights in Lasso type estimators that enhance model selection.


The paper closest in spirit to ours is \cite{van2013asymptotically, van2014} who cleverly showed how the classical Lasso estimator may be \textit{desparsified} to construct asymptotically valid confidence bands for a low-dimensional subset of a high-dimensional parameter vector. This paper in turn is related to \cite{zhang2014confidence}, \cite{javanmard2013hypothesis} and \cite{javanmard2013confidence}. The idea behind desparsification is to remove the bias introduced by shrinkage by desparsifying the estimator using a cleverly constructed approximate inverse of the non-invertible empirical Gram matrix. Furthermore, these confidence bands do not suffer from the critique of \cite{potscher2009confidence} regarding the overly large size of confidence bands based on consistent variable selection techniques. By using the desparsified Lasso to construct confidence bands and tests, \cite{van2014} strike a middle ground between classical low dimensional inference, which relies heavily on testing, and Lasso-type techniques which perform estimation and variable selection in one step without any testing. 

In the framework of the high-dimensional linear regression model and inspired by the work of \cite{van2014} we study the so-called conservative Lasso. The important observation here is that, in the presence of an oracle inequality on the plain Lasso, the penalty of the conservative Lasso on the non-zero parameters will be no larger than the one for the Lasso while the penalty on the zero parameters will be the same as the one induced by the plain Lasso. Hence, the conservative Lasso may be expected to deliver more precise parameter estimates (in finite samples) than the Lasso. And indeed, our theoretical results and simulations strongly indicate that this is the case. Also note that recently \cite{fanx2014} proposed a weighted $\ell_1$ penalized estimator with very similar weights. Their focus is on strong oracle optimality and we show that a variant of our conservative Lasso possesses the strong oracle optimality property.

We provide an oracle inequality for the conservative Lasso estimator and use the method of desparsification introduced in \cite{van2014}. This approach has the advantage that the zero and non-zero coefficients do not have to be well-separated (no $\beta_{\min}$-condition is imposed) in order to conduct valid inference. We only assume the existence of $r$ moments as opposed to the classical sub-gaussianity assumption. The oracle inequalities rely on the use of the Marcinkiewicz-Zygmund inequality which we argue delivers slightly more precise estimates than Nemirovski's inequality.

We also show that hypotheses involving an increasing number of parameters can be tested (we are considering a \textit{fixed} sequence of hypotheses) which generalizes the results on hypotheses involving a bounded number of parameters in \cite{van2014}. Furthermore, we allow for heteroskedastic error terms and provide a uniformly consistent estimator of the high-dimensional asymptotic covariance matrix. This is an important generalization in practical problems as heteroskedasticity is omniscient in econometrics and statistics.  A similar approach could be of interest in large linear panel data models under strict exogeneity.

The simulations show that vast improvements can be obtained by using the  desparsified conservative Lasso as opposed to the plain desparsified Lasso. To be precise, the true parameter $\beta_0$ is in general estimated much more precisely and $\chi^2$-tests based on the desparsified conservative Lasso have much better size properties (and often also higher power) than their counterparts based on the desparsified Lasso. 

When implementing Lasso-type estimators the choice of tuning parameter is important. Thus, in Theorem  \ref{thm5} in the appendix, we show how the method of \cite{fanx2014} can be used to choose the tuning parameter of the variant of the conservative Lasso when the objective is consistent model selection in high dimensions.

The rest of the paper is organized as follows. Section \ref{model} introduces the model and the conservative Lasso. Section \ref{Desp} introduces nodewise regression, desparsification, and the approximate inverse to the empirical Gram matrix. Section \ref{Inference} introduces inference and establishes honest confidence intervals and shows that they contract at the optimal rate. The simulations can be found in Section \ref{MC}. Section \ref{Co} concludes the paper.
All proofs are deferred to the appendix

\section{The Model}\label{model}

Before stating the model setup we introduce some notation used throughout the paper.

\subsection{Notation}
For any real vector $x$, we let $\enVert[0]{x}_q$ denote the  $\ell_q$-norm. We will primarily use the $\ell_1$-, $\ell_2$-, and the $\ell_\infty$-norm. For any $m\times n$ matrix $A$, we define $\enVert[0]{A}_\infty=\max_{1\leq i\leq m, 1\leq j\leq n}|A_{i,j}|$. Occasionally we shall also use the induced $\ell_\infty$-norm. This will be denoted by $\enVert[0]{A}_{\ell_\infty}$ and equals the maximum absolute row sum of $A$. For any symmetric matrix $B$, let $\phi_{\min}(B)$ and $\phi_{\max}(B)$ denote the smallest and largest eigenvalue of $B$, respectively. If $x\in\mathbb{R}^n$ and $S$ is a subset of $\cbr[0]{1,...,n}$ we let $x_S$ be the modification of $x$ that places zeros in all entries of $x$ whose index does not belong to $S$. For an $n\times n$ matrix $B$ let $B_S$ denote the submatrix of $B$ consisting only of the rows and columns indexed by $S$. If $S=\cbr[0]{j}$ is a singleton set, we use $B_j$ as shorthand for the $j$'th diagonal element of $B$.

For any set $S$, let $|S|$ denote its cardinality and for $x\in \mathbb{R}^n$ its prediction norm is defined as $\enVert[0]{x}_n=\sqrt{\frac{1}{n}\sum_{i=1}^nx_i^2}$. $\stackrel{d}{\to}$ will indicate convergence in distribution and $o_p(a_n)$ as well as $O_p(b_n)$ are used in their usual meaning for sequences $a_n$ and $b_n$. $a_n\asymp b_n$ means that these sequences differ at most by strictly positive multiplicative constants.

\subsection{The model}
We consider the model
\begin{equation}
Y=X\beta _{0}+u\text{,}  \label{0}
\end{equation}%
where $X$ is the $n\times p$ matrix of explanatory variables and $u$ is a vector of error terms. $\beta _{0}$ is the $p\times 1$ population regression coefficient which we shall assume to be sparse. However, the location of the non-zero coefficients is unknown and potentially $p$ could be much greater than $n$. The sparsity assumption can be replaced by a weak sparsity assumption as we shall make precise after Theorem \ref{thm1} below. We assume that the explanatory variables are exogenous and precise assumptions will be made in Assumption 1 below. Let $S_0=\cbr[0]{j:\beta_{0,j}\neq 0}$ and $s_0=|S_0|$. For later purposes define $X_j$ as the $j$'th column of $X$ and  $X_{-j}$ as all columns of $X$ except for the $j$'th one.

\subsection{The conservative Lasso and comparison to (adaptive) Lasso}\label{CLsub}

The conservative Lasso is a two-step estimator defined as the weighted Lasso 
\begin{align}
\hat{\beta}=\argmin_{\beta\in\mathbb{R}^p} \{ \enVert[1]{Y-X\beta}_n^2+2\lambda_n\sum_{j=1}^p\hat{w}_j\envert[0]{\beta_j} \}\label{ConsLassoObj}
\end{align}
with weights $\hat{w}_j = \frac{\lambda_{prec}}{|\hat{\beta}_{L,j}| \vee \lambda_{prec}}$ where $\hat{\beta}_L$ is the plain Lasso estimator which is used to construct the weights $\hat{w}_j$. The plain Lasso corresponds to $w_j=1$ for $j=1,...,p$ in (\ref{ConsLassoObj}). Here $\lambda_n$ and $\lambda_{prec}$ are positive non-random quantities chosen by the researcher which we shall be specific about shortly. In Lemma \ref{lemma9} and the simulation section we show that $\lambda_{prec}$ can be chosen as an estimable multiple of $\lambda_n$. Hence, the only tuning parameter is $\lambda_n$. We choose $\lambda_n$ by  either BIC or the Generalized Information Criterion (GIC) of \cite{fantang13}. Details are provided in the Monte Carlo section. A theorem tying GIC to model selection consistency of a variant of our conservative Lasso (which will be described in the next subsection) is at the end of Appendix B.

As opposed to the adaptive Lasso, the conservative Lasso gives variables that were excluded by the first step initial Lasso estimator a second chance --- even if $|\hat{\beta}_{L,j}|=0$ one has $\hat{w}_j=1$ instead of an ``infinitely" large penalty. Hence, the name ``conservative" Lasso. The adaptive Lasso usually performs its worst when a relevant variable has been left out by the initial Lasso estimator. The conservative Lasso rules out such a situation while still using more intelligent weights than the Lasso as we shall see shortly. Note that our definition of the conservative Lasso is at first glance slightly different from the one on page 205 in \cite{bvdg2011}. 

We shall choose $\lambda_{prec}$ to equal an upper bound on the estimation error of the first step Lasso for reasons to be made clear next. In particular, assume that $\lambda_{prec}$ is such that $\mathcal{C}_1=\cbr[1]{\enVert[0]{\hat{\beta}_L-\beta_0}_{\infty} \leq \lambda_{prec}}$ is a set with large probability.  Lemma \ref{lemma9} in the Appendix provides a concrete choice of $\lambda_{prec}$ ensuring that $\mathcal{C}_1$ occurs with high probability. 
In Theorem \ref{thm1} below we shall give examples of $\lambda_{prec}$. 

Recently \cite{fanx2014} proposed a one step solution to folded concave penalized estimation of which a subcase is the SCAD of \cite{fan2001variable}. This weighted $\ell_1$ penalty approach is similar to our conservative Lasso. Unlike our fractional weight structure their weights are normalized and truncated by a multiple of $\lambda_n$. Like us, \cite{fanx2014} also solve the zero denominator issue of the adaptive Lasso, as pointed out by \cite{fanl2008} and \cite{fan2010}. However, their paper's emphasis is on strong oracle optimality, which we shall discuss in more details when introducing our variant of the conservative Lasso in the next subsection, while we are interested in constructing tests and confidence bands. 


As is standard in the literature we assume that the covariates $X_i$ are i.i.d.\ with $\Sigma=E(X_1 X_1')$ satisfying an adaptive restricted eigenvalue condition: 
\begin{align}
\phi_{\Sigma}^2(s)= \min_{\substack{ \delta\in \mathbb{R}^{p}\setminus\left\{0\right\}\\ \enVert{\delta_{{S}^c}}_{1}\leq 3\sqrt{s}\hspace{.05cm} \enVert{\delta_{S}}_{2}}}\frac{\delta'\Sigma\delta}{\enVert{\delta_{S}}_2^2}>0,\label{REs}
\end{align}
where $S\subseteq \left\{1,...,p\right\}$. Instead of minimizing over all of $\mathbb{R}^{p}$, the minimum in (\ref{REs}) is restricted to those vectors which satisfy $\enVert{\delta_{{S}^c}}_{1}\leq 3 \sqrt{s} \enVert{\delta_{S}}_{2}$. Thus, the adaptive restricted eigenvalue condition is satisfied in particular when $\Sigma$ has full rank.


In order to establish an oracle inequality for the conservative Lasso we shall assume the following.

{\bf Assumption 1}. {\it The covariates $X_{i}\in\mathbb{R}^p,\ i=1,...,n$ are independently and identically distributed while the error terms $u_i\in\mathbb{R},\ i=1,...,n$ are independently distributed with $E(u_i|X_i)=0$. Furthermore,  $\max_{1\leq j\leq p}E|X_{1,j}|^r\leq C$ and $\max_{1\leq i\leq n}E|u_i|^r\leq C$ for some $r\geq 2$ and a positive universal constant $C$. $\phi^2_\Sigma(s_0)$ is bounded away from 0.}

Assumption 1 states that the covariates are independently and identically distributed with uniformly bounded $r$'th moments. The assumption of identical distribution of the covariates is mainly made to keep expressions simple but could be relaxed. We will comment in more detail on this later. The error terms are allowed to be non-identically distributed and may, in particular, be conditionally heteroskedastic. Thus, many applications of interest are covered. At this point it is also worth mentioning that in the literature one often assumes that the covariates as well as the error terms are uniformly sub-gaussian. This is a much stronger assumption than the one imposed here and rules out data with heavy tails. However, strengthening our assumption to sub-gaussianity would not cause any trouble and deliver stronger results. In particular, all powers of $p$ below could be replaced by powers of $\log(p)$ which are asymptotically much smaller. A third route which is sometimes taken is to assume the covariates to be bounded, the error terms to possess bounded second moments and then use Nemirovski's inequality to obtain oracle inequalities which only depend on $p$ through its logarithm.  


Define $\Theta = \Sigma^{-1}$, and $\| \hat{w}_{S_0} \|_{\infty} = \max_{j \in S_0} | \hat{w}_j|$, which is the maximal weight among all the relevant variables. We are now ready to state the oracle inequality for the weighted Lasso estimator in (\ref{ConsLassoObj}).
\begin{thm}\label{thm1}
Let Assumption 1 be satisfied, set $\lambda_n=M\frac{p^{2/r}}{n^{1/2}}$ for $M>0$ and $\lambda_{prec}=\frac{9 \lambda_n}{4} \|\Theta\|_{l_{\infty}}$.
Then, with probability at least $1-\frac{C}{M^{r/2}}-D\frac{p^2s^{r/2}_0}{n^{r/4}}$, the conservative Lasso satisfies the following inequalities
\begin{align}
\| X (\hat{\beta} - \beta_0) \|_n^2 \label{IQ1thm1}
\le
 2 (2 \enVert[0]{\hat{w}_{S_0}}_{\infty} +1)^2 \frac{ \lambda_n^2s_0 }{\phi^2_\Sigma (s_0)},\\
\|\hat{\beta} - \beta_0 \|_1 
\le 
4 (\enVert[0]{\hat{w}_{S_0}}_{\infty} +1) (2 \enVert[0]{\hat{w}_{S_0}}_{\infty} +1) \frac{\lambda_ns_0}{\phi^2_\Sigma (s_0)},\label{IQ2thm1}
\end{align}
for universal constants  $C, D>0$. Furthermore, these bounds are valid uniformly over the $\ell_0$-ball $\mathcal{B}_{\ell_0}(s_0)=\cbr[1]{\enVert{\beta_{0}}_{\ell_0}\leq s_0}$.
\end{thm}

\textbf{Remarks}. 

1. For the Lasso $\enVert[0]{\hat{w}_{S_0}}_{\infty}=1$. Thus, the upper bound in (\ref{IQ1thm1}) takes the value $18  \frac{ \lambda_n^2s_0 }{\phi^2_\Sigma (s_0)}$. For the conservative Lasso we always have $0<\enVert[0]{\hat{w}_{S_0}}_{\infty}\leq 1$ such that the upper bound is no worse than for the Lasso. In fact, the multiplicative constant 18 can be considerably improved in certain settings. To give a concrete example consider Lemma \ref{lt1}(iii) where $\enVert[0]{\hat{w}_{S_0}}_{\infty}\to 0$ on $\mathcal{C}_1=\cbr[1]{\enVert[0]{\hat{\beta}_L-\beta_0}_\infty\leq \lambda_{prec}}$ \footnote{$\mathcal{C}_1$ occurs whenever the event in Theorem \ref{thm1} having probability at least $1-\frac{C}{M^{r/2}}-D\frac{p^2s^{r/2}_0}{n^{r/4}}$ occurs. Thus, we are not working on a smaller event than for the Lasso.} Therefore, the upper bound in (\ref{IQ1thm1}) approaches $2 \frac{ \lambda_n^2s_0 }{\phi^2_\Sigma (s_0)}$ which is 9 times smaller than the bound for the Lasso. Note that Lemma \ref{lt1} (iii) relies on a $\beta_{\min}$-type condition. However, even without this condition, one has $\enVert[0]{\hat{w}_{S_0}}_{\infty}\leq 1$ implying upper bounds for the conservative Lasso that are no worse than the ones for the plain Lasso.

%
%

2. Next consider $\ell_1$ error bounds for the estimation error. For the Lasso the upper bound in (\ref{IQ2thm1}) is $24 \frac{\lambda_ns_0}{\phi^2_\Sigma (s_0)}$ and when $\max_{j \in S_0} \hat{\omega}_j \to 0$, (\ref{IQ2thm1}) approaches $ 4 \frac{\lambda_ns_0}{\phi^2_\Sigma (s_0)}$ for the conservative Lasso by Lemma \ref{lt1} (iii).

3. To simplify the notation in future lemmas and proofs, define $d_{n1}= 2 (2 \enVert[0]{\hat{w}_{S_0}}_{\infty} +1)^2$ and $d_{n2}= 4 (\enVert[0]{\hat{w}_{S_0}}_{\infty}+1) (2 \enVert[0]{\hat{w}_{S_0}}_{\infty} +1)$. 

4. $\lambda_{prec} = \frac{9 \lambda_n}{4} \|\Theta \|_{l_{\infty}}$, and in the simulation section we provide a consistent estimator of 
$\|\Theta \|_{\ell_{\infty}}$. This choice of $\lambda_{prec}$ is motivated by Lemma \ref{lemma9} in the appendix which shows that $\lambda_{prec}$ is a high probability upper bound on the $\ell_\infty$ estimation error of the Lasso. 

Finally, the sparsity assumption on $\beta_0$ can be replaced by a bound on $\sum_{j=1}^p|\beta_{0,j}|^q$ for $0<q<1$ as it is not difficult to establish oracle inequalities in such a ``weakly sparse" setting. Thus, none of the entries of $\beta_0$ need to equal exactly zero but we stick to the classical $\ell_0$-sparsity here.

Define $S_j=\cbr[1]{k=1,...,p: \Theta_{j,k}\neq 0}$ as the indices of the non-zero entries of the $j$th row of $\Theta_j$. Let $s_j=|S_j|$.   Define also $\eta_j = X_j - X_{-j} \gamma_j$, which is a $n \times 1$ vector.

{\bf Assumption 2}:
\begin{itemize}
\item[a)] $\phi_{\min}(\Sigma)$ is bounded away from zero.
\item[b)] $\frac{p^2 (\max(s_0, \max_{1 \le j \le p} s_j))^{r/2}}{n^{r/4}}\to 0$. 
\item[c)] $E(|\eta_{j,i}|^r)$ uniformly bounded over $i=1,...,n$ and $j=1,...,p$.
\end{itemize}
Assumption 2a) states that the smallest eigenvalue of the \textit{population} covariance matrix is bounded away from zero. It is used to make sure that $\tau_j^2 = 1/\Theta_{j,j}\geq 1/\phi_{\max}(\Theta)=\phi_{\min}(\Sigma)$ are bounded away from zero. Part b) is needed to show that $\enVert[1]{\hat{\Sigma}-\Sigma}_\infty$ converges to zero sufficiently fast to conclude that the adaptive restricted eigenvalue of $\hat{\Sigma}=\frac{1}{n}X'X$ is close to the one of $\Sigma$. It implies an upper bound on how fast the dimension, $p$, of the model can increase. The more moments one assumes the covariates and the error terms to possess, the faster can $p$ grow. From Assumption 2b), it is clear that since $p \ge \max(s_0, \max_{1 \le j \le p} s_j) $, $\max(s_0, \max_{1 \le j \le p} s_j) = o(n^{\frac{r}{2(r+4)}})$. This restricts the number of non-zero coefficients in the model and each row of $\Theta$. On the other hand, if the error terms and covariates are subgaussian, it is not difficult to show that this requirement is relaxed to $\max(s_0, \max_{1 \le j \le p} s_j) =o(\sqrt{n})$. Intuitively this can be seen by letting $r\to\infty$. Nevertheless, the inverse covariance matrix must be sparse. This is satisfied if $\Sigma$ is e.g. block diagonal or has the Toeplitz structure $\Sigma_{i,j}=\rho^{|i-j|},\ -1<\rho<1$. In the simulations we shall also see that our method works well even if $\Theta=\Sigma^{-1}$ is not sparse as long as its entries are not too far from zero. This is not surprising as the sparsity assumption can easily be relaxed to the weak sparsity assumption of $\sum_{l=1}^p|\Omega_{j,l}|^q$ not being too large for any $j\in H$ for some $0<q<1$ as similar bounds to the ones in Lemma \ref{Node} below remain valid under this assumption. Thus, no entry of $\Theta$ needs to be zero as long as each row can be well approximated by a sparse vector. This observation was also made in \cite{yuan2010high} for a different estimator of $\Theta$.



Observe that $\hat{w}_j\leq 1$ for all $j=1,...,p$. We now provide a lemma that shows desirable properties of the weights of the conservative Lasso. We caution that the fact that the weights of the non-zero coefficients approaching zero in (iii) comes at the pice of a $\beta_{\min}$ type of condition which rules out very small coefficients. This kind of assumption may not be suitable in economics. Without this type of condition the result in (iii) will not be true, however the weights of the non-zero coefficients of the conservative Lasso will still be smaller than for the plain Lasso.
\begin{lemma}\label{lt1}
Under Assumptions 1-2, with $\lambda_n = M \frac{p^{2/r}}{n^{1/2}}$ for $M>0$, and $\lambda_{prec} = \frac{9 \lambda_n}{4} \|\Theta\|_{\ell_{\infty}}$. Then, 

(i). \[ \lambda_{prec} \to 0,\]

\noindent and on $\mathcal{C}_1=\cbr[1]{\enVert[0]{\hat{\beta}_L-\beta_0}_{\infty} \leq \lambda_{prec}}$, with $P(\mathcal{C}_1)\to 1$, the following two statements hold

(ii). \[ \min_{j\in S_0^c}|\hat{w}_j| = 1,\]

(iii). In addition, if  $\min_{ j \in S_0} |\beta_{0,j}|/\lambda_{prec}\to \infty$ then 
\[  \max_{j\in S_0}|\hat{w}_j| \to 0,\]
\end{lemma}

\textbf{Remarks.} 

1. Lemma \ref{lt1} shows that $\lambda_{prec}= O(\lambda_n \sqrt{\max_{1 \le j \le p} s_j})=o(1)$. Its proof reveals that even if we replace $\|\Theta\|_{\ell_{\infty}}$ by $\|\hat{\Theta}_L \|_{\ell_{\infty}}$ we still get $\lambda_{prec} \stackrel{p}{\to}0$. Note that $\hat{\Theta}_L$ represents the Lasso nodewise regression estimate of the inverse matrix of $\Sigma^{-1}$. It is a subcase of our conservative nodewise regression in Section \ref{AI}, and explained in footnote 4 there.

2. An even better result can be achieved in terms of the conditions needed for $\lambda_{prec}$ to converge to zero. If, for example, $\Sigma$ is an equicorrelation matrix then $\|\Theta \|_{\ell_{\infty}} = O(1)$ by Example 2.5.1 of  \cite{vdG14}. Thus, $\lambda_{prec} = O (\lambda_n)  = o(1)$. The same is the case if $\Sigma_{i,j}=\rho^{|i-j|}$ for some $-1<\rho<1$ which is an often considered structure.

3. Parts (ii) and (iii) of the lemma show that asymptotically no penalty is applied to the coefficients of the relevant variables while the same penalty as for the Lasso is applied to the coefficients of the irrelevant variables. In particular, it is guaranteed that all non-zero coefficients are penalized less than all zero coefficients.


We shall see in Section \ref{MC} that the above advantages of the conservative Lasso over the plain Lasso materialize in better performance also in the simulations. 


\subsection{A Variant of the Conservative Lasso}
In this section we introduce a variant of the conservative Lasso estimator. This variant possess the property of strong oracle optimality under slightly stronger conditions than Assumptions 1 and 2, see Theorem \ref{thm4} at the end of Appendix B. Strong oracle optimality is defined as an estimator being equal to the oracle estimator with probability approaching one (p.822 of \cite{fanx2014}). As the plain Lasso is generally not strongly oracle optimal this shows superiority of the variant of the conservative Lasso to the former. First, we define the variant of the conservative Lasso as 
\begin{equation}
\tilde{\beta} = \argmin_{\beta \in \mathbb{R}^p} \{ \|Y - X \beta \|_n^2 + 2 \lambda_n \sum_{j=1}^p \tilde{w}_j | \beta_j|\},\label{vcl}
\end{equation}
where $\tilde{w}_j = 1_{ \{ |\hat{\beta}_{L,j}| \le \lambda_{prec} \} }$. In this new variant the weights only take the values 0 or 1. Importantly, this is a variant of the conservative Lasso since all variables still get a second chance after the first step Lasso estimation.

Strong oracle optimality of (\ref{vcl}) is established in Theorem \ref{thm4} at the end of Appendix B. Theorem \ref{thm4} (i) shows that $\min_{j \in S_0^c}\tilde{w}_j =1$ and $\max_{j\in S_0}\tilde{w}_j =0$ with with probability approaching one. Thus, in this variant of  the conservative Lasso the weights pertaining to the non-zero coefficients will be \textit{exactly} equal to zero with probability approaching one while for the conservative Lasso these weights only converge to zero with probability approaching one. This slightly stronger property contributes to obtaining the strong oracle property for the variant of the conservative Lasso in the proof of Theorem \ref{thm4}. Note, however, that Theorem \ref{thm4}(i) comes with a $\beta_{\min}$ type of condition similar to the one in Lemma \ref{lt1}(iii). However, under Assumption 1, Theorem \ref{thm1} above holds also when $\tilde{\beta}$ replaces $\hat{\beta}$. The same holds for Theorems \ref{thm2} and \ref{thm3} if one desparsifies $\tilde{\beta}$ instead of $\hat{\beta}$.

\section{Desparsification}\label{Desp}


\subsection{The Desparsified Conservative Lasso}

\noindent In order to conduct inference we shall use the idea of desparsification proposed in \cite{van2014}. The idea is that the shrinkage bias introduced due to the presence of penalization in (\ref{ConsLassoObj}) will show up in the properly scaled limiting distribution of $\hat{\beta}_{j}$. Hence, we remove this bias prior to conducting statistical inference. Letting $\hat{W}=diag\del[1]{\hat{w}_1,...,\hat{w}_p}$ be a $p\times p$ diagonal matrix containing the weights of the conservative Lasso, the first order condition of (\ref{ConsLassoObj}) may be written as
\[ -X' (Y- X \hat{\beta} )/n + \lambda_n \hat{W} \hat{\kappa} =0, \]
with $ \| \hat{\kappa} \|_{\infty} \le 1,$
and $\hat{\kappa}_j = sign(\hat{\beta}_j) $ if $\hat{\beta}_j \neq 0$ for $j=1,...,p$. Thus, 
\begin{equation}
 \lambda \hat{W} \hat{\kappa} =X' (Y - X \hat{\beta})/n.\label{kt1}
 \end{equation}
Next, as $Y = X \beta_0 + u$ and defining $\hat{\Sigma} = X'X/n$, the above display yields
\[ \lambda_n \hat{W} \hat{\kappa} + \hat{\Sigma} (\hat{\beta} - \beta_0) = X'u/n.\]
In order to isolate $\hat{\beta}-\beta_0$ we need to invert $\hat{\Sigma}$. However, when $p>n$, $\hat{\Sigma}$ is not invertible. Thus, the idea is now to construct an approximate inverse, $\hat{\Theta}$, to $\hat{\Sigma}$ and control the error term resulting from this approximation. We shall be explicit about the construction of $\hat{\Theta}$ in the next section. For any $p\times p$ matrix we may write, by multiplying the above equation by $\hat{\Theta}$, and adding $\hat{\beta}- \beta_0$ to each side of the above equation,
\begin{equation}
\hat{\beta} = \beta_0 -\hat{\Theta} \lambda_n \hat{W} \hat{\kappa}+\hat{\Theta} X'u/n - \Delta/n^{1/2},\label{2.1}
\end{equation}
where 
\[ \Delta = \sqrt{n} ( \hat{\Theta }  \hat{\Sigma} - I_p ) ( \hat{\beta} - \beta_0),\]
is the error resulting from using an approximate inverse, $\hat{\Theta}$, as opposed to an exact inverse.
We shall show that $\Delta$ is asymptotically negligible. Adding $\hat{\Theta} \lambda_n \hat{W} \hat{\kappa}$ to both sides of (\ref{2.1}) results in the following estimator by using (\ref{kt1})
\begin{eqnarray}
\hat{b} 
&=&
\hat{\beta} + \hat{\Theta} \lambda_n \hat{W} \hat{\kappa} = \hat{\beta}+ \hat{\Theta} X' (Y - X \hat{\beta})/n \label{despClasso} \\  
&=& 
\beta_0 + \hat{\Theta} X'u/n-\Delta/n^{1/2}.\label{2.2}
\end{eqnarray}
Hence, for any $p\times 1$ vector $\alpha$ with $\enVert[0]{\alpha}_2=1$ we can consider
\begin{align}
\sqrt{n}\alpha'\del[1]{\hat{b}-\beta_{0}}
=
\alpha'\hat{\Theta} X'u/n^{1/2}-\alpha'\Delta\label{stat},
\end{align}
such that a central limit theorem for $\alpha'\hat{\Theta} X'u/n^{1/2}$ and a verification of asymptotic negligibility of $\alpha'\Delta$ will yield asymptotically gaussian inference. Furthermore, we provide a uniformly consistent estimator of the asymptotic variance of $\sqrt{n}\alpha'\del[1]{\hat{b}-\beta_{0}}$ which makes inference practically feasible. In connection with Theorem \ref{thm2} we shall see that a central limit theorem for $\alpha'\hat{\Theta} X'u/n^{1/2}$ puts certain limitations on the number of non-zero entries of $\alpha$ in (\ref{stat}), i.e. the number of parameters involved in a hypothesis. A leading special case of the above setting is of course $\alpha=e_j$ where $e_j$ is the $j$'th unit vector for $\mathbb{R}^p$. Then, (\ref{stat}) reduces to
\begin{align}
\sqrt{n}\del[1]{\hat{b}_j-\beta_{0,j}}
=
\del[1]{\hat{\Theta} X'u/n^{1/2}}_j-\Delta_j.\label{stat2}
\end{align}
In general, let $H=\cbr[1]{j=1,...,p: \alpha_j\neq 0}$ with cardinality $h=|H|$. Thus, $H$ contains the indices of the coefficients involved in the hypothesis being tested. We shall allow for $h\to \infty$ as the first in the literature on inference in high-dimensional regression models with more regressors than observations ($p>n$) but require $h/n \to 0$ as $n \to \infty$.

In the next section we construct the approximate inverse $\hat{\Theta}$ which enters in both terms in the above display and thus plays a crucial role for the limiting inference. 
The above desparsification procedure is similar in spirit to the one outlined in \cite{van2014}. However, $\hat{\beta}$ is used instead of $\hat{\beta}_L$. Furthermore, the construction of the approximate inverse $\hat{\Theta}$ in the next section relies on the conservative Lasso as opposed to the plain Lasso.


\subsection{Constructing the Approximate Inverse of the Gram Matrix: $\hat{\Theta}$}\label{AI}

In this subsection we construct the approximate inverse $\hat{\Theta}$ of $\hat{\Sigma}$. This is done by nodewise regression a la \cite{meinshausen2006high} and \cite{van2014}. To formally define the nodewise regression recall that $X_j$ is the $j$'th column in $X$ and $X_{-j}$ all columns of $X$ except for the $j$'th one. First, along the lines of \cite{van2014} we define the Lasso nodewise regression estimates%
\begin{align}
\hat{\gamma}_{L,j}=\argmin_{\gamma\in\mathbb{R}^{p-1}} \enVert[1]{X_j-X_{-j}\gamma}_n^2+2\lambda_{node, n}\sum_{k\neq j}|\gamma_k|\label{NodeLObj}
\end{align}
for each $j=1,...,p$. We use  these estimates to construct the weights of the conservative Lasso nodewise regression which is defined as follows 
\begin{align}
\hat{\gamma}_{j}=\argmin_{\gamma \in\mathbb{R}^{p-1}} \enVert[1]{X_j-X_{-j}\gamma}_n^2+2\lambda_{node, n}\enVert[1]{\hat{\Gamma}_j\gamma}_1,\label{NodeCLObj}
\end{align}
where $\hat{\Gamma}_j=diag\del[1]{\frac{\lambda_{prec}}{|\hat{\gamma}_{L,l}|\vee \lambda_{prec}},\ l=1,...,p,l\neq j}$ is a $(p-1)\times (p-1)$ matrix of weights. \footnote {For the variant of the conservative Lasso we have
$\tilde{\Gamma}_j=diag \left( 1_{ \{ |\hat{\gamma}_{L,l}|\le \lambda_{prec} \}}, \ l=1,...,p,l\neq j \right)$.}

Note that we choose $\lambda_{node, n}$  to be the same in all of the nodewise regressions. This is needed for the uniform results in Lemma \ref{Node} below to be valid. Thus, the conservative Lasso is run $p$ times as an intermediate step to construct $\hat{\Theta}$. Using the notation $\hat{\gamma}_j=\cbr[1]{\hat{\gamma}_{j,k};\ k=1,...,p,\ k\neq j}$ we define
\begin{equation}
 \hat{C} = \left( \begin{array}{cccc}
			1 & -\hat{\gamma}_{1,2} &  \cdots & -\hat{\gamma}_{1,p} \\
			-\hat{\gamma}_{2,1} & 1 & \cdots & -\hat{\gamma}_{2,p} \\
			\hdots & \hdots & \ddots & \hdots \\
			-\hat{\gamma}_{p,1}&  -\hat{\gamma}_{p,2} &  \cdots &  1 \end{array} \right).\label{Chat}
\end{equation}
To define $\hat{\Theta}$ we introduce $\hat{T}^2 = diag ( \hat{\tau}_1^2, \cdots, \hat{\tau}_p^2)$ which is a $p\times p$ diagonal matrix with 
\begin{align}
\hat{\tau}_j^2 = \| X_j - X_{-j} \hat{\gamma}_j \|_n^2 + \lambda_{node,n} \| \hat{\Gamma}_j \hat{\gamma}_j \|_1\label{tauhat},
\end{align}
for all $j=1,...,p$. We now define
\begin{align}
\hat{\Theta} = \hat{T}^{-2} \hat{C}\label{hattheta}.
\end{align}
\footnote{A practical benefit is that the nodewise regressions actually only have to be run for $j\in H$ and not all $j=1,...,p$ as we only need to estimate the covariance matrix of those parameters involved in the hypothesis being tested.} \footnote{Denote by $\hat{\Theta}_L$ the nodewise regression estimate of $\Theta$ based on the Lasso. This can be obtained by using $\hat{\gamma}_L$ from  (\ref{NodeLObj}) instead of  $\hat{\gamma}$ in (\ref{Chat})-(\ref{hattheta}).} It remains to be shown that this $\hat{\Theta}$ is close to being an inverse of $\hat{\Sigma}$. To this end, we define $\hat{\Theta}_j$ as the $j$'th row of $\hat{\Theta}$ but understood as a $p\times 1$ vector and analogously for $\hat{C}_j$. Thus, $\hat{\Theta}_j=\hat{C}_j/\hat{\tau}_j^2$. 
Denoting by $e_j$ the $j$'th $p\times 1$ unit vector, arguments detailed in appendix C show that
\begin{equation}
\|  \hat{\Theta}_j'\hat{\Sigma} - e_j' \|_{\infty} 
\leq
\frac{\lambda_{node,n}}{\hat{\tau}_j^2}.\label{2.8}
\end{equation}
Hence, the above display provides an upper bound on the maximal absolute entry of the $j$'th row of $\hat{\Theta}\hat{\Sigma}-I_p$ which, combined with the oracle inequality for $\enVert[0]{\hat{\beta}-\beta_0}_1$, will yield an upper bound on $\Delta_j$ in (\ref{stat}) by arguments made rigorous in the appendix.



Before stating Assumption 3 we introduce the following notation in connection to the asymptotic covariance matrix. Set  $\bar{s}=\max_{j \in H}s_j$, $\Sigma_{xu} = \lim_{n \to \infty} n^{-1} \sum_{i=1}^n E X_i X_i' u_i^2$ and $\hat{\Sigma}_{xu} = n^{-1} \sum_{i=1}^n X_i X_i' \hat{u}_i^2$, where $\hat{u}_i = Y_i - X_i' \hat{\beta}$. 
 
{\bf Assumption 3}.

Let $r\geq 6$ and
\begin{itemize}
\item[a)] $s_0\frac{h^{2/r+1/2}p^{4/r}}{n^{1/2}}\to 0$. 
\item[b)] $\frac{p^{8/r}h\bar{s}}{n^{1/2}}\to 0$.
\item[c)] $\frac{p^{2/r}\sqrt{s_0}h\bar{s}}{n^{1/2}}\to 0$, $\frac{p^{8/r}\sqrt{s_0}h\bar{s}}{n^{3/4}}\to 0$ and $\frac{p^{8/r}s_0h\bar{s}}{n^{(r-2)/r}}\to 0$.
\item[d)] $\frac{(h\bar{s})^{r/4+1}\wedge (h\bar{s})^{r/4}p}{n^{r/4-1}}\to 0$.
\item[e)] $\phi_{\min}(\Sigma_{xu})$ is bounded away from 0 and $\phi_{\max}\del[0]{\Sigma_{xu}}$ is uniformly bounded. $\phi_{\max}(\Sigma)$ is bounded from above. 
\end{itemize}
Assumptions 3a)-d) all restrict the rate at which the size of the model ($p$), the number of relevant variables ($s_0$) as well as the number of coefficients involved in the hypothesis being tested ($h$) are allowed to increase. However, part b) of Assumption 3 reveals that the number of $\beta_{0,j}$ involved in the hypothesis being tested must be of order $o(n^{1/2})$. Letting the number of parameters involved in the hypothesis increase with the sample size is a useful generalization of \cite{van2014} who only mention the possibility of $H$ possessing a fixed or growing number of elements. Part b) also reveals that if one encounters a situation where $p$ increases faster than the sample size, then one needs $r > 16$ for our theory. If one is willing to assume subgaussianity of the covariates and the error terms the powers of $p$ in Assumption 3 could be replaced by powers of $\log(p)$. Furthermore, in a different context, \cite{vicecta2012, belloni2014inference} have used moderate deviation inequalities for self-normalized sums to get results which are of the same order as if subgaussianity was imposed but only assuming certain moments to exist for the covariates and the error terms. In that case, as usual, $p$ can grow almost exponentially in $n$. Assumptions 3a)-d) are trivially satisfied in the classical setting where $p$, $h$, $s_0$ and $\bar{s}$ are fixed.  
Finally, Assumption 3e) restricts the eigenvalues of $\Sigma$ and $\Sigma_{xu}$.

\section{Inference}\label{Inference}

This section has two main results. The first result provides sufficient conditions for asymptotically gaussian inference to be valid for linear combinations of the entries of desparsified conservative Lasso $\hat{b}$. The second result shows that the resulting confidence bands are uniformly valid and contract at the optimal rate.

\begin{thm}\label{thm2}
Let Assumptions 1-3 \footnote{Assumption 2b) is of course implied by Assumption 3b) but to keep the statement clean we shall simply assume all of Assumption 2 to be valid.} be satisfied. Then,
\begin{align}
\frac{n^{1/2} \alpha'(\hat{b} - \beta_{0})}{\sqrt{\alpha'\hat{\Theta} \hat{\Sigma}_{xu} \hat{\Theta}'\alpha}}\stackrel{d}{\to}N(0,1),\label{asymdist}
\end{align}
where $\alpha$ is a $p\times 1$ vector with $\enVert[0]{\alpha}_2=1$. Furthermore,
\begin{align}
\sup_{\beta_0\in\mathcal{B}_{\ell_0}(s_0)}\envert[1]{\alpha'\hat{\Theta} \hat{\Sigma}_{xu} \hat{\Theta}'\alpha - \alpha'\Theta \Sigma_{xu} \Theta'\alpha}
=
o_p(1).\label{asymcov}
\end{align}
\end{thm}

Theorem \ref{thm2} provides sufficient conditions for asymptotically gaussian inference to be valid. We stress again that the number of $\beta_{0,j}$, $h$, involved in the statistic in (\ref{asymdist}) is allowed to increase as the sample size tends to infinity as long as this does not happen too fast. Furthermore, these results can be valid in the presence of more variables than observations ($p>n$).


We also emphasize that the above results allow the error terms to be heteroskedastic.  (\ref{asymcov}) provides a uniformly consistent estimator of the asymptotic variance of $n^{1/2} \alpha'(\hat{b} - \beta_{0})$.  The uniformity of  (\ref{asymcov}) will also be used in the proof of Theorem \ref{thm3} below. (\ref{asymcov}) is also interesting as it gives the limit of the variance in the denominator of (\ref{asymdist}) even as the dimension ($p\times p$) of the matrices involved in the expression increases. 

Note that while $d_{n1}$ and $d_{n2}$ do not directly enter in the first order asymptotic result of Theorem \ref{thm2}, equations (\ref{b.0}), (\ref{l2auxxb}), (\ref{numerator}) and (\ref{thm2aa}) in the appendix still reveal that the desparsified conservative Lasso is likely to result in more precise inference than the plain desparsified Lasso. The effect comes directly from more precise parameter estimates as well as through more precise covariance matrix estimation using the nodewise regressions and is clearly seen in the simulations.

In the case where $H$ is a set of fixed cardinality $h$, (\ref{asymdist}) implies that
\begin{align}
\enVert[2]{\del[1]{\hat{\Theta} \hat{\Sigma}_{xu} \hat{\Theta}'}_H^{-1/2}\sqrt{n}\del[0]{\hat{b}_H-\beta_{0,H}}}_2^2\stackrel{d}{\to}\chi^2(h),\label{chi2}
\end{align} 
as it is asymptotically a sum of $h$ independent standard normal random variables. Thus, asymptotically valid $\chi^2$-inference can be performed in order to test a hypothesis on $h$ parameters simultaneously. Wald tests of general  restrictions of the type $H_0: g(\beta_0)=0$ (where $g:\mathbb{R}^p\to\mathbb{R}^h$ is differentiable in an open neighborhood around $\beta_0$ and has derivative matrix of rank $h$) can now also be constructed in the usual manner, see e.g. \cite{davidson00} Chapter 12, even when $p>n$ which has hitherto been impossible.

Consider the leading special case where $H=\cbr[0]{j}$ such that $\alpha$ reduces to the $j$'th unit vector $e_j$ of $\mathbb{R}^p$ and $h=1$.  As a corollary to the previous theorem we consider testing a hypothesis about a single coefficient. The number of regressors is assumed to be a positive multiple of the sample size and the maximal number of non-zero entries in the $j$th row of the inverse population covariance matrix is bounded. This is satisfied when, eg, $\Sigma$ is a Toeplitz matrix. The important thing to notice is that all dimensionality assumptions from Assumptions 1-3 are automatically satisfied in the setting considered in Corollary \ref{cor1.0}.
\begin{corollary}\label{cor1.0}
Let Assumptions 1, 2a, 2c and 3e be satisfied with $p=an,\ a>0$, with $r >16$, $s_0 = O(n^{1/4})$, $\bar{s}=O(1)$. Then,
\begin{align}
\frac{n^{1/2} (\hat{b}_j - \beta_{j0})}{\sqrt{(\hat{\Theta} \hat{\Sigma}_{xu} \hat{\Theta}')_{jj}}}\stackrel{d}{\to}N(0,1),\label{asymdistc1}
\end{align}
Furthermore,
\begin{align}
\sup_{\beta_{0}\in\mathcal{B}_{\ell_0}(s_0)}\envert[1]{(\hat{\Theta} \hat{\Sigma}_{xu} \hat{\Theta}')_{jj} - (\Theta \Sigma_{xu} \Theta')_{jj}}
=
o_p(1).\label{asymcovc1}
\end{align}
\end{corollary}

 If, furthermore, the covariates and the error terms are independent and the latter are homoskedastic with variance $\sigma^2$ we get that
\begin{align*}
\alpha'\Theta \Sigma_{xu} \Theta'\alpha
=
e_j'\Sigma^{-1}\sigma^2\Sigma\Sigma^{-1}e_j
=
\sigma^2(\Sigma^{-1})_{j,j},
\end{align*}
which is nothing else than the standard formula for the asymptotic variance of the least squares estimator of the $j$'th coefficient $\hat{\beta}_{OLS,j}$ in a fixed dimensional linear regression model. Thus, there is no efficiency loss. Corollary \ref{cor1.0} is in the spirit of \cite{robinson1988root} who constructed a $\sqrt{n}$ consistent estimator of the coefficients pertaining to the linear part of a semiparametric model. See also \cite{van2014} Section 2.3.3 for more discussion and relations to the semiparametric framework. In the context of uniformly valid confidence bands for a single parameter the work of \cite{belloni2014inference} is also relevant. These authors consider inference on treatment effects using a post-double-selection procedure.

\subsection{Uniform Convergence}\label{unif}
The next theorem shows that the confidence bands based on the desparsified conservative Lasso are honest and that they contract at the optimal rate. Recall that $\mathcal{B}_{\ell_0}(s_0)=\cbr[1]{\enVert{\beta_{0}}_{\ell_0}\leq s_0}$.
\begin{thm}\label{thm3}
Let Assumptions 1-3 be satisfied and let $\alpha=\alpha_n\in\mathbb{R}^p$ denote any fixed sequence of vectors satisfying $\enVert[0]{\alpha}_2=1$. Then we have
\begin{align}
\sup_{t\in\mathbb{R}}\sup_{\beta_0\in\mathcal{B}_{\ell_0}(s_0)}\envert[4]{P\del[4]{\frac{n^{1/2} \alpha'(\hat{b} - \beta_{0})}{\sqrt{\alpha'\hat{\Theta} \hat{\Sigma}_{xu} \hat{\Theta}'\alpha}}\leq t}-\Phi(t)}\to 0.\label{t3p1}
\end{align}
Furthermore, letting $\hat{\sigma}_j=\sqrt{e_j'\hat{\Theta} \hat{\Sigma}_{xu} \hat{\Theta}'e_j}$ (corresponding to $\alpha=e_j$ in (\ref{t3p1})) and $z_{1-\delta/2}$ the $1-\delta/2$ percentile of the standard normal distribution, one has for all $j=1,...,p$
\begin{align}
\lim_{n\to\infty}\inf_{\beta_0\in\mathcal{B}_{\ell_0}(s_0)}P\del[3]{\beta_{0,j}\in\sbr[2]{\hat{b}_j-z_{1-\delta/2}\frac{\hat{\sigma}_j}{\sqrt{n}},\hat{b}_j+z_{1-\delta/2}\frac{\hat{\sigma}_j}{\sqrt{n}}}}= 1-\delta. \label{t3p2}
\end{align}
Finally, letting $\diam([a,b])=b-a$ be the length of an interval $[a,b]$ in the real line, we have that
\begin{align}
\sup_{\beta_0\in\mathcal{B}_{\ell_0}(s_0)}\diam\del[3]{\sbr[2]{\hat{b}_j-z_{1-\delta/2}\frac{\hat{\sigma}_j}{\sqrt{n}},\hat{b}_j+z_{1-\delta/2}\frac{\hat{\sigma}_j}{\sqrt{n}}}}
=
O_p\del[3]{\frac{1}{\sqrt{n}}}.\label{t3p3}
\end{align}
\end{thm}
(\ref{t3p1}) reveals that convergence to the standard normal distribution is actually valid uniformly over the $\ell_0$-ball of radius at most $s_0$. We stress, however, that (\ref{t3p1}) ceases to be valid if one also takes the supremum over all $\alpha$ sequences satisfying the assumptions of the theorem. Thus, the asymptotics are uniform over $\beta_0$ but pointwise in $\alpha$. 
(\ref{t3p2}) is a consequence of (\ref{t3p1}) and entails that the confidence band $\sbr[1]{\hat{b}_j-z_{1-\delta/2}\frac{\hat{\sigma}_j}{\sqrt{n}},\hat{b}_j+z_{1-\delta/2}\frac{\hat{\sigma}_j}{\sqrt{n}}}$ is \textit{asymptotically honest} for $\beta_{0,j}$ over $\mathcal{B}(s_0)$ in the sense of \cite{li1989honest}. 

(\ref{t3p3}) is important as it reveals that the confidence band $\sbr[1]{\hat{b}_j-z_{1-\delta/2}\frac{\hat{\sigma}_j}{\sqrt{n}},\hat{b}_j+z_{1-\delta/2}\frac{\hat{\sigma}_j}{\sqrt{n}}}$ has the optimal rate of contraction $1/\sqrt{n}$. Furthermore, these confidence bands are uniformly narrow over $\mathcal{B}_{\ell_0}(s_0)$ such that for all $\epsilon>0$ there exists an $M>0$, not depending on $\beta_0$, with the property that
 \[ \diam\del[3]{\sbr[2]{\hat{b}_j-z_{1-\delta/2}\frac{\hat{\sigma}_j}{\sqrt{n}},\hat{b}_j+z_{1-\delta/2}\frac{\hat{\sigma}_j}{\sqrt{n}}}}\leq M/\sqrt{n},\]
 for all $\beta_0\in\mathcal{B}_{\ell_0}(s_0)$ with probability at least $1-\epsilon$. Here it is vital that at the same time the confidence intervals are asymptotically honest. Since the desparsified conservative Lasso is not a sparse estimator, (\ref{t3p3}) does not contradict inequality 6 in Theorem 2 of \cite{potscher2009confidence} who shows that honest confidence bands based on sparse estimators must be large. 


Finally, the above results are valid without any sort of $\beta_{\min}$-condition which requires the absolute value of the smallest non-zero coefficient to be greater than $s_0\lambda_n$. In total, Theorem \ref{thm3} reveals that the inference of our procedure is very robust as the confidence bands are honest and contract \textit{uniformly} at the optimal rate.

We provide a brief overview of the proofs here. Lemmata \ref{A1}-\ref{conc} in the appendix are crucial ingredients of our main theorems. Lemma \ref{A1} provides an oracle inequality for general weighted Lasso type estimators subject to a condition on the smallest weight of the truly zero coefficients. Lemmata \ref{MZ} and \ref{conc} are very important to get maximal inequalities for certain sums that determine the order of $\lambda_n$ in our setting of regressors and error terms with bounded $r$th moments. Our use of the Marcinkiewicz-Zygmund inequality provides sharper bounds than Nemirowski's inequality. The technical details are given in the remarks after Lemma \ref{conc}. Thus, Lemma \ref{conc} is a novel contribution in high dimensional statistics. Lemma \ref{lemma9} provides an $\ell_{\infty}$ bound for the estimation error of the Lasso in the case of error terms and regressors with bounded $r$th moments and heteroskedasticity. Furthermore, Lemma \ref{lemma9} provides a theoretical choice of $\lambda_{prec}$. Theorem \ref{thm2} is key in getting a heteroskedasticity consistent estimate of the variance of linear combinations of the parameters involved in the hypothesis being tested and is new in the literature. At the end of Appendix B we also establish strong oracle optimality for the variant of the conservative Lasso and a way to choose $\lambda_n$ for consistent variable selection for this estimator.

\section{Monte Carlo}\label{MC}
In this section we investigate the finite sample performance of the (desparsified) conservative Lasso and compare it to the one of the (desparsified) Lasso of \cite{van2014}. We also implement the procedure of \cite{javanmard2013confidence} using the authors' code\footnote{Available at \url{https://web.stanford.edu/~montanar/sslasso/code.html}.}.  The Lasso as well as the conservative Lasso are implemented in R by means of the publicly available \texttt{glmnet} package.

To choose $\lambda_n$ as well as $\lambda_{node,n}$, we follow \cite{fantang13} and use their Generalized Information Criterion (GIC). In the regression equation (\ref{0}) 
\[ \lambda_n^* = \argmin_{\lambda_n \in \{ \lambda_l,..., \lambda_u\}} GIC (\lambda_n),\]
where 
\[ GIC (\lambda_n) = \log (\|Y - X \hat{\beta}_{\lambda_n}\|_n^2) + \frac{\log\log(n) \log(p) |S_{\lambda_n}|}{n},\]
and $\lambda_l, \lambda_u$ are lower and upper bounds for $\lambda_n$ while $\hat{\beta}_{\lambda_n}$ is the conservative Lasso estimate pertaining to $\lambda_n$. Finally, $|S_{\lambda_n}|$ denotes the number of non-zero entries in $\hat{\beta}_{\lambda_n}$. The same procedure is used to choose $\lambda_n$ for the variant of the conservative Lasso as well as in the nodewise regressions to choose $\lambda_{node,n}$. At the end of Appendix B we provide a theorem stating that for the variant of the conservative Lasso choosing the tuning parameters by GIC leads to consistent model selection. 

We compare GIC to choosing the tuning parameters by BIC, see e.g. (9.4.9) in \cite{davidson00}. Of course one could also use cross validation to choose $\lambda_n$ but in our experience this does not improve the quality of the results while being considerably slower. All data will be generated from the model (\ref{0}).

As argued in Section \ref{CLsub} a good choice of $\lambda_{prec}$ should be a high probability bound on $\enVert[0]{\hat{\beta}_L-\beta_0}_{\ell_\infty}$. Lemma \ref{lemma9} in the appendix shows that $\lambda_{prec} = \frac{9 \lambda_n}{4} \|\Theta\|_{\ell_{\infty}}$ is exactly such a bound. Next, Lemma \ref{Node} in the appendix justifies using the plug-in estimate $\enVert[0]{\hat{\Theta}_L}_{\ell_\infty}$ for $\|\Theta\|_{\ell_{\infty}}$ in the choice of $\lambda_{prec}$. However, we find that in practice one might as well use $\lambda_{prec}= \frac{9 \lambda_n}{4}$ which corresponds to $\Theta=I_p$. This choice also has the additional computational advantage of avoiding running all $p$ nodewise regressions. Furthermore, it is the fallback option used in \cite{javanmard2013confidence} in case any of their optimizations needed to get $\hat{\Theta}$ fails. Thus, we shall use  $\lambda_{prec}= \frac{9 \lambda_n}{4}$ which, however, does not come with theoretical performance guarantees.

The following algorithm summarizes how to implement the desparsified conservative Lasso and how to conduct inference with it.

{\bf Algorithm to implement the desparsified conservative Lasso}
\begin{enumerate}
  \item For each $\lambda_n\in \cbr[0]{\lambda_l,...,\lambda_u}$ implement the Lasso $\hat{\beta}_{L}$ by imposing $\hat{w}_j =1$ in (\ref{ConsLassoObj}). $\cbr[0]{\lambda_l,...,\lambda_u}$ is constructed by the \texttt{glmnet} package in R to ensure that models of many sizes are implemented. Use either BIC or GIC to select $\lambda_n \in \cbr[0]{\lambda_l,...,\lambda_u}$.
    \item Construct $\hat{w}_j = \frac{\lambda_{prec}}{|\hat{\beta}_{L,j}| \vee \lambda_{prec}}\ j=1,...,p$ with $\lambda_{prec}=\frac{9}{4}\lambda_n$ and implement the conservative Lasso $\hat{\beta}$ as in (\ref{ConsLassoObj}). Use either BIC or GIC to select $\lambda_n \in \cbr[0]{\lambda_l,...,\lambda_u}$.  
  
\item For each $j\in H$ construct the $j$th element of the desparsified conservative Lasso by the following steps.
  
a) Run the nodewise Lasso in (\ref{NodeLObj}) with $\lambda_{node}=\lambda_n$ to get $\hat{\gamma}_{L,j}$.

b) Construct the weights for the nodewise conservative Lasso:  
$ \hat{\Gamma}_j = diag \left( \frac{\lambda_{prec}}{|\hat{\gamma}_{L,l}| \vee \lambda_{prec}},\ l=1,...,p, l\neq j  \right)$.  

c) Run the nodewise conservative Lasso as in (\ref{NodeCLObj}) using $\hat{\Gamma}_j$ from step 3b above.

d) Construct $\hat{C}_j$, the $j$th row of $\hat{C}$, as in as in (\ref{Chat}) and obtain $\hat{\tau}_j^2$ as in (\ref{tauhat}). 

e) Let $\hat{\Theta}_j=\hat{C}_j/\hat{\tau}_j^2$ be the $j$th row of $\hat{\Theta}$.


f) Construct the $j$th element of the desparsified conservative Lasso (\ref{despClasso}) which is  $\hat{b}_j = \hat{\beta}_j + \hat{\Theta}_j X' (Y- X \hat{\beta})/n$.
\item $\chi^2$-tests are constructed as in (\ref{chi2}) while the confidence bands are constructed as in (\ref{t3p2}).
\end{enumerate}

The variant of the conservative Lasso goes through steps 1-4 using $\tilde{w}_j$ instead of $\hat{w}_j$ and $\tilde{\Gamma}_j$ instead of $\hat{\Gamma}_j$.


All simulations are carried out with 1,000 replications unless stated otherwise and we consider the following performance measures for each of the procedures:
\begin{enumerate}
\item Estimation error: We compute the $\ell_2$-estimation error of the Lasso and the conservative Lasso and its variant averaged over the Monte Carlo replications.
\item Size: We evaluate the size of the $\chi^2$-test in (\ref{chi2}) for a hypothesis involving more than one parameter.
\item Power: We evaluate the power of the $\chi^2$-test in (\ref{chi2}) for a hypothesis involving more than one parameter.
\item Coverage rate: We calculate the coverage rate of a gaussian confidence interval constructed as in (\ref{t3p2}). This is done for a non-zero as well as a zero parameter.
\item Length of confidence interval: We calculate the length of the two confidence intervals considered in point 4, above.
\end{enumerate}

In the simulations we investigate the performance of the conservative Lasso in moderate, high, and very high-dimensional settings. The covariance matrices of the covariates are always chosen to have a Toeplitz structure with $(i,j)$'th entry equal to $\rho^{|i-j|}$ for some $0\leq \rho<1$ to be made precise below. The covariates and the error terms are assumed to be $t$-distributed with 10 degrees of freedom. At this point we remark that all experiments reported below were also carried out with the covariates possessing a block diagonal covariance matrix and/or gaussian error terms (all combinations were tried). This only affected the findings in the simulations marginally and we shall not report these results here.

All tests are carried out at a 5\% significance level and all confidence intervals are at the 95\% level. Unless mentioned otherwise, the $\chi^2$-tests involve the two first parameters in $\beta_0$ of which we deliberately make sure that the first one is 1 and the second one is zero. Thus, $h=2$ in our Experiments 1-3. For measuring the size of the $\chi^2$-test, we test the true hypothesis $H_0:(\beta_{0,1},\beta_{0,2})=(1,0)$. For measuring the power of the $\chi^2$-test, we test the false hypothesis $H_0:(\beta_{0,1},\beta_{0,2})=(1,0.4)$. Thus, the hypothesis is only false on the second entry of $\beta_0$. Similarly, we construct confidence intervals for the first two parameters of $\beta_0$ such that the coverage rate can be compared between non-zero and zero parameters.

As our theory allows for heteroskedastic error terms we also investigate the effect of this. To be precise, we consider error terms of the form $u_i=\epsilon_i\del[1]{\frac{1}{\sqrt{2}}X_{1,i}+b_x X_{2,i}}$ where $\epsilon_i\sim t(10)$ is independent of the covariates and $b_x$ is chosen such that the unconditional variance of $u_i$ is still that of a $t$-distribution with 10 degrees of freedom\footnote{To ensure that $u_i$ still has the variance of $\epsilon_i\sim t(10)$ a small calculation shows that it suffices to choose $b_x=\frac{-\sqrt{2}\rho+\sqrt{2\rho^2+2}}{2}$. Thus, the higher the correlation between $X_{1,i}$ and $X_{2,i}$, the smaller $b_x$ should be chosen.}. Note that this $u_i$ satisfies our assumption $E(u_i|X_i)=0$ and has variance conditional on $X_i$ given by $E(\epsilon_i^2)\del[1]{\frac{1}{\sqrt{2}}X_{1,i}+b_x X_{2,i}}^2$. The reason we ensure that the unconditional variance of $u_i$ is still that of a $t(10)$-distribution is that we do not want any findings to be driven by a plain change in the unconditional variance. It is also deliberate that we choose the conditional heteroskedasticity to depend on $X_{1,i}$ and $X_{2,i}$ as these are the variables involved in the $\chi^2$-tests and the confidence intervals. 


\begin{itemize}
\item Experiment 1a (moderate-dimensional setting). $\beta_0$ is $50\times 1$ with 10 ones and 40 zeros. The 10 ones are equidistant in the parameter vector. Thus, $p=50$ and $s_0=10$. We consider $\rho=0,0.5$ and $0.9$ and $n=100$.
\item Experiment 1b (moderate-dimensional setting). As Experiment 1a but with heteroskedastic errors.
\item Experiment 2a (high-dimensional setting). $\beta_0$ is $104\times 1$ with the first four entries being $(1,0,1,0.1)$ and the remaining 100 entries being zero. Thus, $p=104$ and $s_0=3$. We consider $\rho=0,0.5$ and $0.9$ and $n=100$.
\item Experiment 2b (high-dimensional setting). As Experiment 2a but with heteroskedastic errors. 
\item Experiment 3a (very high-dimensional setting). $\beta_0$ is $1000\times 1$ with 10 ones and 990 zeros. The 10 ones are equidistant in the parameter vector. Thus, $p=1000$ and $s_0=10$. $\rho=0.75$. This experiment is carried out for $n=100,150,200,500$ to gauge the effect of an increasing sample size. We also experimented with different values of $\rho$ but this did not qualitatively alter our findings. The number of replications is 100 as the procedure of \cite{javanmard2013confidence} is rather time consuming in high dimensions.
\item Experiment 3b (very high-dimensional setting). As Experiment 3a but with heteroskedastic errors. 
\item Experiment 4: As Experiment 2a with $\rho=0.5$ but testing a hypothesis involving the first ten parameters to investigate the properties of the proposed procedures when many parameters are involved in the hypothesis being tested. When gauging power, the only deviation from the true parameter vector is that the second entry of $\beta_0$ is hypothesized to be 0.4 (as in all other power calculations).
\end{itemize}  

\begin{table}[h]
\centering
\begin{tabular}{ccccccccc}
\toprule
& & & \multicolumn{2}{c}{ $\chi^2$}&\multicolumn{2}{c}{Coverage rate}& \multicolumn{2}{c}{Length}\\
\cmidrule(lr){ 4 - 5 }\cmidrule(lr){ 6 - 7 }\cmidrule(lr){ 8 - 9 }
$n=100$& & $\ell_2$& Size & Power & non-zero & zero & non-zero & zero \\ 
  \midrule
{\multirow{7}{*}{$\rho=0$}}
&Lasso & 0.668 & 0.136 & 0.949 & 0.852 & 0.929 & 0.386 & 0.383 \\ 
  &LassoGIC & 0.721 & 0.116 & 0.936 & 0.873 & 0.940 & 0.411 & 0.404 \\ 
  &CLasso & 0.516 & 0.097 & 0.953 & 0.888 & 0.952 & 0.381 & 0.383 \\ 
  &CLassoGIC & 0.589 & 0.100 & 0.944 & 0.889 & 0.955 & 0.396 & 0.396 \\ 
  &CLassoInd & 0.361 & 0.083 & 0.964 & 0.906 & 0.950 & 0.364 & 0.371 \\ 
  &CLassoIndGIC & 0.371 & 0.077 & 0.962 & 0.913 & 0.951 & 0.369 & 0.375 \\ 
  &J\&M & 0.824 & 0.007 & 0.383 & 0.989 & 0.990 & 0.787 & 0.776 \\ 
   \midrule
{\multirow{7}{*}{$\rho=0.5$}}
&Lasso & 0.709 & 0.146 & 0.900 & 0.852 & 0.918 & 0.394 & 0.409 \\ 
  &LassoGIC & 0.741 & 0.138 & 0.860 & 0.867 & 0.921 & 0.411 & 0.422 \\ 
  &CLasso & 0.491 & 0.093 & 0.917 & 0.888 & 0.954 & 0.397 & 0.417 \\ 
  &CLassoGIC & 0.540 & 0.092 & 0.892 & 0.889 & 0.956 & 0.405 & 0.423 \\ 
  &CLassoInd & 0.392 & 0.086 & 0.941 & 0.897 & 0.953 & 0.387 & 0.415 \\ 
  &CLassoIndGIC & 0.378 & 0.083 & 0.945 & 0.907 & 0.958 & 0.388 & 0.413 \\ 
  &J\&M & 0.867 & 0.012 & 0.300 & 0.993 & 0.991 & 0.896 & 0.992 \\
\midrule
{\multirow{7}{*}{$\rho=0.9$}}
&Lasso & 1.392 & 0.201 & 0.630 & 0.820 & 0.854 & 0.617 & 0.738 \\ 
  &LassoGIC & 1.392 & 0.199 & 0.634 & 0.815 & 0.855 & 0.608 & 0.722 \\ 
  &CLasso & 1.214 & 0.137 & 0.529 & 0.885 & 0.922 & 0.772 & 0.961 \\ 
  &CLassoGIC & 1.224 & 0.132 & 0.524 & 0.887 & 0.927 & 0.769 & 0.947 \\ 
  &CLassoInd & 1.395 & 0.136 & 0.483 & 0.881 & 0.912 & 0.828 & 1.121 \\ 
  &CLassoIndGIC & 1.362 & 0.130 & 0.478 & 0.882 & 0.921 & 0.838 & 1.134 \\ 
  &J\&M & 1.532 & 0.025 & 0.126 & 0.978 & 0.978 & 1.561 & 2.093 \\ 
\bottomrule
\end{tabular}
\caption{\small Summary statistics for Experiment 1a. $\ell_2$: average $\ell_2$-estimation error, $\chi^2$: Size and Power report the size and power of the hypotheses $H_0: (\beta_{0,1}, \beta_{0,2})=(1,0)$ and $H_0: (\beta_{0,1}, \beta_{0,2})=(1,0.4)$, respectively. Coverage rate: the actual coverage rate of the asymptotically gaussian 95\% confidence interval for $\beta_{0,1}$ and $\beta_{0,2}$. Length: the length of the two confidence intervals mentioned above.  Lasso: Lasso with BIC. LassoGIC: Lasso with GIC.
CLasso: Conservative Lasso with BIC.  CLassoGIC: Conservative Lasso with GIC. CLassoInd: Variant of Conservative Lasso with BIC. CLassoIndGIC: Variant of Conservative lasso with GIC. J\&M: Procedure of \cite{javanmard2013confidence}.}
\label{E1a}
\end{table}

\subsection{Results}
Most often, using BIC or GIC to choose $\lambda_n$ is not overly important for our performance measures. However, BIC tends to perform better when $p$ is large compared to $n$ and, unless mentioned otherwise, we will focus on the results for BIC in the sequel. We also note that a general finding is that the conservative Lasso performs better than its variant when $p$ is small compared to $n$ while this ordering reverses when $p$ is large compared to $n$. 


Table \ref{E1a} contains the results for Experiment 1a. First, as predicted in Section \ref{CLsub}, both versions of the conservative Lasso have a lower estimation error than the plain Lasso due to more intelligent weights. The variant of the conservative Lasso fares particularly well for $\rho=0$ and $\rho=0.5$. Furthermore, the conservative Lasso is always less size distorted than the Lasso while having slightly more power except for when $\rho=0.9$. The procedure of \cite{javanmard2013confidence} has even less size distortion but the price is very low power. When $\rho=0.9$ all procedures have serious power deficiencies. Next, our procedure (both versions) always has a coverage rate which is closer to the nominal rate of 95\% than the plain desparsified Lasso. Note, however, that all Lasso-based procedures still have a slight tendency towards undercoverage (a phenomenon which disappears as the sample size is increased (not reported here)). This is the case in particular for the plain Lasso and less pronounced for the conservative Lasso. The reasons for this are that the confidence intervals produced by the Lasso are too narrow compared to the more accurate ones produced by the conservative Lasso and that the latter produces more precise parameter estimates. The confidence intervals of \cite{javanmard2013confidence} have good coverage but are very wide.

\begin{table}[h]
\centering
\begin{tabular}{ccccccccc}
\toprule
& & & \multicolumn{2}{c}{ $\chi^2$}&\multicolumn{2}{c}{Coverage rate}& \multicolumn{2}{c}{Length}\\
\cmidrule(lr){ 4 - 5 }\cmidrule(lr){ 6 - 7 }\cmidrule(lr){ 8 - 9 }
$n=100$& & $\ell_2$& Size & Power & non-zero & zero & non-zero & zero \\ 
  \midrule
{\multirow{7}{*}{$\rho=0$}}
&Lasso & 0.738 & 0.158 & 0.765 & 0.854 & 0.898 & 0.557 & 0.563 \\ 
  &LassoGIC & 0.790 & 0.143 & 0.735 & 0.869 & 0.914 & 0.582 & 0.588 \\ 
  &CLasso & 0.610 & 0.132 & 0.755 & 0.875 & 0.933 & 0.567 & 0.581 \\ 
  &CLassoGIC & 0.685 & 0.130 & 0.734 & 0.875 & 0.932 & 0.578 & 0.591 \\ 
  &CLassoInd & 0.450 & 0.120 & 0.776 & 0.890 & 0.938 & 0.562 & 0.579 \\ 
  &CLassoIndGIC & 0.494 & 0.113 & 0.759 & 0.887 & 0.942 & 0.567 & 0.584 \\ 
  &J\&M & 0.904 & 0.012 & 0.289 & 0.984 & 0.981 & 1.000 & 1.002 \\ 
   \midrule
{\multirow{7}{*}{$\rho=0.5$}}
&Lasso & 0.780 & 0.193 & 0.774 & 0.828 & 0.913 & 0.609 & 0.534 \\ 
  &LassoGIC & 0.815 & 0.183 & 0.737 & 0.835 & 0.925 & 0.630 & 0.554 \\ 
  &CLasso & 0.593 & 0.148 & 0.778 & 0.860 & 0.960 & 0.631 & 0.553 \\ 
  &CLassoGIC & 0.656 & 0.143 & 0.769 & 0.860 & 0.960 & 0.642 & 0.564 \\ 
  &CLassoInd & 0.477 & 0.134 & 0.821 & 0.864 & 0.962 & 0.629 & 0.551 \\ 
  &CLassoIndGIC & 0.485 & 0.130 & 0.813 & 0.868 & 0.968 & 0.638 & 0.557 \\ 
  &J\&M & 0.952 & 0.013 & 0.258 & 0.978 & 0.985 & 1.130 & 1.138 \\ 
\midrule
{\multirow{7}{*}{$\rho=0.9$}}
&Lasso & 1.484 & 0.218 & 0.524 & 0.792 & 0.867 & 0.789 & 0.835 \\ 
  &LassoGIC & 1.482 & 0.225 & 0.523 & 0.790 & 0.870 & 0.784 & 0.823 \\ 
  &CLasso & 1.364 & 0.151 & 0.457 & 0.847 & 0.928 & 0.928 & 1.051 \\ 
  &CLassoGIC & 1.384 & 0.148 & 0.453 & 0.849 & 0.926 & 0.928 & 1.041 \\ 
  &CLassoInd & 1.511 & 0.158 & 0.432 & 0.855 & 0.925 & 0.973 & 1.212 \\ 
  &CLassoIndGIC & 1.483 & 0.151 & 0.428 & 0.860 & 0.932 & 0.987 & 1.228 \\ 
  &J\&M & 1.634 & 0.035 & 0.132 & 0.963 & 0.975 & 1.807 & 2.323 \\ 
\bottomrule
\end{tabular}
\caption{\small Summary statistics for Experiment 1b. $\ell_2$: average $\ell_2$-estimation error, $\chi^2$: Size and Power report the size and power of the hypotheses $H_0: (\beta_{0,1}, \beta_{0,2})=(1,0)$ and $H_0: (\beta_{0,1}, \beta_{0,2})=(1,0.4)$, respectively. Coverage rate: the actual coverage rate of the asymptotically gaussian 95\% confidence interval for $\beta_{0,1}$ and $\beta_{0,2}$. Length: the length of the two confidence intervals mentioned above. Lasso: Lasso with BIC. LassoGIC: Lasso with GIC.
CLasso: Conservative Lasso with BIC.  CLassoGIC: Conservative Lasso with GIC. CLassoInd: Variant of Conservative Lasso with BIC. CLassoIndGIC: Variant of Conservative lasso with GIC. J\&M: Procedure of \cite{javanmard2013confidence}.}
\label{E1b}
\end{table}

Next, Table \ref{E1b} adds heteroskedasticity to the results of Experiment 1a. The main message of this table is that qualitatively the results of Experiment 1a remain unchanged as all procedures only suffer slightly from the introduction of heteroskedasticity in the error terms.

\begin{table}[h]
\centering
\begin{tabular}{ccccccccc}
\toprule
& & & \multicolumn{2}{c}{ $\chi^2$}&\multicolumn{2}{c}{Coverage rate}& \multicolumn{2}{c}{Length}\\
\cmidrule(lr){ 4 - 5 }\cmidrule(lr){ 6 - 7 }\cmidrule(lr){ 8 - 9 }
$n=100$& & $\ell_2$& Size & Power & non-zero & zero & non-zero & zero \\ 
  \midrule
{\multirow{7}{*}{$\rho=0$}}
&Lasso & 0.398 & 0.058 & 0.901 & 0.946 & 0.931 & 0.435 & 0.412 \\ 
  &LassoGIC & 0.425 & 0.051 & 0.902 & 0.959 & 0.933 & 0.444 & 0.414 \\ 
  &CLasso & 0.375 & 0.061 & 0.905 & 0.949 & 0.930 & 0.428 & 0.408 \\ 
  &CLassoGIC & 0.413 & 0.057 & 0.901 & 0.958 & 0.934 & 0.439 & 0.413 \\ 
  &CLassoInd & 0.315 & 0.076 & 0.906 & 0.929 & 0.925 & 0.486 & 0.467 \\ 
  &CLassoIndGIC & 0.368 & 0.070 & 0.911 & 0.941 & 0.934 & 0.422 & 0.406 \\ 
  &J\&M & 0.348 & 0.135 & 0.973 & 0.862 & 0.955 & 0.373 & 0.360 \\  
   \midrule
{\multirow{7}{*}{$\rho=0.5$}}
&Lasso & 0.337 & 0.162 & 0.687 & 0.928 & 0.823 & 0.439 & 0.436 \\ 
  &LassoGIC & 0.354 & 0.189 & 0.613 & 0.937 & 0.790 & 0.451 & 0.442 \\ 
  &CLasso & 0.315 & 0.142 & 0.720 & 0.924 & 0.846 & 0.435 & 0.437 \\ 
  &CLassoGIC & 0.343 & 0.173 & 0.650 & 0.930 & 0.813 & 0.448 & 0.441 \\ 
  &CLassoInd & 0.282 & 0.096 & 0.849 & 0.911 & 0.916 & 0.419 & 0.431 \\ 
  &CLassoIndGIC & 0.334 & 0.131 & 0.774 & 0.919 & 0.881 & 0.432 & 0.434 \\ 
  &J\&M & 0.310 & 0.429 & 0.919 & 0.787 & 0.767 & 0.316 & 0.301 \\ 
\midrule
{\multirow{4}{*}{$\rho=0.9$}}
&Lasso & 0.451 & 0.237 & 0.407 & 0.841 & 0.796 & 0.642 & 0.748 \\ 
  &LassoGIC & 0.456 & 0.275 & 0.381 & 0.844 & 0.768 & 0.637 & 0.728 \\ 
  &CLasso & 0.513 & 0.163 & 0.458 & 0.878 & 0.900 & 0.784 & 0.942 \\ 
  &CLassoGIC & 0.527 & 0.175 & 0.428 & 0.873 & 0.895 & 0.779 & 0.915 \\ 
  &CLassoInd & 0.556 & 0.076 & 0.386 & 0.926 & 0.935 & 0.916 & 1.228 \\ 
  &CLassoIndGIC & 0.647 & 0.071 & 0.359 & 0.932 & 0.934 & 0.944 & 1.251 \\ 
  &J\&M & 0.440 & 0.652 & 0.908 & 0.491 & 0.597 & 0.292 & 0.302 \\
\bottomrule
\end{tabular}
\caption{\small Summary statistics for Experiment 2a. $\ell_2$: average $\ell_2$-estimation error, $\chi^2$: Size and Power report the size and power of the hypotheses $H_0: (\beta_{0,1}, \beta_{0,2})=(1,0)$ and $H_0: (\beta_{0,1}, \beta_{0,2})=(1,0.4)$, respectively. Coverage rate: the actual coverage rate of the asymptotically gaussian 95\% confidence interval for $\beta_{0,1}$ and $\beta_{0,2}$. Length: the length of the two confidence intervals mentioned above. Lasso: Lasso with BIC. LassoGIC: Lasso with GIC.
CLasso: Conservative Lasso with BIC.  CLassoGIC: Conservative Lasso with GIC. CLassoInd: Variant of Conservative Lasso with BIC. CLassoIndGIC: Variant of Conservative lasso with GIC.J\&M: Procedure of \cite{javanmard2013confidence}.}
\label{E2a}
\end{table}


Table \ref{E2a} contains the results for Experiment 2a) in which the number of variables is slightly larger than the sample size. For $\rho=0.5$ both versions of the conservative Lasso are more precise than the Lasso, have less size distortion and higher power. This is the case in particular for the variant of the conservative Lasso with indicator function weights. The coverage probability for the zero parameter is also higher. The procedure of \cite{javanmard2013confidence} is rather size distorted. When $\rho=0.9$ the power of the $\chi^2$-test decreases for all Lasso based procedures. The procedure of \cite{javanmard2013confidence} suffers from severe size distortion. The conservative Lasso has a much better coverage rate, sometimes being more than ten percentage points larger for the zero parameter than the competitors. This comes from more precise parameter estimates and wider bands. 

\begin{table}[h]
\centering
\begin{tabular}{ccccccccc}
\toprule
& & & \multicolumn{2}{c}{ $\chi^2$}&\multicolumn{2}{c}{Coverage rate}& \multicolumn{2}{c}{Length}\\
\cmidrule(lr){ 4 - 5 }\cmidrule(lr){ 6 - 7 }\cmidrule(lr){ 8 - 9 }
$n=100$& & $\ell_2$& Size & Power & non-zero & zero & non-zero & zero \\ 
  \midrule
{\multirow{7}{*}{$\rho=0$}}
&Lasso & 0.445 & 0.082 & 0.714 & 0.923 & 0.945 & 0.631 & 0.634 \\ 
  &LassoGIC & 0.472 & 0.070 & 0.701 & 0.932 & 0.950 & 0.642 & 0.641 \\ 
  &CLasso & 0.430 & 0.088 & 0.715 & 0.920 & 0.950 & 0.626 & 0.634 \\ 
  &CLassoGIC & 0.465 & 0.075 & 0.704 & 0.929 & 0.950 & 0.639 & 0.642 \\ 
  &CLassoInd & 0.396 & 0.085 & 0.713 & 0.914 & 0.946 & 0.696 & 0.702 \\ 
  &CLassoIndGIC & 0.445 & 0.083 & 0.712 & 0.917 & 0.950 & 0.624 & 0.639 \\ 
  &J\&M & 0.395 & 0.136 & 0.771 & 0.848 & 0.954 & 0.567 & 0.573 \\ 
   \midrule
{\multirow{7}{*}{$\rho=0.5$}}
&Lasso & 0.391 & 0.184 & 0.545 & 0.918 & 0.875 & 0.698 & 0.587 \\ 
  &LassoGIC & 0.406 & 0.202 & 0.501 & 0.922 & 0.861 & 0.715 & 0.599 \\ 
  &CLasso & 0.381 & 0.167 & 0.587 & 0.906 & 0.898 & 0.695 & 0.589 \\ 
  &CLassoGIC & 0.403 & 0.186 & 0.528 & 0.912 & 0.877 & 0.711 & 0.600 \\ 
  &CLassoInd & 0.392 & 0.150 & 0.658 & 0.888 & 0.940 & 0.681 & 0.588 \\ 
  &CLassoIndGIC & 0.425 & 0.170 & 0.607 & 0.887 & 0.927 & 0.696 & 0.596 \\ 
  &J\&M & 0.370 & 0.504 & 0.787 & 0.804 & 0.840 & 0.565 & 0.480 \\ 
\midrule
{\multirow{7}{*}{$\rho=0.9$}}
&Lasso & 0.512 & 0.220 & 0.315 & 0.879 & 0.804 & 0.870 & 0.862 \\ 
  &LassoGIC & 0.514 & 0.245 & 0.301 & 0.877 & 0.777 & 0.869 & 0.846 \\ 
  &CLasso & 0.586 & 0.143 & 0.343 & 0.885 & 0.914 & 0.979 & 1.034 \\ 
  &CLassoGIC & 0.597 & 0.148 & 0.317 & 0.882 & 0.896 & 0.978 & 1.011 \\ 
  &CLassoInd & 0.698 & 0.083 & 0.316 & 0.934 & 0.953 & 1.104 & 1.324 \\ 
  &CLassoIndGIC & 0.765 & 0.081 & 0.304 & 0.936 & 0.957 & 1.132 & 1.349 \\ 
  &J\&M & 0.500 & 0.674 & 0.824 & 0.633 & 0.636 & 0.527 & 0.483 \\ 
\bottomrule
\end{tabular}
\caption{\small Summary statistics for Experiment 2b. $\ell_2$: average $\ell_2$-estimation error, $\chi^2$: Size and Power report the size and power of the hypotheses $H_0: (\beta_{0,1}, \beta_{0,2})=(1,0)$ and $H_0: (\beta_{0,1}, \beta_{0,2})=(1,0.4)$, respectively. Coverage rate: the actual coverage rate of the asymptotically gaussian 95\% confidence interval for $\beta_{0,1}$ and $\beta_{0,2}$. Length: the length of the two confidence intervals mentioned above.  Lasso: Lasso with BIC. LassoGIC: Lasso with GIC.
CLasso: Conservative Lasso with BIC.  CLassoGIC: Conservative Lasso with GIC. CLassoInd: Variant of Conservative Lasso with BIC. CLassoIndGIC: Variant of Conservative lasso with GIC. J\&M: Procedure of \cite{javanmard2013confidence}.}
\label{E2b}
\end{table}

When adding heteroskedasticity to Experiment 2a, Table \ref{E2b} shows that the estimation errors of all procedures increase slightly. The coverage rate of all procedures is roughly unchanged but the bands become wider. 

\begin{table}[!h]
\centering
\begin{tabular}{ccccccccc}
\toprule
& & & \multicolumn{2}{c}{ $\chi^2$}&\multicolumn{2}{c}{Coverage rate}& \multicolumn{2}{c}{Length}\\
\cmidrule(lr){ 4 - 5 }\cmidrule(lr){ 6 - 7 }\cmidrule(lr){ 8 - 9 }
$\rho=0.75$& & $\ell_2$& Size & Power & non-zero & zero & non-zero & zero \\ 
  \midrule
{\multirow{7}{*}{$n=100$}}
&Lasso & 1.551 & 0.760 & 0.880 & 0.250 & 0.730 & 0.232 & 0.229 \\ 
  &LassoGIC & 3.066 & 0.060 & 0.040 & 0.960 & 0.860 & 1.541 & 1.580 \\ 
  &CLasso & 1.006 & 0.220 & 0.780 & 0.830 & 0.910 & 0.479 & 0.494 \\ 
  &CLassoGIC & 3.066 & 0.060 & 0.040 & 0.960 & 0.850 & 1.551 & 1.588 \\ 
  &CLassoInd & 1.419 & 0.370 & 0.750 & 0.590 & 0.870 & 1.646 & 1.049 \\ 
  &CLassoIndGIC & 3.066 & 0.070 & 0.060 & 0.970 & 0.860 & 1.631 & 1.653 \\ 
  &J\&M & 1.514 & 0.930 & 0.980 & 0.220 & 0.810 & 0.247 & 0.242 \\
   \midrule
{\multirow{7}{*}{$n=150$}}
&Lasso & 1.099 & 0.320 & 0.780 & 0.670 & 0.800 & 0.336 & 0.361 \\ 
  &LassoGIC & 1.400 & 0.090 & 0.340 & 0.960 & 0.840 & 0.579 & 0.616 \\ 
  &CLasso & 0.798 & 0.050 & 0.770 & 0.960 & 0.880 & 0.416 & 0.454 \\ 
  &CLassoGIC & 1.418 & 0.080 & 0.320 & 0.960 & 0.840 & 0.595 & 0.632 \\ 
  &CLassoInd & 0.875 & 0.270 & 0.820 & 0.710 & 0.910 & 0.537 & 0.433 \\ 
  &CLassoIndGIC & 1.432 & 0.090 & 0.410 & 0.960 & 0.880 & 0.669 & 0.720 \\ 
  &J\&M & 0.937 & 0.830 & 0.990 & 0.400 & 0.740 & 0.204 & 0.205 \\ 
\midrule
{\multirow{7}{*}{$n=200$}}
&Lasso & 0.876 & 0.060 & 0.860 & 0.880 & 0.930 & 0.394 & 0.436 \\ 
  &LassoGIC & 1.036 & 0.070 & 0.710 & 0.930 & 0.930 & 0.450 & 0.489 \\ 
  &CLasso & 0.694 & 0.040 & 0.910 & 0.950 & 0.930 & 0.391 & 0.437 \\ 
  &CLassoGIC & 1.002 & 0.060 & 0.750 & 0.930 & 0.930 & 0.458 & 0.497 \\ 
  &CLassoInd & 0.397 & 0.080 & 0.910 & 0.910 & 0.920 & 0.439 & 0.507 \\ 
  &CLassoIndGIC & 0.864 & 0.100 & 0.740 & 0.900 & 0.870 & 0.496 & 0.556 \\ 
  &J\&M & 0.746 & 0.490 & 1.000 & 0.530 & 0.890 & 0.204 & 0.209 \\ 
\midrule
{\multirow{7}{*}{$n=500$}}
&Lasso & 0.494 & 0.080 & 1.000 & 0.930 & 0.960 & 0.246 & 0.282 \\ 
  &LassoGIC & 0.552 & 0.070 & 1.000 & 0.940 & 0.950 & 0.254 & 0.289 \\ 
  &CLasso & 0.254 & 0.060 & 1.000 & 0.920 & 0.970 & 0.250 & 0.295 \\ 
  &CLassoGIC & 0.307 & 0.050 & 1.000 & 0.930 & 0.970 & 0.252 & 0.295 \\ 
  &CLassoInd & 0.139 & 0.080 & 1.000 & 0.930 & 0.970 & 0.263 & 0.329 \\ 
  &CLassoIndGIC & 0.139 & 0.080 & 1.000 & 0.930 & 0.970 & 0.263 & 0.329 \\ 
  &J\&M & 0.420 & 0.150 & 1.000 & 0.770 & 0.930 & 0.193 & 0.217 \\
\bottomrule
\end{tabular}
\caption{\footnotesize Summary statistics for Experiment 3a. $\ell_2$: average $\ell_2$-estimation error, $\chi^2$: Size and Power report the size and power of the hypotheses $H_0: (\beta_{0,1}, \beta_{0,2})=(1,0)$ and $H_0: (\beta_{0,1}, \beta_{0,2})=(1,0.4)$, respectively. Coverage rate: the actual coverage rate of the asymptotically gaussian 95\% confidence interval for $\beta_{0,1}$ and $\beta_{0,2}$. Length: the length of the two confidence intervals mentioned above.  Lasso: Lasso with BIC. LassoGIC: Lasso with GIC.
CLasso: Conservative Lasso with BIC.  CLassoGIC: Conservative Lasso with GIC. CLassoInd: Variant of Conservative Lasso with BIC. CLassoIndGIC: Variant of Conservative lasso with GIC. J\&M: Procedure of \cite{javanmard2013confidence}.}
\label{E3a}
\end{table}

The results for the very high-dimensional Experiment 3a are found in Table \ref{E3a}. Here GIC performs quite badly (for low values of $n$) for all methods and we thus focus on the results for BIC. When the sample size is $n=100$, the plain Lasso has an estimation error which is 50\% larger than the one of the conservative Lasso. Furthermore, the $\chi^2$-test based on the Lasso is so size distorted (the size is 76\%) that its usefulness may be questioned. While the conservative Lasso also suffers from size distortion (the size is 22\%) it is still \textit{much} more reliable than the Lasso. The version of the conservative Lasso lies in between in terms of estimation error and size of the $\chi^2$-test. The procedure of \cite{javanmard2013confidence} is severely size distorted when $n=100$ but this gradually improves as the sample size is increased. 

Turning to the coverage rates of the confidence intervals of the non-zero coefficients, the Lasso provides such a poor coverage (25 \%) that it may almost be deemed useless. The conservative Lasso, while not being perfect, still has a coverage of 83\%. It also performs much better for the truly zero parameter than the Lasso. The superior coverage of conservative Lasso is due to much more precise parameter estimates and wider confidence bands than the Lasso. The coverage of the version of the conservative Lasso is higher than for the Lasso but lower than for the conservative Lasso.

When the sample size is increased to just $n=150$ the conservative Lasso performs well along all dimensions even in this high-dimensional setting. The size distortion has disappeared and the coverage for the non-zero parameter has increased to 96\% (from 83\%). The Lasso has also improved. However, it is remarkable that the size of its $\chi^2$-test for $n=150$ is still higher than the one for the conservative Lasso when $n=100$. Similarly, the coverage rate of the confidence bands for the zero as well as the non-zero parameters based on the Lasso is still lower than the one the conservative Lasso produced for $n=100$. 

 
Next, for $n=200$, the conservative Lasso still estimates the parameters much more precisely than the plain Lasso. It also has better size and power properties but the gap has narrowed as these quantities approach their asymptotic values of $0.05$ and 1, respectively. Regarding the coverage rate, the conservative Lasso also remains the superior procedure. The variant of the conservative Lasso now actually delivers the lowest estimation error which is in accordance with our initial observation of the variant performing relatively well as $p/n$ decreases. 

Finally, for $n=500$, both procedures work very well, but the conservative Lasso remains by far the most precise estimator in terms of $\ell_2$-estimation error (three times as precise as the plain Lasso). The size distortion of the procedure of \cite{javanmard2013confidence} is now only moderate while its confidence bands still undercover the non-zero coefficient.

\begin{table}[!h]
\centering
\begin{tabular}{ccccccccc}
\toprule
& & & \multicolumn{2}{c}{ $\chi^2$}&\multicolumn{2}{c}{Coverage rate}& \multicolumn{2}{c}{Length}\\
\cmidrule(lr){ 4 - 5 }\cmidrule(lr){ 6 - 7 }\cmidrule(lr){ 8 - 9 }
$\rho=0.75$& & $\ell_2$& Size & Power & non-zero & zero & non-zero & zero \\ 
  \midrule
{\multirow{7}{*}{$n=100$}}
&Lasso & 1.667 & 0.680 & 0.880 & 0.300 & 0.870 & 0.297 & 0.271 \\ 
  &LassoGIC & 3.074 & 0.080 & 0.050 & 0.950 & 0.860 & 1.664 & 1.616 \\ 
  &CLasso & 1.225 & 0.370 & 0.790 & 0.640 & 0.920 & 0.557 & 0.512 \\ 
  &CLassoGIC & 3.074 & 0.080 & 0.050 & 0.950 & 0.860 & 1.673 & 1.623 \\ 
  &CLassoInd & 1.578 & 0.380 & 0.730 & 0.600 & 0.890 & 1.814 & 1.120 \\ 
  &CLassoIndGIC & 3.092 & 0.090 & 0.060 & 0.950 & 0.870 & 1.765 & 1.711 \\ 
  &J\&M & 1.610 & 0.860 & 0.980 & 0.330 & 0.840 & 0.360 & 0.330 \\
   \midrule
{\multirow{7}{*}{$n=150$}}
&Lasso & 1.206 & 0.370 & 0.710 & 0.690 & 0.900 & 0.465 & 0.424 \\ 
  &LassoGIC & 1.693 & 0.120 & 0.290 & 0.950 & 0.930 & 0.841 & 0.800 \\ 
  &CLasso & 0.906 & 0.100 & 0.690 & 0.910 & 0.960 & 0.592 & 0.550 \\ 
  &CLassoGIC & 1.703 & 0.110 & 0.280 & 0.950 & 0.930 & 0.850 & 0.812 \\ 
  &CLassoInd & 1.070 & 0.370 & 0.800 & 0.640 & 0.910 & 0.527 & 0.478 \\ 
  &CLassoIndGIC & 1.708 & 0.090 & 0.360 & 0.900 & 0.940 & 0.910 & 0.910 \\ 
  &J\&M & 1.040 & 0.810 & 0.980 & 0.480 & 0.910 & 0.352 & 0.325 \\ 
\midrule
{\multirow{7}{*}{$n=200$}}
&Lasso & 0.978 & 0.150 & 0.680 & 0.850 & 0.930 & 0.548 & 0.517 \\ 
  &LassoGIC & 1.170 & 0.120 & 0.520 & 0.880 & 0.920 & 0.622 & 0.587 \\ 
  &CLasso & 0.856 & 0.110 & 0.730 & 0.850 & 0.950 & 0.561 & 0.531 \\ 
  &CLassoGIC & 1.150 & 0.100 & 0.520 & 0.900 & 0.930 & 0.632 & 0.598 \\ 
  &CLassoInd & 0.628 & 0.120 & 0.750 & 0.860 & 0.970 & 0.586 & 0.590 \\ 
  &CLassoIndGIC & 1.067 & 0.090 & 0.600 & 0.890 & 0.960 & 0.667 & 0.671 \\ 
  &J\&M & 0.842 & 0.610 & 0.990 & 0.550 & 0.910 & 0.340 & 0.304 \\  
\midrule
{\multirow{7}{*}{$n=500$}}
&Lasso & 0.548 & 0.100 & 0.980 & 0.880 & 0.940 & 0.380 & 0.332 \\ 
  &LassoGIC & 0.610 & 0.100 & 0.970 & 0.890 & 0.950 & 0.389 & 0.341 \\ 
  &CLasso & 0.316 & 0.100 & 1.000 & 0.890 & 0.950 & 0.381 & 0.341 \\ 
  &CLassoGIC & 0.378 & 0.100 & 1.000 & 0.890 & 0.940 & 0.384 & 0.343 \\ 
  &CLassoInd & 0.177 & 0.100 & 0.990 & 0.910 & 0.950 & 0.387 & 0.367 \\ 
  &CLassoIndGIC & 0.182 & 0.100 & 0.990 & 0.910 & 0.950 & 0.387 & 0.367 \\ 
  &J\&M & 0.473 & 0.190 & 1.000 & 0.780 & 0.930 & 0.327 & 0.272 \\ 
\bottomrule
\end{tabular}
\caption{\footnotesize Summary statistics for Experiment 3b. $\ell_2$: average $\ell_2$-estimation error, $\chi^2$: Size and Power report the size and power of the hypotheses $H_0: (\beta_{0,1}, \beta_{0,2})=(1,0)$ and $H_0: (\beta_{0,1}, \beta_{0,2})=(1,0.4)$, respectively. Coverage rate: the actual coverage rate of the asymptotically gaussian 95\% confidence interval for $\beta_{0,1}$ and $\beta_{0,2}$. Length: the length of the two confidence intervals mentioned above. Lasso: Lasso with BIC. LassoGIC: Lasso with GIC.
CLasso: Conservative Lasso with BIC.  CLassoGIC: Conservative Lasso with GIC. CLassoInd: Variant of Conservative Lasso with BIC. CLassoIndGIC: Variant of Conservative lasso with GIC. J\&M: Procedure of \cite{javanmard2013confidence}.}
\label{E3b}
\end{table}

\begin{table}[!h]
\centering
\begin{tabular}{ccccccccc}
\toprule
& & & \multicolumn{2}{c}{ $\chi^2$}&\multicolumn{2}{c}{Coverage rate}& \multicolumn{2}{c}{Length}\\
\cmidrule(lr){ 4 - 5 }\cmidrule(lr){ 6 - 7 }\cmidrule(lr){ 8 - 9 }
$\rho=0.5$& & $\ell_2$& Size & Power & non-zero & zero & non-zero & zero \\ 
  \midrule
{\multirow{7}{*}{$n=100$}}
  &Lasso & 0.337 & 0.174 & 0.640 & 0.928 & 0.823 & 0.439 & 0.436 \\ 
  &LassoGIC & 0.354 & 0.187 & 0.600 & 0.937 & 0.790 & 0.451 & 0.442 \\ 
  &CLasso & 0.315 & 0.160 & 0.678 & 0.924 & 0.846 & 0.435 & 0.437 \\ 
  &CLassoGIC & 0.343 & 0.181 & 0.629 & 0.930 & 0.813 & 0.448 & 0.441 \\ 
  &CLassoInd & 0.282 & 0.161 & 0.807 & 0.911 & 0.916 & 0.419 & 0.431 \\ 
  &CLassoIndGIC & 0.334 & 0.200 & 0.766 & 0.919 & 0.881 & 0.432 & 0.434 \\ 
  &J\&M & 0.310 & 0.597 & 0.930 & 0.787 & 0.767 & 0.316 & 0.301 \\
\bottomrule
\end{tabular}
\caption{\small Summary statistics for Experiment 4. $\ell_2$: average $\ell_2$-estimation error, $\chi^2$: Size and Power report the size and power of the hypotheses $H_0: (\beta_{0,1}, \beta_{0,2})=(1,0,1,0.1,0,0,0,0,0,0)$ and $H_0: (\beta_{0,1}, \beta_{0,2})=(1,0.4,1,0.1,0,0,0,0,0,0)$, respectively. Coverage rate: the actual coverage rate of the asymptotically gaussian 95\% confidence interval for $\beta_{0,1}$ and $\beta_{0,2}$. Length: the length of the two confidence intervals mentioned above. Lasso: Lasso with BIC. LassoGIC: Lasso with GIC.
CLasso: Conservative Lasso with BIC.  CLassoGIC: Conservative Lasso with GIC. CLassoInd: Variant of Conservative Lasso with BIC. CLassoIndGIC: Variant of Conservative lasso with GIC. J\&M: Procedure of \cite{javanmard2013confidence}.}
\label{E4a}
\end{table}

Table \ref{E3b} adds heteroskedasticity to the results in Table \ref{E3a}. Qualitatively nothing changes in the sense that the rankings between the Lasso and the conservative Lasso remain the same in terms of estimation precision, size, power and coverage for all sample sizes. The conservative Lasso again estimates the parameters more precisely and has much better size and coverage properties. For $n=500$ both procedures work well but as usual the conservative Lasso remains the most precise estimator in terms of $\ell_2$-estimation error.

Table \ref{E4a} considers the effect of testing a hypothesis involving many parameters. The results should be compared to those of Table \ref{E2a}. The main message is that the size of the Lasso based tests only inflates slightly compared to the case where only two parameters were involved in the hypothesis. Among the Lasso based tests the inflation is largest for the variant of the conservative Lasso.  The size of \cite{javanmard2013confidence} increases by much more. Furthermore, the conservative Lasso is still found to slightly outperform the plain Lasso in terms of size and power.

\section{Conclusion}\label{Co}
This paper shows how the conservative Lasso can be used to conduct inference in the high-dimensional linear regression model. We allow for conditional heteroskedasticity in the error terms and also show how to consistently estimate the population covariance matrix in this case. In fact, the convergence is uniform over sparse sub vectors of the parameter space. Next, we show that the confidence bands based on the desparsified conservative are honest and that they contract at the optimal rate. This rate of contraction is also uniform over sparse sub vectors of the parameter space. $\chi^2$-inference is also briefly discussed. Our simulations show that the conservative Lasso provides much more precise parameter estimates than the plain Lasso and that tests based on it have superior size properties. Furthermore, confidence intervals based on the desparsified conservative Lasso have better coverage rates than the ones based on the desparsified plain Lasso. Future work may include bootstrapping the desparsified conservative Lasso to gain further finite sample improvements. 

\setcounter{equation}{0}\setcounter{lemma}{0}\renewcommand{\theequation}{A.%
\arabic{equation}}\renewcommand{\thelemma}{A.\arabic{lemma}}%
\renewcommand{\baselinestretch}{1}\baselineskip=15pt

\section*{Appendix}

In Appendix A we begin by providing some auxiliary lemmas used for the proofs of the main results in Appendix B. 
The details of (\ref{2.8}) can be found in Appendix C.

\subsection*{Appendix A -- auxiliary lemmas}

First, we provide the proof of Lemma \ref{lt1} in the main text.

\begin{proof}[Proof of Lemma \ref{lt1}]
(i). Note that by (\ref{thetaj}) with $H=\cbr[0]{1,...,p}$ it follows under Assumptions 1 and 2 that
\begin{equation}
\| \Theta \|_{\ell_{\infty}} = \max_{1 \le j \le p} \| \Theta_j \|_1 = O (\sqrt{\max_{1 \le j \le p} s_j}),\label{01}
\end{equation}
It actually also follows from (\ref{sbar}) that (since $\hat{\Theta}_L$ is a subcase of $\hat{\Theta}$)
\begin{equation}
\| \hat{\Theta}_L \|_{\ell_{\infty}} = \max_{1 \le j \le p} \| \hat{\Theta}_{L,j} \|_1 = O_p (\sqrt{\max_{1 \le j \le p} s_j}),\label{02}
\end{equation}
Thus,
\begin{equation}
\lambda_{prec} = O (\lambda_n \sqrt{\max_{1 \le j \le p} s_j})\label{03}
\end{equation}
where
\begin{equation}
\lambda_n \sqrt{\max_{1 \le j \le p} s_j} = \frac{M p^{2/r}}{\sqrt{n}} \sqrt{\max_{1 \le j \le p} s_j} = M \left[ \frac{p^2 (\max_{1 \le j \le p} 
s_j)^{r/2}}{n^{r/4}} \right]^{1/r} \frac{1}{n^{1/4}} \to 0,\label{04}
\end{equation}
by Assumption 2b. Therefore, we get $\lambda_{prec} \to 0$. Note that replacing $\|\Theta \|_{\ell_{\infty}} $ by $\|\hat{\Theta}_L  \|_{\ell_{\infty}}$ in the definition of $\lambda_{prec}$ makes no difference since by (\ref{02}) we still get $\lambda_{prec} \stackrel{p}{\to}0$.

(ii). By Lemma \ref{lemma9} the set  ${\cal C}_1 = \{ \|\hat{\beta}_L - \beta_0 \|_{\infty} \le \lambda_{prec} \}$ has probability approaching one. First, note that on $\mathcal{C}_1$ one has $\max_{j\in S_0^c}|\hat{\beta}_{L,j}| = \max_{j\in S_0^c}| \hat{\beta}_{L,j} - \beta_{0,j} | \le \lambda_{prec}$. Thus, by the definition of $\min_{j\in S_0^c}\hat{w}_j\to 1$ on $\mathcal{C}_1$.

(iii). On $\mathcal{C}_1$
%
\begin{eqnarray}
\min_{j\in S_0}| \hat{\beta}_{L,j} | &\ge &  | \beta_{0,j}| - |\hat{\beta}_{L,j} - \beta_{0,j}| 
 \ge 
\min_{j\in S_0}(|\beta_{0,j}| - \lambda_{prec} )
= 
\lambda_{prec} \min_{j\in S_0}\left[ \left| \frac{\beta_{0,j}}{\lambda_{prec}}\right| - 1 \right].\label{05}
\end{eqnarray}
Thus, since $\min_{j \in S_0} |\beta_{0,j} |/\lambda_{prec}\to\infty$ we have that $\min_{j\in S_0}| \hat{\beta}_{L,j} |\geq \lambda_{prec}$ for $n$ sufficiently large.
Hence, by (\ref{05}), on ${\cal C}_1$ which has probability tending to one,
\begin{equation}
\max_{j\in S_0}\hat{w}_j  = \frac{\lambda_{prec}}{\min_{j\in S_0}|\hat{\beta}_{L,j}| \vee \lambda_{prec}}  = \frac{\lambda_{prec}}{\min_{j\in S_0}|\hat{\beta}_{L,j}|}
\le \frac{1}{\min_{j\in S_0}\frac{|\beta_{0,j}|}{\lambda_{prec}} -1 } \to 0.\label{07}
\end{equation}
\end{proof}

Now, we provide an oracle inequality for a general weighted Lasso which satisfies certain assumptions and then utilize that the plain Lasso and the conservative Lasso satisfy these assumptions. Define 
\begin{align*}
 \hat{\beta}_w = \argmin_{\beta \in \mathbb{R}^p} \del[2]{ \| Y - X \beta\|_n^2 + 2\lambda_n \sum_{j=1}^p \hat{w}_{g,j} |\beta_j| },
\end{align*}
where $\hat{w}_{g,j}$ denotes a general weight. When $\hat{w}_{g,j} =1$ one recovers the Lasso, when $\hat{w}_{g,j}=\hat{w}_j$ the result is the conservative Lasso. In particular, we shall work on the intersection of $\mathcal{A}=\cbr[1]{\enVert[0]{X'u/n}_\infty \leq \lambda_n/2}$ and $\mathcal{B}=\cbr[1]{\phi^2_{\hat{\Sigma}}\geq \phi^2_\Sigma/2}$. On these sets we have a handle on the maximal empirical ``correlation" between the covariates and the error terms, and a lower bound on the empirical adaptive restricted eigenvalue, respectively. Define $a_n=\enVert[0]{\hat{w}_{S_0}}_{\infty}$.

\begin{lemma}\label{A1}
Let $\hat{w}_{g,S_0^c}^{min} = \min_{j \in S_0^c} \hat{w}_j=1$ and $a_n\leq 1$. Then, on the set $\mathcal{A}\cap\mathcal{B}$ the following inequalities are valid.
\begin{align}
\| X (\hat{\beta}_w - \beta_0) \|_n^2 
\le 
2 (2a_n+1)^2 \frac{ \lambda_n^2s_0 }{\phi^2_\Sigma (s_0)}\label{IQ1}.\\
\|\hat{\beta}_w - \beta_0 \|_1 
\le 
4 (a_n+1) (2 a_n +1) \frac{\lambda_ns_0}{\phi^2_\Sigma (s_0)}\label{IQ2}. 
\end{align}

\end{lemma}
\begin{proof} 
We begin by establishing (\ref{IQ1}). By the minimizing property of $\hat{\beta}_w$ it follows that
\begin{equation}
\| Y - X \hat{\beta}_w \|_n^2 + 2\lambda_n \sum_{j=1}^p \hat{w}_{g,j}  | \hat{\beta}_{w,j}| 
\leq
\| Y - X \beta_0 \|_n^2 + 2\lambda_n  \sum_{j=1}^p \hat{w}_{g,j} |\beta_{0,j}|.\label{cl.1}
\end{equation}
Inserting $Y=X\beta_0+u$, using Hölder's inequality, and using that we are on the set $ \mathcal{A}$ we arrive at
\begin{equation}
\| X (\hat{\beta}_w - \beta_0) \|_n^2 + 2\lambda_n \sum_{j=1}^p \hat{w}_{g,j} | \hat{\beta}_{w,j}| 
\le 
\lambda_n \|\hat{\beta}_w - \beta_0 \|_1 + 2\lambda_n \sum_{j=1}^p \hat{w}_{g,j} | \beta_{0,j}|.\label{cl.2}
\end{equation} 
Then, using $ \|\hat{\beta}_w \|_1 = \| \hat{\beta}_{w,S_0} \|_1 + \| \hat{\beta}_{w,S_0^c} \|_1$ one gets
\begin{align}
 \| X (\hat{\beta}_w - \beta_0) \|_n^2 + 2\lambda_n \sum_{j \in S_0^c} \hat{w}_{g,j} | \hat{\beta}_{w,j}|
& \le  \lambda_n \| \hat{\beta}_w - \beta_0 \|_1 - 2\lambda_n\sum_{j \in S_0} \hat{w}_{g,j} | \hat{\beta}_{w,j} | + 2\lambda_n \sum_{j=1}^p \hat{w}_{g,j} | \beta_{0,j} | \nonumber \\
& \le  \lambda_n \|\hat{\beta}_w - \beta_0 \|_1 + 2\lambda_n \sum_{j \in S_0} \hat{w}_{g,j} | \hat{\beta}_{w,j} - \beta_{0,j} |.\label{cl.2a}
\end{align} 
Noting that $\| \hat{\beta}_w - \beta_0 \|_1 = \| \hat{\beta}_{w,S_0} - \beta_{0,S_0} \|_1 +\| \hat{\beta}_{w,S_0^c} \|_1$
and $ \sum_{j \in S_0^c} \hat{w}_{g,j} |\hat{\beta}_{w,j}| \ge \hat{w}_{S_0^c}^{min} \| \hat{\beta}_{w, S_0^c} \|_1=\enVert[0]{\hat{\beta}_{w, S_0^c}}$ rewrite (\ref{cl.2a}) as 
\begin{equation}
 \| X (\hat{\beta}_w - \beta_0) \|_n^2 + 2\lambda_n  \|\hat{\beta}_{w,S_0^c} \|_1
 \le  
\lambda_n \|\hat{\beta}_{w,S_0} - \beta_{0,S_0} \|_1 + \lambda_n \|\hat{\beta}_{w, S_0^c} \|_1
 +2\lambda_n  \sum_{j \in S_0} \hat{w}_{g,j} | \hat{\beta}_{w,j} - \beta_{0,j} |.\label{cl.2aa}
\end{equation}
Subtract $\lambda_n \|\hat{\beta}_{w, S_0^c} \|_1$ from both sides of (\ref{cl.2aa}) to get
\begin{equation}
 \| X (\hat{\beta}_w - \beta_0) \|_n^2 + \lambda_n  \|\hat{\beta}_{w,S_0^c} \|_1
 \le  
\lambda_n \|\hat{\beta}_{w,S_0} - \beta_{0,S_0} \|_1
 +2\lambda_n  \sum_{j \in S_0} \hat{w}_{g,j} | \hat{\beta}_{w,j} - \beta_{0,j} |.\label{cl.2ab}
\end{equation}
Next, use the Cauchy-Schwarz inequality, $\|.\|_1 \le  \sqrt{s_0} \|.\|_2$, as well as $\|\hat{w}_{g,S_0}\|_2\leq a_n \sqrt{s_0}$, and $0 < a_n \le 1$ to get
\begin{align} 
\| X (\hat{\beta}_w - \beta_0) \|_n^2 + \lambda_n \|\hat{\beta}_{w, S_0^c} \|_1 
&\le 
\lambda_n \sqrt{s_0}\|\hat{\beta}_{w, S_0} - \beta_{0, S_0} \|_2 + 2\lambda_n \|\hat{w}_{g,S_0}\|_2 \|\hat{\beta}_{w,S_0} - \beta_{0,S_0} \|_2\notag\\
&\le
 (2 a_n +1)  \lambda_n \sqrt{s_0} \|\hat{\beta}_{w,S_0} - \beta_{0,S_0} \|_2\label{cl4.00}\\
&
\le 
3\lambda_n\sqrt{s_0}\|\hat{\beta}_{w,S_0} - \beta_{0,S_0} \|_2. \label{cl.4}
\end{align}
(\ref{cl.4}) implies that
\[ \|\hat{\beta}_{w, S_0^c}\|_1 \le  3  \sqrt{s_0} \|\hat{\beta}_{w,S_0} - \beta_{0,S_0} \|_2.\]
Hence, by the adaptive restricted eigenvalue condition, (\ref{cl4.00}) implies
\begin{equation}
\| X (\hat{\beta}_w - \beta_0) \|_n^2 + \lambda_n \|\hat{\beta}_{w, S_0^c} \|_1
 \le 
(2 a_n +1)  \lambda_n \sqrt{s_0}  \frac{ \| X ( \hat{\beta}_w - \beta_0) \|_n}{\phi_{\hat{\Sigma}}(s_0)}.\label{cl.5}
\end{equation}
Then, using $(2 a_n +1)  uv \le u^2/2 + (2 a_n +1)^2 v^2/2$, with $v=\lambda_n \sqrt{s_0}/\phi_{\hat{\Sigma}}(s_0)$, $u=\| X ( \hat{\beta}_w - \beta_0) \|_n$, one gets 
\begin{equation}
\| X (\hat{\beta}_w - \beta_0) \|_n^2 + \lambda_n \|\hat{\beta}_{w, S_0^c} \|_1
 \le \frac{\| X ( \hat{\beta}_w - \beta_0) \|_n^2}{2}+ \frac{2 (a_n+1)^2}{2} \frac{ \lambda_n^2 s_0 }{\phi_{\hat{\Sigma}}^2(s_0)}.\label{cl.6}
\end{equation}
Subtracting the first right hand side term in (\ref{cl.6}) from the left and right hand sides of (\ref{cl.6}) and multiplying all terms by 2 yields
\begin{equation}
 \| X (\hat{\beta}_w - \beta_0) \|_n^2 +  2\lambda_n \|\hat{\beta}_{w, S_0^c} \|_1
\le 
(2 a_n+1)^2 \frac{ \lambda_n^2 s_0 }{\phi_{\hat{\Sigma}}^2(s_0)},\label{cl.7}
\end{equation}
which, using that we are on $\mathcal{B}$, implies (\ref{IQ1}).

Next, we turn to proving (\ref{IQ2}). By adding $\lambda_n\enVert[0]{\hat{\beta}_{w,S_0}-\beta_{0,S_0}}_1$ to both sides of (\ref{cl4.00}) and using 
$\|.\|_1 \le \sqrt{s_0} \|.\|_2$ one gets
%
\begin{align}
\lambda_n \|\hat{\beta}_w - \beta_0 \|_1 
&\leq
\lambda_n \|\hat{\beta}_{w,S_0} - \beta_{0,S_0} \|_1 + (2 a_n +1) \lambda_n \sqrt{s_0} \|\hat{\beta}_{w, S_0} - \beta_{0,S_0} \|_2\\
&\leq
2 ( a_n +1) \lambda_n \sqrt{s_0} \|\hat{\beta}_{w, S_0} - \beta_{0,S_0} \|_2.
\end{align}
The adaptive restricted eigenvalue condition and inequality (\ref{IQ1}) of this Lemma yield
\begin{align}
 \|\hat{\beta}_w - \beta_0 \|_1 
\leq
2 (a_n +1)  \sqrt{s_0} \frac{ \| X (\hat{\beta}_{w} - \beta_0)\|_n}{\phi_{\hat{\Sigma}}(s_0)}
\le 
 \frac{4 (a_n+1) (2 a_n +1) s_0\lambda_n}{\phi_{\Sigma}^2(s_0)},
\end{align}
which, using that we are on $\mathcal{B}$, implies (\ref{IQ2}).
\end{proof}
    
To prove Lemma \ref{Lasso} and Theorem \ref{thm1} it suffices to provide a lower bound on the probabilities of $\mathcal{A}$ and $\mathcal{B}$. To do so, recall the Marcinkiewicz-Zygmund inequality:
\begin{lemma}\label{MZ}[Marcinkiewicz-Zygmund inequality, see \cite{linb10}, result 9.7.a]
Let $\cbr{U_i}_{i=1}^n$ be a sequence of independent mean zero real random variables with finite $r'th$ moment. Then, for positive constants $a_r$ and $b_r$, only depending on $r$, $r\ge2$
\begin{align}
a_rE\del[3]{\sum_{i=1}^nU_i^2}^{r/2}\leq E\envert[3]{\sum_{i=1}^nU_i}^r
\leq
b_r E\del[3]{\sum_{i=1}^nU_i^2}^{r/2}\label{MZ1}
\end{align}
\end{lemma}
Note in particular that, by an application of the summation version of Jensen's inequality on the convex map $x\mapsto x^{r/2}$, (\ref{MZ1}) implies that 
\begin{align*}
E\envert[3]{\sum_{i=1}^nU_i}^r
\le
b_r n^{r/2}E\del[3]{\frac{1}{n}\sum_{i=1}^nU_i^2}^{r/2}
\leq
 b_rn^{r/2-1}\sum_{i=1}^nE\envert{U_i}^r
\leq
b_rn^{r/2}\max_{1\leq i\leq n}E\envert[0]{U_i}^r.
\end{align*}
Hence, by a union bound and Markov's inequality we arrive at the following result which we shall use frequently throughout the appendix.

\begin{lemma}\label{conc}
For each $j\in\cbr[0]{1,...,m}$ let $\cbr{U_{j,i}}_{i=1}^n$ be a sequence of independent mean zero real random variables with finite $r'th$ moment and define $S_{j,n}=\sum_{i=1}^nU_{j,i}$. Then,
\begin{align*}
P\del[2]{\max_{1\leq j\leq m}|S_{j,n}|\geq t}\leq b_r m\frac{n^{r/2}\max_{1\leq j\leq m}\max_{1\leq i\leq n}E|U_{j,i}|^r}{t^r}.
\end{align*}
\end{lemma}

{\bf Remarks}: 1. In Lemma \ref{conc} above we used the Marcinkiewicz-Zygmund inequality.  Another common approach is using Nemirowski's inequality, see \cite{van2014}. We show that application of Nemirowski's inequality will bring an additional $\left(8 \log (2m) \right)^{r/2}$ in Lemma \ref{conc}. To make this point clear, for $r\ge 2$, note that Nemirovski's inequality in Lemma 14.24 of \cite{van2014} yields  
\begin{align}
E \del[2]{\max_{1\leq j\leq m} |S_{j,n}|^r} 
\le
 \left(8 \log (2m) \right)^{r/2} E \sbr[3]{ \max_{1 \le j \le m} \sum_{i=1}^n U_{j,i}^2}^{r/2}.\label{nem}
\end{align}
Thus, we need to bound $E \sbr[1]{ \max_{1 \le j \le m} \sum_{i=1}^n U_{j,i}^2}^{r/2}$. By convexity of $x\mapsto x^{r/2}$ and Jensen's inequality
\begin{align*}
E \sbr[3]{ \max_{1 \le j \le m} \sum_{i=1}^n U_{j,i}^2}^{r/2}
&=
n^{r/2}E  \max_{1 \le j \le m} \sbr[3]{\frac{1}{n}\sum_{i=1}^nU_{j,i}^2}^{r/2}
\leq
n^{r/2}E \max_{1 \le j \le m} \frac{1}{n}\sum_{i=1}^n |U_{j,i}|^r\\
&\leq
n^{r/2-1}E \sum_{j=1}^m \sum_{i=1}^n |U_{j,i}|^r
\leq
n^{r/2}m\max_{1\leq j\leq m}\max_{1\leq i\leq n}E|U_{j,i}|^r.
\end{align*}
Inserting the above display into (\ref{nem}) and using Markov's inequality yields
\begin{align*}
P\del[2]{\max_{1\leq j\leq m}|S_{j,n}|\geq t}
\leq \frac{\left(8 \log (2m) \right)^{r/2}n^{r/2}m\max_{1\leq j\leq m}\max_{1\leq i\leq n}E|U_{j,i}|^r}{t^r}.
\end{align*}
Note that the above bound, relying on Nemirovski's inequality, is larger by a factor $\left(8 \log (2m) \right)^{r/2}$ (which increases in $m$) than the bound in Lemma \ref{conc}. This will result in lower choices of the tuning parameter and hence sharper bounds. This is a new theoretical contribution of the paper.

2. In a seminal paper about optimal instrumental variable selection, \cite{vicecta2012} use self-normalized moderate deviation results to get the tuning parameter and its rate. They propose a heteroskedasticity consistent penalty term unlike our data dependent penalty which focuses on creating a wedge between zero and nonzero parameters.  Condition  RF (iii) in the analysis of \cite{vicecta2012} results in $\log^3(p)/n=o(1)$. However, our rate for $\lambda_n$ will require $p^{2/r}/n^{1/2} \to 0$. The reason for this is we are interested in maxima of sums, as in the previous lemma, unlike \cite{vicecta2012} who use maxima of self normalized sum (i.e. sum normalized by the $\ell_2$ norm of the vector of variables) which provides their rate.


We are now ready to provide a lower bound on the probability of $\mathcal{A}$.
\begin{lemma}\label{Noise}
Let $M>0$ be an arbitrary positive number. Then, under Assumption 1, for $\lambda_n=M\frac{p^{2/r}}{\sqrt{n}}$ the set $\mathcal{A}=\cbr[1]{\enVert[0]{X'u/n}_\infty \leq \lambda_n/2}$ has probability at least $1-\frac{C}{M^{r/2}}$, for a universal constant $C>0$.
\end{lemma}

\begin{proof}
For each $j\in \cbr[0]{1,...,p}$, $\cbr[0]{X_{j,i}u_i}_{i=1}^n$ is a sequence of independent mean zero random variables with $(r/2)'th$ moment $E|X_{j,i}u_i|^{r/2}\leq \sqrt{E|X_{j,i}|^rE|u_i|^r}\leq C$. Hence, Lemma \ref{conc} yields
\begin{align*}
P(\mathcal{A}^c)
=
P\del[2]{\enVert[0]{X'u}_\infty > n\lambda_n/2}
\leq
p\frac{b_{r/2}Cn^{r/4}}{(n\lambda_n/2)^{r/2}}
=
\frac{C}{M^{r/2}}, 
\end{align*}
where the last equality follows from the choice of $\lambda_n$ and has merged the constants.   
\end{proof}

The next two lemmas will provide a lower bound on the probability of set ${\cal B}$.
\begin{lemma}\label{vdGB}
Let $A$ and $B$ be two positive semi-definite $p\times p$ matrices and assume that $A$ satisfies the restricted eigenvalue condition RE($s$) for some $\phi_A(s)>0$. Then, for $\delta=\max_{1\leq i,j\leq p}\envert[0]{A_{i,j}-B_{i,j}}$, one also has $\phi_B^2\geq \phi_A^2-16s\delta$.
\end{lemma}

\begin{proof}
The proof is similar to Lemma 10.1 in \cite{vdGB09}. For any (non-zero) $p\times 1$ vector $v$ such that $\enVert[0]{v_{S^c}}_{1}\leq 3\sqrt{s}\enVert[0]{v_S}_{2}$ one has
\begin{align*}
v'Av-v'Bv
&\leq
\envert[0]{v'Av-v'Bv}
=
\envert[0]{v'(A-B)v}
\leq 
\enVert[0]{v}_{1}\enVert[0]{(A-B)v}_{\infty}
\leq
\delta\enVert[0]{v}_{1}^2\\
&=\delta\del[1]{\enVert[0]{v_S}_{1}+\enVert[0]{v_{S^c}}_1}^2
\leq 
\delta 16s\enVert[0]{v_S}_{2}^2.
\end{align*} 
Hence, rearranging the above, yields
\begin{align*}
v'Bv
\geq
v'Av - 16s\delta\enVert[0]{v_S}_2^2, 
\end{align*}
or equivalently,
\begin{align*}
\frac{v'Bv}{v_S'v_S}
\geq
\frac{v'Av}{v_S'v_S}-16s\delta. 
 \end{align*}
Minimizing  over $\cbr[0]{v\in \mathbb{R}^n\setminus \{0\}: \enVert[0]{v_{S^c}}_{1}\leq 3\sqrt{s}\enVert[0]{v_S}_{2}}$ and using the adaptive restricted eigenvalue condition yields the claim.
\end{proof}

In order to verify the restricted eigenvalue condition we present the following lemma.

\begin{lemma}\label{RE1}
Let Assumption 1 be satisfied. Then, the set $\mathcal{B}=\cbr[1]{\phi^2_{\hat{\Sigma}}\geq \phi^2_\Sigma/2}$ has probability at least $1-D\frac{p^2s^{r/2}_0}{n^{r/4}}$ for a universal constant $D>0$. 
\end{lemma}

\begin{proof}
By Lemma \ref{vdGB}, with $s=s_0$, it suffices to show that $\delta= \enVert[0]{\hat{\Sigma}-\Sigma}_\infty \leq \frac{\phi^2_\Sigma(s_0)}{32s_0}$. The $(k,l)$ entry of $\hat{\Sigma}-\Sigma$ is given by $\frac{1}{n}\sum_{i=1}^n\del[1]{X_{k,i}X_{l,i}-E(X_{k,i}X_{l,i})}$. Each summand has mean zero and $E\envert[0]{X_{k,i}X_{l,i}-E(X_{k,i}X_{l,i})}^{r/2}$ is bounded by a universal constant $D$ by the Cauchy-Schwarz inequality. Hence, merging constants, Lemma \ref{conc} yields 
\begin{align*}
P(\mathcal{B}^c)
\leq
P\del[3]{\enVert[0]{\hat{\Sigma}-\Sigma}_\infty> \frac{\phi^2_\Sigma(s_0)}{32s_0}}
\leq p^2\frac{Dn^{r/4}}{(\frac{n}{s_0})^{r/2}}
=
D\frac{p^2s^{r/2}_0}{n^{r/4}}.
\end{align*}
\end{proof}

\begin{lemma}\label{lemma9}
Let Assumption 1 be satisfied. Then on $\mathcal{A}\cap\mathcal{B}$ (defined prior to Lemma \ref{A1})
\begin{equation}
\|\hat{\beta}_L - \beta_0 \|_{\infty} \le (\frac{9 \lambda_n}{4})  \| \Theta \|_{\ell_{\infty}},\label{linfty}
\end{equation}
and $\mathcal{A}\cap\mathcal{B}$ occurs with probability at least  $1- \frac{C}{M^{r/2}} -  \frac{D p^2 s_0^{r/2} }{n^{r/4}}$.
\end{lemma}

\begin{proof}
By Lemma 2.5.1 of \cite{vdG14}

\begin{equation}
\|\hat{\beta}_L - \beta_0 \|_{\infty} \le \| \Theta \|_{\ell_{\infty}}
\left[ \frac{\|X'u \|_{\infty}}{n} + \|\hat{\Sigma} - \Sigma\|_{\infty} \|\hat{\beta}_L - \beta_0 \|_1 + \lambda_n \right].\label{l92}
\end{equation}
Using Lemma \ref{Lasso} (see below) we get that on $\mathcal{A}\cap\mathcal{B}$
\begin{equation}
\|\hat{\beta}_L - \beta_0 \|_{\infty} \le \| \Theta \|_{\ell_{\infty}}
\left[ \frac{\lambda_n}{2} + \left(\frac{\phi_{\Sigma}^2 (s_0)}{32 s_0} \right)  \left( \frac{24 \lambda_n s_0}{\phi_{\Sigma}^2 (s_0)} \right) + \lambda_n \right],\label{l93}
\end{equation}
which provides the result after some simple algebra and upon using that Lemmas \ref{Noise}, \ref{RE1} give the lower bound on the probability of $\mathcal{A}\cap\mathcal{B}$.
\end{proof}

\subsection*{Appendix B}
This appendix provides the proofs of the main theorems.

We state the following result on the Lasso. It is very similar to the classical oracle inequality for the Lasso that assumes subgaussianity of the error terms in \cite{bickelrt2009}. However, it is tailored to our Assumption 1 which only assumes $r$ moments of the covariates and the error terms and hence we still mention it here. Furthermore, the result is needed in order to guide our choice of $\lambda_{prec}$ for the conservative Lasso.

\begin{lemma}\label{Lasso}
Let Assumption 1 be satisfied and set $\lambda_n=M\frac{p^{2/r}}{n^{1/2}}$ for $M>0$. Then, with probability at least $1-\frac{C}{M^{r/2}}-D\frac{p^2s^{r/2}_0}{n^{r/4}}$, the Lasso satisfies the following inequalities
\begin{align}
\| X (\hat{\beta}_L - \beta_0) \|_n^2 \label{IQ1lem1}
\le 18\frac{ \lambda_n^2s_0 }{\phi^2_\Sigma (s_0)},\\
\|\hat{\beta}_L - \beta_0 \|_1 \leq
24\frac{\lambda_ns_0}{\phi^2_\Sigma (s_0)},\label{IQ2lem1} 
\end{align}
for universal constants $C, D>0$. Furthermore, these bounds are valid uniformly over the $\ell_0$-ball $\mathcal{B}_{\ell_0}(s_0)=\cbr[1]{\enVert{\beta_{0}}_{\ell_0}\leq s_0}$.

\end{lemma}

\begin{proof}[Proof of Lemma \ref{Lasso}]
The Lasso corresponds to $\hat{w}_j=1$ for all $j=1,...,p$. Thus, Lemma \ref{A1} combined with the lower bounds on the probabilities of the sets $\mathcal{A}$ and $\mathcal{B}$ from Lemmas \ref{Noise} and \ref{RE1} yields (\ref{IQ1lem1}) and (\ref{IQ2lem1}). The uniformity over $\mathcal{B}_{\ell_0}(s_0)$ follows by noting that the right hand sides of (\ref{IQ1lem1}) and (\ref{IQ2lem1}) only depend on $\beta_0$ through $s_0$.

\end{proof}

\begin{proof}[Proof of Theorem \ref{thm1}] 
The oracle inequalities will follow upon verifying the conditions of Lemma \ref{A1} and showing that $\mathcal{A}\cap\mathcal{B}$ has high probability. As all weights of the conservative Lasso as are less than or equal to one it remains to show that $\min_{j \in S_0^c} \hat{w}_j =1$. To this end Lemma \ref{lemma9} (which uses only Assumption 1) shows that $\max_{j \in S_0^c} | \hat{\beta}_{L,j} |
\le \lambda_{prec}=\frac{9 \lambda_n}{4} \|\Theta\|_{\ell_{\infty}}$ on ${\cal A} \cap {\cal B}$ such that $\min_{j \in S_0^c} \hat{w}_j =1$.
The lower bound on $\mathcal{A}\cap\mathcal{B}$ follows from Lemmas \ref{Noise} and \ref{RE1}. The uniformity over $\mathcal{B}_{\ell_0}(s_0)$ follows by noting that the right hand sides of (\ref{IQ1thm1}) and (\ref{IQ2thm1}) only depend on $\beta_0$ through $s_0$.
\end{proof} 

{\bf $\Theta$'s relation to the regression coefficients}\\
In order to establish a central limit theorem for $\alpha'\hat{\Theta}X'u/n^{1/2}$ in (\ref{stat}) we need to understand the asymptotic properties of $\hat{\Theta}$. To do so we relate $\hat{\Theta}$ to $\Theta:=\Sigma^{-1}$. First, let $\Sigma_{-j,-j}$ represent the $(p-1)\times (p-1)$ submatrix of $\Sigma$ where the $j$th row and column have been removed. $\Sigma_{j,-j}$ is the $j$th row of $\Sigma$ with $j$th element of that row removed. $\Sigma_{-j,j}$ represent the $j$ th column of $\Sigma$ with its $j$th element removed. By Section 2.1 of \cite{yuan2010high} we know that 

\begin{align*}
\Theta_{j,j}
=
\del{\Sigma_{j,j}-\Sigma_{j,-j}\Sigma_{-j,-j}^{-1}\Sigma_{-j,j}}^{-1}
\end{align*}
and
\begin{align*}
\Theta_{j,-j}
=
-\del{\Sigma_{j,j}-\Sigma_{j,-j}\Sigma_{-j,-j}^{-1}\Sigma_{-j,j}}^{-1}\Sigma_{j,-j}\Sigma_{-j,-j}^{-1}
=
-\Theta_{j,j}\Sigma_{j,-j}\Sigma_{-j,-j}^{-1}
\end{align*}
Next, let $X_{j,i}$ denote the $i$th element of $X_j$ and $X_{-j,i}$ the $i$th element of $X_{-j}$ (recall the definition of $X_j$ and $X_{-j}$ just prior to (\ref{NodeLObj})). Now, defining $\gamma_{j}$ as the value of $\gamma$ minimizing, 
\begin{align*}
E\del[1]{X_{j,i}-X_{-j,i}\gamma}^2
\end{align*}
implies that
\begin{align*}
\gamma_{j}'
=
\Sigma_{j,-j}\Sigma_{-j,-j}^{-1}
\end{align*}
such that
\begin{align}
\Theta_{j,-j}=-\Theta_{j,j}\gamma_{j}'\label{nodeaux1}.
\end{align}
Thus, for $\eta_{j,i}:=X_{j,i}-X_{-j,i}\gamma_j$, it follows from the definition of $\gamma_j$ as an $L^2$-projection that all entries of $X_{-j,i}\eta_{j,i}$ have mean zero such that
\begin{align}
X_{j,i}=X_{-j,i}\gamma_{j}+\eta_{j,i}\label{nodeaux}
\end{align}
is a regression model with covariates orthogonal in $L^2$ to the error terms for all $j=1,...,p$ and $i=1,...,n$. Let $\Theta_j$ be the $j$'th row of $\Theta$ written as a column vector. Then the crux is that (\ref{nodeaux}) is sparse if and only if $\Theta_j$ is sparse as can be seen from (\ref{nodeaux1}). Let $S_j=\cbr[1]{k=1,...,p: \Theta_{j,k}\neq 0}$ with cardinality $s_j=|S_j|$ denote the indices of the non-zero terms of $\Theta_j$. Then, the regression model (\ref{nodeaux}) will also be sparse with $\gamma_j$ possessing $s_j$ non-zero entries. Thus, with Theorem \ref{thm1} in mind it is sensible that the estimator $\hat{\gamma}_j$ resulting from (\ref{NodeCLObj}) is close to $\gamma_j$. We make this claim rigorous in Lemma \ref{Node}. Next, by (\ref{nodeaux}), 
\begin{align*}
\Sigma_{j,j}
=
E(X_{j,i}^2)
=
\gamma_j'\Sigma_{-j,-j}\gamma_j+E(\eta_{j,i}^2)
=
\Sigma_{j,-j}\Sigma_{-j,-j}^{-1}\Sigma_{-j,j}+E(\eta_{j,i}^2),
\end{align*}
such that
\begin{align*}
\tau_j^2
:=
E(\eta_{j,i}^2)
=
\Sigma_{-,j}-\Sigma_{j,-j}\Sigma_{-j,-j}^{-1}\Sigma_{-j,j}
=
\frac{1}{\Theta_{j,j}}.
\end{align*}
Thus, defining
\[ C = \left( \begin{array}{cccc}
			1 & -\gamma_{1,2} &  \cdots & -\gamma_{1,p} \\
			-\gamma_{2,1} & 1 & \cdots & -\gamma_{2,p} \\
			\hdots & \hdots & \ddots & \hdots \\
			-\gamma_{p,1}&  -\gamma_{p,2} &  \cdots &  1 \end{array} \right),\]
and $T^2 = diag ( \tau_1^2, \cdots, \tau_p^2)$ we can write $\Theta=T^{-2}C$ using (\ref{nodeaux1}). In Lemma \ref{Node}  we  show that $\hat{\tau}_j^2$ as defined in (\ref{tauhat}) is close to $\tau_j^2$ such that $\hat{\Theta}_j$ is close to $\Theta_j$ when $\hat{\gamma}_j$ is close to $\gamma_j$.

\noindent{\bf Remark:} The above arguments have relied on $X_i$ being i.i.d. such that $\Sigma=E\del[1]{X_iX_i'}$ is constant and does not depend on $i=1,...,n$. At the cost of more involved notation and proofs the arguments above would also be valid in the case of non-identically distributed covariates if we consider $\Sigma=\frac{1}{n}\sum_{i=1}^nE\del[1]{X_iX_i'}$ instead of $E(X_1X_1')$. However, we shall not pursue this generalization here. 

We can now state the asymptotic properties of $\hat{\Theta}$.
\begin{lemma}\label{Node}
Let Assumptions 1 and 2 be satisfied and set $\lambda_{node,n}\asymp \frac{h^{2/r}p^{2/r}}{n^{1/2}}$. Then,
\begin{align}
\max_{j\in H}\| X_{-j} (\hat{\gamma}_j - \gamma_j) \|_n^2 
&=
O_p\del[2]{\frac{ d_{n1}\bar{s}h^{4/r}p^{4/r}}{n}\label{NodeEq3}}.\\
\max_{j\in H}\|\hat{\gamma}_j - \gamma_j \|_1 
&= 
O_p\del[2]{\frac{d_{n2} \bar{s}h^{2/r}p^{2/r}}{n^{1/2}}}\label{NodeEq4}. \\
\max_{j\in H}\envert[0]{\hat{\tau}_j^2-\tau_j^2}
&=
O_p \del[2]{\bar{s}^{1/2} \frac{h^{2/r}p^{2/r}}{\sqrt{n}}}.\\
\max_{j\in H}\enVert[1]{\hat{\Theta}_j - \Theta_j}_1 
&=
O_p \del[2]{ d_{n2}\bar{s} \frac{h^{2/r}p^{2/r}}{\sqrt{n}}}.\label{l1theta}\\
\max_{j\in H}\|\hat{\Theta}_j - \Theta_j \|_2
&=
O_p \del[2]{\sqrt{d_{n1}}\bar{s}^{1/2} \frac{h^{2/r}p^{2/r}}{\sqrt{n}}}.\label{l2theta}\\
\max_{j\in H}\enVert[0]{\hat{\Theta}_j}_1
&=
O_p(\bar{s}^{1/2}).\label{thetahatnorm}
\end{align}
\end{lemma}
\textbf{Remark}. Clearly we see that divergences $d_{n1}$ and $d_{n2}$ between the Lasso and the conservative Lasso influence the upper bounds in the nodewise regressions. The roles of $d_{n1}$ and $d_{n2}$ are explained in detail in Remark 3 after Theorem \ref{thm1}. Clearly we see that the conservative nodewise regression Lasso can have smaller errors in prediction norm, $\ell_1$ and $\ell_2$ errors for estimates than the its Lasso counterpart since $d_{n1}=18$ for the Lasso and as low as nearly 2 for the former. Furthermore, $d_{n2}$ is 24 in the Lasso nodewise regression and as small as almost 4 in conservative Lasso nodewise regression as also explained in the Remarks to Theorem \ref{thm1}.

Lemma \ref{Node} is an auxiliary lemma which will be of great importance in the proof of Theorem \ref{thm2} below. Note that all bounds provided are uniform in $H$ with upper bounds tending to zero even when $h=|H|\to \infty$ as long as this does not happen too fast. (\ref{NodeEq3}) and (\ref{NodeEq4}) reduce to inequalities of the type (\ref{IQ1thm1}) and (\ref{IQ2thm1}) in Theorem \ref{thm1} when $H$ is a singleton such that $h=1$. Note also that (\ref{l1theta}) can be used to bound the estimation error of each row of $\hat{\Theta}$ for the corresponding row of $\Theta$. Thus, choosing $H=\cbr[1]{1,...,p}$, (\ref{l1theta}) provides a bound on $\enVert[0]{\hat{\Theta}-\Theta}_{\ell_\infty}$. Finally, we remark that the uniformity of the above results is crucial for establishing the limiting distribution of $\alpha'\hat{\Theta}X'u/n^{1/2}$ in (\ref{stat}) as well as for estimating the variance of the limiting distribution. 

\begin{proof}[Proof of Lemma \ref{Node}]
We start by establishing the order of magnitude of $\enVert[0]{X_{-j}(\hat{\gamma}_j-\gamma_j)}_n^2$ and $\enVert[0]{\hat{\gamma}_j-\gamma_j}_1$. For concreteness, consider nodewise regression $j$. Define
\begin{align*}
\mathcal{A}_{node}=\cbr[2]{\max_{j\in H}\enVert[0]{X'_{-j}\eta_j}_\infty\leq \lambda_{node,n}/2}\text{ and }\mathcal{B}_{j}=\cbr[2]{\phi^2_{\hat{\Sigma}_{-j}}(s_j)\geq \phi^2_{\Sigma_{-j}}(s_j)/2}. 
\end{align*}
By an exact adaptation of the proof of Lemma \ref{A1} it can be shown for each $j\in H$ that with definition of $d_{n1} = 2 (2 a_n+1)^2$, and $d_{n2}= 4 (a_n+1)(2 a_n +1)$, $0 < a_n \le 1$
\begin{align}
\| X_{-j} (\hat{\gamma}_j - \gamma_j) \|_n^2 
\le 
d_{n1}
\frac{ \lambda_{node,n}^2s_j }{\phi^2_\Sigma (s_j)},\label{IQ1node}\\
\|\hat{\gamma}_j - \gamma_j \|_1 
\le 
d_{n2}\frac{\lambda_{node,n}s_j}{\phi^2_\Sigma (s_j)}\label{IQ2node} 
\end{align}
are valid on the set $\mathcal{A}_{node}\cap \mathcal{B}_j$ for $j \in H$. 


Note that (\ref{IQ1node}) and (\ref{IQ2node}) are valid simultaneously for all $j \in H$ on $\mathcal{A}_{node}\cap \del[0]{\cap_{j \in H}\mathcal{B}_j}$ \footnote{It will turn out later that it is quite important that (\ref{IQ1node}) and (\ref{IQ2node}) are valid simultaneously for all $j\in H$ since this will give us a vital uniformity when bounding $\hat{\tau}_j^2$ away from 0. If one is only interested in one nodewise regression the outer maximum in the definition of $\mathcal{A}_{node}$ can be omitted.}. Thus, we establish a lower bound on the probability of this set. First, consider $\mathcal{A}_{node}$. Since $\eta_{j,i}$ is the residual from the $L^2$-projection of $X_{j,i}$ on the linear span of the elements of $X_{-j,i}$ it follows that $E(X_{-j,i}\eta_{j,i})=0$ for all $i=1,...,n$ and all $j\in H$. Furthermore, by the Cauchy-Schwarz inequality, every entry of $X_{-j,i}\eta_{j,i}$ has bounded $r/2$-norm via Assumption 2c. The maximum in the definition of $\mathcal{A}_{node}$ is over $h(p-1)$ terms. Thus, merging constants and choosing $\lambda_{node,n}=M\frac{h^{2/r}p^{2/r}}{\sqrt{n}}$ for some $M>0$, Lemma \ref{conc} yields,
\begin{align*}
P(\mathcal{A}_{node}^c)
=
P\del[2]{\max_{j\in H}\enVert[0]{X_{-j}'\eta_j}_\infty > n\lambda_{node,n}/2}
\leq
hp\frac{b_rC^2n^{r/4}}{(n\lambda_{node,n}/2)^{r/2}}
=
\frac{C}{M^{r/2}}, 
\end{align*}
which also shows that
\begin{align}
\max_{j\in H}\enVert[0]{X_{-j}'\eta_j/n}_\infty 
=
O_p\del[1]{\lambda_{node,n}}
=
O_p\del[2]{\frac{h^{2/r}p^{2/r}}{\sqrt{n}}}\label{NodeNoise}
\end{align}
by choosing $M$ sufficiently large. 

Next, we provide a lower bound on the probability of the set $\cap_{j\in H}\mathcal{B}_j$. We know by Lemma \ref{vdGB} that $\cbr[2]{\enVert[0]{\hat{\Sigma}_{-j}-\Sigma_{-j}}_\infty\leq  \frac{\phi^2_{\Sigma_{-j}}(s_j)}{32s_j}}\subseteq\cbr[1]{\phi^2_{\hat{\Sigma}_{-j}}(s_j)\geq \phi^2_{\Sigma_{-j}}(s_j)/2}=\mathcal{B}_j$. Thus, the relation
\begin{align*}
\enVert[0]{\hat{\Sigma}_{-j}-\Sigma_{-j}}_\infty
\leq
\enVert[0]{\hat{\Sigma}-\Sigma}_\infty 
\leq
\frac{\phi^2_\Sigma(\bar{s})}{32\bar{s}}
\leq
 \frac{\phi^2_{\Sigma_{-j}}(s_j)}{32s_j}
\end{align*}
implies that $\cbr[1]{\enVert[0]{\hat{\Sigma}-\Sigma}_\infty 
\leq
\frac{\phi^2_\Sigma(\bar{s})}{32\bar{s}}}\subseteq\mathcal{B}_j$ for all $j\in H$ and therefore $\cbr[1]{\enVert[0]{\hat{\Sigma}-\Sigma}_\infty 
\leq
\frac{\phi^2_\Sigma(\bar{s})}{32\bar{s}}}\subseteq \cap_{j\in H}\mathcal{B}_j$.

Next, by arguments exactly parallel to those in Lemma \ref{RE1}, it follows that
\begin{align*}
P\del[2]{\del[1]{\cap_{j\in H}\mathcal{B}_j}^c}
\leq
P\del[2]{\enVert[0]{\hat{\Sigma}-\Sigma}_\infty > \frac{\phi^2_\Sigma(\bar{s})}{32\bar{s}}}
\leq
D\frac{p^2\bar{s}^{r/2}}{n^{r/4}}.
\end{align*}
Hence, with probability at least $1-\frac{C}{M^{r/2}}-D\frac{p^2\bar{s}^{r/2}}{n^{r/4}}$
\begin{align}
\| X_{-j} (\hat{\gamma}_j - \gamma_j) \|_n^2 
&\le 
d_{n1} \frac{ \lambda_{node,n}^2s_j}{\phi^2_\Sigma (s_j)}.\label{n1}\\
\|\hat{\gamma}_j - \gamma_j \|_1 
&\le
d_{n2} \frac{\lambda_{node,n}s_j}{\phi^2_\Sigma (s_j)}.\label{n2} 
\end{align}
By choosing $M$ sufficiently large, using $\frac{p^2\bar{s}^{r/2}}{n^{r/4}}\to 0$, and inserting the definition of $\lambda_{node,n}$ (\ref{NodeEq3}) and (\ref{NodeEq4}) follow upon taking the maximum in the above display and utilizing that the above inequalities are all valid simultaneously on $\mathcal{A}_{node,n}\cap\del[1]{\cap_{j\in H}\mathcal{B}_j}$.

We shall also need an upper bound on $\max_{j\in H}\enVert[0]{\hat{\gamma}_j-\gamma_j}_2$ in the proof of Theorem \ref{thm2}. Let $\hat{v}_j$ and $v_j$ be $p\times 1$ vectors containing $0$ in the $j$'th position and the elements of $\hat{\gamma}_j$ and $\gamma_j$, respectively, in the remaining positions in the same order as they appear in $\hat{\gamma}_j$ and $\gamma_j$. Thus, $\max_{j\in H}\enVert[0]{\hat{\gamma}_j-\gamma_j}_2= \max_{j \in H} \enVert[0]{\hat{v}_j - v_j}_2$. Thus,
\[ | (\hat{v}_j - v_j)' \hat{\Sigma} (\hat{v}_j - v_j) - (\hat{v}_j - v_j)' \Sigma (\hat{v}_j - v_j) |
\le 
\|\hat{\Sigma} - \Sigma\|_{\infty} \|\hat{v}_j - v_j \|_1^2\]
such that 
\begin{align}
\max_{j\in H}(\hat{v}_j - v_j)' \Sigma (\hat{v}_j - v_j)
\leq
\max_{j\in H}(\hat{v}_j - v_j)' \hat{\Sigma} (\hat{v}_j - v_j) + \max_{j\in H}\|\hat{\Sigma} - \Sigma\|_{\infty} \|\hat{v}_j - v_j \|_1^2.\label{l2}
\end{align}
Next, we bound each term on the right hand side of the above display. First,
\begin{align*}
\max_{j\in H}(\hat{v}_j - v_j)' \hat{\Sigma} (\hat{v}_j - v_j)
=
\max_{j\in H}\enVert[1]{X(\hat{v}_j-v_j)}_n^2
=
\max_{j\in H}\enVert[1]{X_{-j}(\hat{\gamma}_j-\gamma_j)}_n^2
=
O_p\del[3]{\frac{d_{n1} \bar{s}h^{4/r}p^{4/r}}{n}},
\end{align*}
by (\ref{NodeEq3}). Next, consider the second term in (\ref{l2}). To this end, apply Lemma \ref{conc} and Assumption 1,  for any $t>0$ to get 
\begin{align*}
P\del[2]{\|\hat{\Sigma} - \Sigma\|_{\infty}>t}
=
P\del[3]{\max_{1\leq k,l\leq p}\envert[2]{\frac{1}{n}\sum_{i=1}^n\del[1]{X_{k,i}X_{l,i}-E(X_{k,i}X_{l,i})}}>t}
\leq 
b_{r/2}\frac{p^2n^{r/4}C}{(tn)^{r/2}}.
\end{align*}
Thus, choosing $t=M\frac{p^{4/r}}{n^{1/2}}$ for $M>0$ sufficiently large yields
\begin{equation}
\|\hat{\Sigma} - \Sigma\|_{\infty}
=
O_p\del[3]{\frac{p^{4/r}}{n^{1/2}}}.\label{shat}
\end{equation}
In combination with (\ref{NodeEq4}) this implies (using $\enVert[0]{\hat{\gamma}_j-\gamma_j}_1=\enVert[0]{\hat{v}_j - v_j}_1$)
\begin{align*}
\max_{j\in H}\|\hat{\Sigma} - \Sigma\|_{\infty} \|\hat{v}_j - v_j \|_1^2
=
O_p\del[3]{\frac{p^{4/r}}{n^{1/2}}}O_p\del[3]{\frac{d_{n2}^2 \bar{s}^2h^{4/r}p^{4/r}}{n}}
=
O_p\del[3]{\frac{d_{n2}^2 \bar{s}^2h^{4/r}p^{8/r}}{n^{3/2}}}.
\end{align*}
But since $d_{n2}^2$ is bounded by constants
\begin{align*}
O_p\del[3]{\frac{d_{n2}^2 \bar{s}^2h^{4/r}p^{8/r}}{n^{3/2}}}
=
O_p\del[3]{d_{n2}^2 \frac{\bar{s}p^{4/r}}{n^{1/2}}\frac{\bar{s}h^{4/r}p^{4/r}}{n}}
=
o_p\del[3]{\frac{\bar{s}h^{4/r}p^{4/r}}{n}},
\end{align*}
as $\frac{\bar{s}p^{4/r}}{n^{1/2}}=\del[2]{\frac{p^2 \bar{s}^{r/2}}{n^{r/4}}}^{2/r}\to 0$ by Assumption 2b) we conclude
\begin{align*}
\max_{j\in H}(\hat{v}_j - v_j)' \Sigma (\hat{v}_j - v_j)
\le
O_p\del[3]{\frac{d_{n1} \bar{s}h^{4/r}p^{4/r}}{n}}.
\end{align*}
Therefore, by
\begin{align*}
\max_{j\in H}\phi_{\min}(\Sigma)\enVert[0]{\hat{v}_j-v_j}_2^2
\leq
\max_{j\in H}(\hat{v}_j - v_j)' \Sigma (\hat{v}_j - v_j)
\leq
O_p\del[3]{d_{n1} \frac{\bar{s}h^{4/r}p^{4/r}}{n}},
\end{align*}
one gets

\begin{equation}
\max_{j\in H}\enVert[0]{\hat{\gamma}_j-\gamma_j}_2^2
=
\max_{j\in H}\enVert[0]{\hat{v}_j-v_j}_2^2
=
O_p\del[3]{d_{n1} \frac{\bar{s}h^{4/r}p^{4/r}}{n}}.\label{eqg2}
\end{equation}
since $\phi_{\min}(\Sigma)$ is bounded away from zero by Assumption 2a).

Next, we consider $\envert[0]{\hat{\tau}_j^2-\tau_j^2}$. First, by (\ref{2.5}) and $X_j = X_{-j} \gamma_j + \eta_j$,
\begin{eqnarray*}
\hat{\tau}_j^2 & = & \frac{(X_j - X_{-j} \hat{\gamma}_j)' X_j}{n} \\
& = & \frac{[\eta_j - X_{-j} (\hat{\gamma}_j - \gamma_j)]'[X_{-j} \gamma_j + \eta_j]}{n} \\
& = & \frac{\eta_j' \eta_j}{n} + \frac{\eta_j' X_{-j} \gamma_j}{n} - \frac{(\hat{\gamma}_j - \gamma_j)' X_{-j}' X_{-j} \gamma_j}{n} - \frac{(\hat{\gamma}_j - \gamma_j )' X_{-j}' \eta_j}{n}.
\end{eqnarray*}
Using the above expression one gets
\begin{align}
\max_{j\in H} |\hat{\tau}_j^2 - \tau_j^2|
& \le 
\max_{j\in H}\envert[2]{\frac{\eta_j' \eta_j}{n} - \tau_j^2} + \max_{j\in H}|\eta_j' X_{-j} (\hat{\gamma}_j - \gamma_j)/n|\notag\\
 &+ \max_{j\in H}|\eta_j' X_{-j} \gamma_j/n| + \max_{j\in H}\envert[3]{\frac{\gamma_j' X_{-j}' X_{-j} (\hat{\gamma}_j - \gamma_j)}{n}}\label{fourterms}.
\end{align}
Since $\frac{\eta_j' \eta_j}{n} - \tau_j^2=\frac{1}{n}\sum_{i=1}^n\del[1]{\eta_{j,i}^2-E(\eta^2_{j,i})}$ is a sum of mean zero terms with $r/2$ moments uniformly bounded by a constant $C$ (the latter is seen by means of the Cauchy-Schwarz inequality and Assumption 2c) it follows from Lemma \ref{conc}  
\begin{align*}
P\del[3]{\max_{j\in H}\envert[2]{\frac{\eta_j' \eta_j}{n} - \tau_j^2}>Mh^{2/r}/n^{1/2}}
=
P\del[2]{\max_{j\in H}\envert[1]{\frac{1}{n}\sum_{i=1}^n\del[1]{\eta_{j,i}^2-E(\eta^2_{j,i})}}>Mh^{2/r}/n^{1/2}}
\leq
\frac{b_rC}{M^{r/2}},
\end{align*}
which implies that
\begin{equation}
\max_{j\in H}\envert[2]{\frac{\eta_j' \eta_j}{n} - \tau_j^2}=O_p\del[2]{\frac{h^{2/r}}{n^{1/2}}}.\label{l21}
\end{equation}
Next, consider the second term in (\ref{fourterms}). By (\ref{NodeEq4}) and (\ref{NodeNoise}) it follows that
\begin{eqnarray}
\max_{j\in H}|\eta_j' X_{-j} (\hat{\gamma}_j - \gamma_j)/n| 
&\le &
\max_{j\in H}\| \eta_j' X_{-j}/n \|_{\infty} \max_{j\in H}\|\hat{\gamma}_j - \gamma_j \|_1 \nonumber \\
&=&
O_p\del[3]{\frac{h^{2/r}p^{2/r}}{\sqrt{n}}}O_p\del[3]{\frac{d_{n2} \bar{s}h^{2/r}p^{2/r}}{\sqrt{n}}} \nonumber \\
&=&
O_p\del[4]{\sbr[3]{\sqrt{d_{n2}} \bar{s}^{1/2}\frac{h^{2/r}p^{2/r}}{\sqrt{n}}}^2}.\label{l22}
\end{eqnarray}
Before we bound the third term in (\ref{fourterms}) we show that $\max_{j\in H}\enVert[0]{\gamma_j}_1=O(\sqrt{\bar{s}})$. To this end, define the $(p-1)\times (p-1)$ matrix $\Sigma_{-j}$ consisting of all rows and columns of $\Sigma$ except the $j$'th row and column. Then, note that
\begin{align*}
\frac{\gamma_j'\Sigma_{-j}\gamma_j}{\gamma_j'\gamma_j}
\geq
\phi_{\min}(\Sigma_{-j})
\geq 
\phi_{\min}(\Sigma),
\end{align*}
such that
\begin{align*}
\gamma_{j}'\gamma_{j}
\leq
\frac{\gamma_{j}'\Sigma_{-j,-j}\gamma_{j}}{\phi_{\min}(\Sigma)}.
\end{align*}
Since $X_{j,i}=X_{-j,i}\gamma_j+\eta_{j,i}$ it follows from the orthogonality in $L^2$ of each entry in $X_{-j,i}$ to $\eta_{j,i}$ that $E(X_{j,i}^2)=\gamma_j'\Sigma_{-j}\gamma_j+E(\eta_{j,i}^2)$ such that $\gamma_{j}'\Sigma_{-j}\gamma_{j}\leq E(X_{j,i}^2)\leq \max_{j\in H}E(X_{j,i}^2)$. Since $\del[1]{E(X_{j,i}^2)}^{1/2}\leq \del[1]{E(X_{j,i}^r)}^{1/r}\leq C^{1/r}$ for all $j\in H$ one has $\max_{j\in H}E(X_{j,i}^2)\leq C^{2/r}$. Hence, 
\begin{equation}
\gamma_{j}'\gamma_{j}\leq \frac{C^{2/r}}{\phi_{\min}(\Sigma)}.\label{eqgj2}
\end{equation} 
Thus, by Assumption 2a), $\gamma_{j}'\gamma_{j}$ is bounded by a constant not depending on $j$ which implies that $\max_{j\in H}\enVert[0]{\gamma_j} _1=O(\sqrt{\bar{s}})$. Hence, returning to the third term of (\ref{fourterms}),
\begin{equation}
\max_{j\in H}|\eta_j' X_{-j} \gamma_j/n| 
\le
\max_{j\in H} \|\eta_j' X_{-j}/n \|_{\infty} \max_{j\in H}\|\gamma_j \|_1 
= 
O_p\del[2]{\sqrt{\bar{s}}\frac{h^{2/r}p^{2/r}}{\sqrt{n}}},\label{l23}
\end{equation}
where we have also used (\ref{NodeNoise}). It remains to bound the fourth summand in (\ref{fourterms}). By the Karush-Kuhn-Tucker conditions for the conservative lasso nodewise regression one has
\begin{align*}
\lambda_{node,n} \hat{\Gamma}_j \hat{\kappa}_j + \frac{ X_{-j}' X_{-j} \hat{\gamma}_j}{n} - \frac{X_{-j}' X_{j} }{n} =0, 
\end{align*} 
which, using $X_j = X_{-j} \gamma_j + \eta_j$, is equivalent to
\begin{align*}
\lambda_{node,n} \hat{\Gamma}_j \hat{\kappa}_j + \frac{X_{-j}'X_{-j}  \hat{\gamma}_j }{n} - \frac{X_{-j}' \eta_j}{n} - \frac{X_{-j}' X_{-j} \gamma_j}{n} =0.
\end{align*} 
The above equation can be rewritten as
\[ \frac{X_{-j}' X_{-j}}{n} (\hat{\gamma}_j - \gamma_j) = \frac{X_{-j}' \eta_j}{n} - \lambda_{node,n} \hat{\Gamma}_j \hat{\kappa}_j.\]
This implies
\begin{align*}
\enVert[3]{ \frac{X_{-j}' X_{-j}}{n} (\hat{\gamma}_j - \gamma_j)}_{\infty} 
\le 
\enVert[3]{\frac{X_{-j}' \eta_j}{n}}_{\infty} + \|  \lambda_{node,n} \hat{\Gamma}_j \hat{\kappa}_j\|_{\infty} .
\end{align*}
The second term on the right hand side in the above display can be bounded as
\begin{align*}
\|  \lambda_{node,n} \hat{\Gamma}_j \hat{\kappa}_j\|_{\infty} 
\le 
\|  \lambda_{node,n} \hat{\Gamma}_j\|_{\ell_\infty}  \|\hat{\kappa}_j\|_{\infty} 
\leq
\lambda_{node,n},
\end{align*}
for all $j\in H$ since $\|\hat{\kappa}_j \|_{\infty} \le 1$ and $\| \hat{\Gamma}_j\|_{\ell_\infty}\leq 1$. Hence, using (\ref{NodeNoise}),
\begin{align*}
\max_{j\in H}\enVert[3]{\frac{X_{-j}' X_{-j}}{n} (\hat{\gamma}_j - \gamma_j)}_{\infty}
=
O_p(\lambda_{node,n})+O_p(\lambda_{node,n})
=
O_p\del[2]{\frac{h^{2/r}p^{2/r}}{\sqrt{n}}}
\end{align*}
This means, using $\max_{j\in H}\|\gamma_j \|_1 = O(\bar{s}^{1/2})$,
\begin{equation}
\max_{j\in H}\envert[2]{\gamma_j' \frac{X_{-j}' X_{-j}}{n} (\hat{\gamma}_j - \gamma_j ) } 
= 
O_p \del[2]{\bar{s}^{1/2} \frac{h^{2/r}p^{2/r}}{\sqrt{n}}}.\label{l24}
\end{equation}
Since $h\leq p$, Assumption 2b) implies that
\[\bar{s}^{1/2} \frac{h^{2/r}p^{2/r}}{\sqrt{n}}\leq\bar{s}^{1/2} \frac{p^{4/r}}{\sqrt{n}}= \frac{1}{\bar{s}^{1/2}}\del[3]{\frac{\bar{s}^{r/2} p^2}{n^{r/4}}}^{2/r}\to 0,\]
such that the dominant term in (\ref{fourterms}) is $O_p \del[2]{\bar{s}^{1/2} \frac{h^{2/r}p^{2/r}}{\sqrt{n}}}$ given $d_{n2}$. Thus, 
\[\max_{j\in H}\envert[0]{\hat{\tau}_j^2-\tau_j^2}=O_p (\bar{s}^{1/2} \frac{h^{2/r}p^{2/r}}{n^{1/2}}).\] 
Next, note that $\tau_j^2 = 1/\Theta_{j,j}\geq 1/\phi_{\max}(\Theta)=\phi_{\min}(\Sigma)$ for all $j=1,...,p$ with $\phi_{\min}(\Sigma)$ bounded away from zero by Assumption 2. Thus, $\min_{1\leq j\leq p}\tau_j^2$ is bounded away from zero, and so 
\begin{align*}
\min_{1\le j\le p} \hat{\tau}_j^2 = \min_{1 \le j \le p} [\hat{\tau}_j^2 - \tau_j^2 + \tau_j^2] 
\ge \min_{1\le j \le p} \tau_j^2 - \max_{1 \le j \le p} |\hat{\tau}_j^2 - \tau_j^2|
\end{align*}
is bounded away from zero with probability tending to one using $\max_{j\in H}\envert[0]{\hat{\tau}_j^2-\tau_j^2}=O_p \del[2]{\bar{s}^{1/2} \frac{h^{2/r}p^{2/r}}{\sqrt{n}}}=o_p(1)$. This implies
\begin{align}
\max_{j\in H}\envert[3]{\frac{1}{\hat{\tau}_j^2 } - \frac{1}{\tau_j^2}} 
=
\max_{j\in H}\frac{|\tau_j^2 - \hat{\tau}_j^2|}{\hat{\tau}_j^2 \tau_j^2} 
=
O_p \del[2]{\bar{s}^{1/2} \frac{h^{2/r}p^{2/r}}{\sqrt{n}}}.\label{pt32.3i}
\end{align}

We are now ready to bound $\max_{j\in H}\| \hat{\Theta}_j - \Theta_j \|_1$. Recall that $\hat{\Theta}_j$ is formed by dividing $\hat{C}_j$ by $\hat{\tau}_j^2$.  Let $\Theta_j$ denote the $j$'th row of $\Theta$ written as a column vector. Then, $\Theta_j$ is formed by dividing $C_j$ ($j$'th row of $C$ written as a column vector) by $\tau_j^2$. Therefore, using $\max_{j\in H}\| \gamma_j \|_1 = O(\bar{s}^{1/2})$, (\ref{NodeEq4}), and (\ref{pt32.3i})
\begin{align}
\max_{j\in H}\enVert[1]{\hat{\Theta}_j - \Theta_j}_1 
&=
\max_{j\in H}\enVert[3]{\frac{\hat{C}_j}{\hat{\tau}_j^2} - \frac{C_j}{\tau_j^2}}_1\\
&\le
\max_{j\in H}\envert[3]{\frac{1}{\hat{\tau}_j^2} - \frac{1}{\tau_j^2}}+ \max_{j\in H}\enVert[3]{\frac{\hat{\gamma}_j}{\hat{\tau}_j^2} - \frac{\gamma_j}{\tau_j^2}}_1 \nonumber \\
& =
\max_{j\in H}\envert[3]{\frac{1}{\hat{\tau}_j^2} - \frac{1}{\tau_j^2}} + \max_{j\in H}\enVert[3]{\frac{\hat{\gamma}_j}{\hat{\tau}_j^2} - \frac{\gamma_j}{\hat{\tau}_j^2} + \frac{\gamma_j}{\hat{\tau}_j^2} -\frac{\gamma_j}{\tau_j^2} }_1 \nonumber \\
& 
\leq
\max_{j\in H}\envert[3]{\frac{1}{\hat{\tau}_j^2} - \frac{1}{\tau_j^2}} + \max_{j\in H}  \frac{\|\hat{\gamma}_j - \gamma_j \|_1}{\hat{\tau}_j^2} + \max_{j\in H}
\|\gamma_j \|_1 \max_{j \in H} \left(\envert[3]{ \frac{1}{\hat{\tau}_j^2} - \frac{1}{\tau_j^2} }\right) \nonumber \\
& = 
 O_p \del[2]{\bar{s}^{1/2} \frac{h^{2/r}p^{2/r}}{\sqrt{n}}}+O_p\del[2]{\frac{d_{n2} \bar{s}h^{2/r}p^{2/r}}{\sqrt{n}}} +O_p \del[2]{\bar{s} \frac{h^{2/r}p^{2/r}}{\sqrt{n}}} \nonumber \\
& =
O_p\del[2]{\frac{ d_{n2}\bar{s}h^{2/r}p^{2/r}}{\sqrt{n}}}.\label{pt32c.1i}
\end{align}
Next, for later purposes, we also bound $\|\hat{\Theta}_j - \Theta_j \|_2$. By (\ref{eqg2}), and $\max_{j \in H} \|\gamma_j \|_2^2 = O(1)$ by (\ref{eqgj2})
\begin{align}
\max_{j\in H}\| \hat{\Theta}_j - \Theta_j \|_2 
&\le
 \max_{j\in H}\envert[3]{\frac{1}{\hat{\tau}_j^2} - \frac{1}{\tau_j^2}} + \max_{j\in H}\frac{ \| \hat{\gamma}_j - \gamma_j \|_2}{\hat{\tau}_j^2}
+ \max_{j\in H}\| \gamma_j \|_2  \max_{j \in H} \del[3]{\envert[2]{\frac{1}{\hat{\tau}_j^2} - \frac{1}{\tau_j^2}}} \nonumber \\
& =  
O_p \del[2]{\bar{s}^{1/2} \frac{h^{2/r}p^{2/r}}{\sqrt{n}}}+ O_p\del[3]{\frac{\sqrt{d_{n1}} \bar{s}^{1/2}h^{2/r}p^{2/r}}{n^{1/2}}} + O_p \del[2]{\bar{s}^{1/2} \frac{h^{2/r}p^{2/r}}{\sqrt{n}}},\nonumber \\
& =
O_p \del[2]{\sqrt{d_{n1}} \bar{s}^{1/2} \frac{h^{2/r}p^{2/r}}{\sqrt{n}}}.\label{pt32d.1i}
\end{align}
Finally, we show that $\max_{j\in H}\| \hat{\Theta}_j \|_1=O_p(\sqrt{\bar{s}})$. To this end, 
\begin{align}
\max_{j\in H}\enVert[0]{\Theta_j}_1\leq\max_{j\in H}\frac{1}{\tau_j^2}+\max_{j\in H}\enVert[0]{\gamma_j/\tau_j^2}_1
=
O(\bar{s}^{1/2})\label{thetaj}
\end{align}
(as $\tau_j^2$ is uniformly bounded away from zero). Then, as $h\leq p$ implies $\frac{\bar{s} h^{2/r} p^{2/r}}{n^{1/2}} \leq [ p^2 \bar{s}^{r/2}/n^{r/4}]^{2/r}\to 0$ by Assumption 2b, we get
\begin{align} 
\max_{j\in H}\| \hat{\Theta}_j \|_1 
\le \max_{j\in H}\| \hat{\Theta}_j - \Theta_j \|_1 + \max_{j\in H}\| \Theta_j \|_1
=
O_p\del[2]{\frac{d_{n2} \bar{s}h^{2/r}p^{2/r}}{n^{1/2}}}+O(\sqrt{\bar{s}})
=
O_p(\sqrt{\bar{s}})\label{sbar}.
\end{align}
\end{proof}

\begin{proof}[Proof of Theorem \ref{thm2}]
We show that the ratio
\begin{equation}
t = \frac{n^{1/2} \alpha'(\hat{b}- \beta_{0})}{\sqrt{\alpha'\hat{\Theta} \hat{\Sigma}_{xu} \hat{\Theta}'\alpha}},\label{t.1}
\end{equation}
is asymptotically standard normal. First, note that one can write. By (\ref{stat})
\[ t = t_1 + t_2,\]
where
\begin{align*}
t_1 = \frac{\alpha'\hat{\Theta} X' u /n^{1/2}}{\sqrt{\alpha'\hat{\Theta} \hat{\Sigma}_{xu} \hat{\Theta}'\alpha}} \text{ and }
t_2 = -\frac{ \alpha'\Delta}{\sqrt{\alpha'\hat{\Theta} \hat{\Sigma}_{xu} \hat{\Theta}'\alpha}}.
\end{align*}
It suffices to show that $t_1$ is asymptotically standard normal and $t_2=o_p(1)$.

\noindent{\bf Step 1}. We first show that $t_1$ is asymptotically standard normal. 

{\bf a)} To show that $t_1$ is asymptotically standard normal we first show that
\[ t_1'=\frac{\alpha'\Theta X'u/n^{1/2}}{\sqrt{\alpha'\Theta\Sigma_{xu}\Theta'\alpha}}\]
converges in distribution to a standard normal where $\Sigma_{xu} = n^{-1} \sum_{i=1}^nE(X_i X_i' u_i^2)$. Then we show that $t_1'$ and $t_1$ are asymptotically equivalent. Note that, using $E(u_i|X_i)=0$ for all $i=1,...,n$, we obtain
\begin{equation}
E \left[ \frac{\alpha'\Theta X'u/n^{1/2}}{\sqrt{\alpha'\Theta\Sigma_{xu}\Theta'\alpha}} \right] 
=
E \left[ \frac{\alpha'\Theta \sum_{i=1}^nX_iu_i/n^{1/2}}{\sqrt{\alpha'\Theta\Sigma_{xu}\Theta'\alpha}} \right] 
=
0,\label{4.0}
\end{equation}
and
\begin{align*}
E \left[ \frac{\alpha'\Theta X'u/n^{1/2}}{\sqrt{\alpha'\Theta\Sigma_{xu}\Theta'\alpha}}  \right] ^2
=
E \left[ \frac{\alpha'\Theta \sum_{i=1}^nX_iu_i/n^{1/2}}{\sqrt{\alpha'\Theta\Sigma_{xu}\Theta'\alpha}} \right]^2 
=
1.
\end{align*}
Hence, in order to apply Lyapounov's condition in  central limit theorem for independent random variables, it suffices to show that
\begin{align}
\frac{1}{\del[1]{\alpha'\Theta\Sigma_{xu}\Theta'\alpha}^{r/4}}\sum_{i=1}^nE\envert[1]{\alpha'\Theta X_iu_i/n^{1/2}}^{r/2}\to 0\label{Lyap}.
\end{align}
First, using the symmetry of $\Theta$, we get (recall that $\Theta_j$ is the $j$'th row of $\Theta$ written as a column vector)
\begin{align*}
\enVert[1]{\alpha'\Theta}_1
=
\enVert[1]{\Theta\alpha}_1
=
\enVert[4]{\sum_{j\in H}\Theta_j\alpha_j}_1
\leq
\sum_{j\in H}|\alpha_j|\enVert[1]{\Theta_j}_1
=
O\del[2]{\sqrt{h\bar{s}}},
\end{align*}
since $\enVert[0]{\alpha}_2=1$ and $\max_{j\in H}\enVert[0]{\Theta_j}_1=O(\sqrt{\bar{s}})$ by (\ref{thetaj}). Note also that
\begin{align*}
\alpha'\Theta
=
\del[1]{\Theta\alpha}'
=
\del[4]{\sum_{j\in H}\Theta_j\alpha_j}'
\end{align*}
such that the non-zero entries of $\alpha'\Theta$ must be contained in $\bar{S}=\cup_{j\in H}S_j$ which has cardinality at most $|\bar{S}|=h\bar{s}\wedge p$, where $S_j=\cbr[0]{\Theta_{j,i}\neq 0}$. Thus,
\begin{align*}
E\envert[1]{\alpha'\Theta X_iu_i/n^{1/2}}^{r/2}
&\leq
E\del[3]{\enVert[1]{\alpha'\Theta}_1^{r/2}\max_{k\in \bar{S}}\envert[1]{X_{k,i}u_i/n^{1/2}}^{r/2}}\\
&\leq
O\del{\del{\frac{h\bar{s}}{n}}^{r/4}}\del[1]{h\bar{s}\wedge p}\max_{k\in \bar{S}}E|X_{k,i}u_i|^{r/2}\\
&\leq
O\del[4]{\del{\frac{h\bar{s}}{n}}^{r/4}\del[1]{h\bar{s}\wedge p}}\\
&=
O\del{\frac{(h\bar{s})^{r/4+1}\wedge (h\bar{s})^{r/4}p}{n^{r/4}}},
\end{align*}
where the third inequality follows from the Cauchy-Schwarz inequality and using that $X_{k,i}$ and $u_i$ have uniformly bounded $r'th$ moments. Hence,
\begin{align*}
\sum_{i=1}^nE\envert[1]{\alpha'\Theta X_iu_i/n^{1/2}}^{r/2}
=
O\del{\frac{(h\bar{s})^{r/4+1}\wedge (h\bar{s})^{r/4}p}{n^{r/4-1}}} = o(1),
\end{align*}
by Assumption 3d). Next, we show that $\alpha'\Theta \Sigma_{xu} \Theta'\alpha$ is asymptotically bounded away from zero in (\ref{Lyap}). Clearly,
\begin{align}
\alpha'\Theta \Sigma_{xu} \Theta'\alpha
\geq
\phi_{\min}(\Sigma_{xu})\enVert[0]{\Theta'\alpha}_2^2
\geq
\phi_{\min}(\Sigma_{xu})\phi_{\min}^2(\Theta)\enVert[0]{\alpha}_2^2
=
\phi_{\min}(\Sigma_{xu})\frac{1}{\phi_{\max}^2(\Sigma)},
\label{lowbound}
\end{align} 
which is bounded away from zero since $\phi_{\min}(\Sigma_{xu})$ is bounded away from zero and $\phi_{\max}(\Sigma)$ is bounded from above. Hence, the Lyapounov condition is satisfied and $t_1'$ converges in distribution to a standard normal.

{\bf b)} We now show that $t_1'-t_1=o_p(1)$. To do so it suffices that the numerators as well as the denominators of $t_1'$ and $t_1$ are asymptotically equivalent since $\alpha'\Theta\Sigma_{xu} \Theta'\alpha$ is bounded away from 0 by (\ref{lowbound}). We first show that the denominators of $t_1'$ and $t_1$ are asymptotically equivalent, i.e.
\begin{equation}
| \alpha'\hat{\Theta}\hat{\Sigma}_{xu} \hat{\Theta}'\alpha - \alpha'\Theta \Sigma_{xu} \Theta'\alpha | = o_p (1).\label{t.2}
\end{equation}
Set $\tilde{\Sigma}_{xu} = n^{-1} \sum_{i=1}^n X_i X_i' u_i^2$. To establish (\ref{t.2}) it suffices to show the following relations:
\begin{equation}
| \alpha'\hat{\Theta} \hat{\Sigma}_{xu} \hat{\Theta}'\alpha - \alpha'\hat{\Theta} \tilde{\Sigma}_{xu} \hat{\Theta}'\alpha| = o_p(1).\label{t.3}
\end{equation}
\begin{equation}
| \alpha'\hat{\Theta}\tilde{\Sigma}_{xu} \hat{\Theta}'\alpha - \alpha\hat{\Theta} \Sigma_{xu} \hat{\Theta}'\alpha| = o_p(1).\label{t.4}
\end{equation}
\begin{equation}
| \alpha'\hat{\Theta} \Sigma_{xu} \hat{\Theta}'\alpha - \alpha'\Theta\Sigma_{xu} \Theta'\alpha| = o_p(1).\label{t.5}
\end{equation} 
We first prove (\ref{t.3}). 
 \begin{equation}
| \alpha'\hat{\Theta} \hat{\Sigma}_{xu} \hat{\Theta}'\alpha - \alpha'\hat{\Theta} \tilde{\Sigma}_{xu} \hat{\Theta}'\alpha|
\le 
\| \hat{\Sigma}_{xu} - \tilde{\Sigma}_{xu}\|_{\infty} \|\hat{\Theta}'\alpha\|_1^2
.\label{t.5a}
\end{equation}
But by (\ref{sbar}) and $\enVert[0]{\alpha}_2=1$
\begin{equation}
\enVert[1]{\hat{\Theta}'\alpha}_1
=
\enVert[4]{\sum_{j\in H}\hat{\Theta}_j\alpha_j}_1
\leq
\sum_{j\in H}|\alpha_j|\enVert[1]{\hat{\Theta}_j}_1
=
O_p\del[1]{\sqrt{h\bar{s}}}.
\label{t.6}
\end{equation}
To proceed, we bound $\enVert[0]{\hat{\Sigma}_{xu} - \tilde{\Sigma}_{xu}}_{\infty}$. Using $\hat{u}_i = u_i - X_i' (\hat{\beta} - \beta_0)$ in the definition of $\hat{\Sigma}_{xu}$ we get 
\begin{equation}
\hat{\Sigma}_{xu} - \tilde{\Sigma}_{xu} 
=
- \frac{2}{n} \sum_{i=1}^n X_i X_i' u_i X_i' (\hat{\beta} - \beta_0) + \frac{1}{n} \sum_{i=1}^n X_i X_i' (\hat{\beta} - \beta_0)' X_i X_i' (\hat{\beta} - \beta_0).\label{t.7}
\end{equation}
We bound each sum separately. First, by the Cauchy-Schwarz inequality, 
\begin{align}
\max_{1\leq k,l\leq p}\envert[3]{\frac{2}{n}\sum_{i=1}^nX_{k,i}X_{l,i}u_iX_i'(\hat{\beta}-\beta_0)}
\leq
2\sqrt{\max_{1\leq k,l\leq p}\frac{1}{n}\sum_{i=1}^nX_{k,i}^2X_{l,i}^2u_i^2}\cdot\enVert[1]{X(\hat{\beta}-\beta_0)}_n.\label{CS1}
\end{align}
Now for any three random variables $Z_1,Z_2$ and $Z_3$ with finite $r$'th moment it follows from two applications of Hölder's inequality
\begin{align}
E|Z_1^2Z_2^2Z_3^2|^{r/6}
=
E|Z_1^{r/3}Z_2^{r/3}Z_3^{r/3}|
&\leq
E\del[1]{|Z_1|^{r/2}|Z_2|^{r/2}}^{2/3}E\del[1]{|Z_3^r|}^{1/3}\notag\\
&\leq
E\del[1]{|Z_1^r|}^{1/3}E\del[1]{|Z_2^r|}^{1/3}E\del[1]{|Z_3^r|}^{1/3}.\label{holder}
\end{align}
Thus, by Assumption 1, all summands in (\ref{CS1}) have uniformly bounded $r/6$ moments and therefore Lemma \ref{conc} implies that
\begin{align*}
P\del[3]{\max_{1\leq k,l\leq p}\envert[3]{\frac{1}{n}\sum_{i=1}^n\del[2]{X_{k,i}^2X_{l,i}^2u_i^2-E(X_{k,i}^2X_{l,i}^2u_i^2)}}>t}
\leq
b_{r/6}\frac{Cp^2n^{r/12}}{(tn)^{r/6}}.
\end{align*}
Hence, choosing $t=M\frac{p^{12/r}}{n^{1/2}}$ for $M>0$ sufficiently large shows that
\begin{align*}
\max_{1\leq k,l\leq p}\envert[3]{\frac{1}{n}\sum_{i=1}^n\del[2]{X_{k,i}^2X_{l,i}^2u_i^2-E(X_{k,i}^2X_{l,i}^2u_i^2)}}
=
O_p\del[2]{\frac{p^{12/r}}{n^{1/2}}}.
\end{align*}
Furthermore, since the $L^r$-norm is non-decreasing in $r$ and since $r\ge6$ we have, using (\ref{holder}) above,
\begin{align*}
\max_{1\leq k,l\leq p}\frac{1}{n}\sum_{i=1}^{n}E\del[1]{X_{k,i}^2X_{l,i}^2u_i^2}
&\leq
\max_{1\leq k,l\leq p}\frac{1}{n}\sum_{i=1}^{n}\del[2]{E\del[1]{X_{k,i}^2X_{l,i}^2u_i^2}^{r/6}}^{6/r}\\
&\leq
\max_{1\leq k,l\leq p}\frac{1}{n}\sum_{i=1}^{n}\sbr[2]{\del[1]{E|X_{k,i}|^r}^{1/3}\del[1]{E|X_{l,i}|^r}^{1/3}\del[1]{E|u_{i}|^r}^{1/3}}^{6/r},
\end{align*}
which is uniformly bounded by Assumption 1 since the $r$'th moments of $X_{k,i}$ and $u_i$ are uniformly bounded. Therefore, $\sqrt{\max_{1\leq k,l\leq p}\frac{1}{n}\sum_{i=1}^nX_{k,i}^2X_{l,i}^2u_i^2}=O(1)+O_p\del[1]{\frac{p^{6/r}}{n^{1/4}}}$ in (\ref{CS1}). By Theorem \ref{thm1} it follows from choosing $M$ sufficiently large
\begin{align}
\enVert[1]{X(\hat{\beta}-\beta_0)}_n=O_p\del[3]{\frac{\sqrt{d_{n1}}p^{2/r}\sqrt{s_0}}{n^{1/2}} }.\label{unif1}
\end{align}
Thus,
\begin{align}
\max_{1\leq k,l\leq p}\envert[3]{\frac{2}{n}\sum_{i=1}^nX_{k,i}X_{l,i}u_iX_i'(\hat{\beta}-\beta_0)}
=
O_p\del[3]{\frac{p^{8/r}\sqrt{s_0}}{n^{3/4}}}+O_p\del[3]{\frac{\sqrt{d_{n1}}p^{2/r}\sqrt{s_0}}{n^{1/2}}}.
\label{a}
\end{align}
Regarding the second term in (\ref{t.7}) note that
\begin{align}
\max_{1\leq k,l \leq p}\envert[3]{\frac{1}{n} \sum_{i=1}^n X_{k,i} X_{l,i} (\hat{\beta} - \beta_0)' X_i X_i' (\hat{\beta} - \beta_0)}
\leq
\max_{1\leq k,l \leq p}\max_{1\leq i \leq n}\envert[1]{X_{k,i}X_{l,i}}\frac{1}{n}\sum_{i=1}^n\del[1]{X_i'(\hat{\beta}-\beta_0)}^2.\label{Bound2}
\end{align}
By the Cauchy-Schwarz inequality, $X_{k,i}X_{l,i}$ has uniformly bounded $r/2$ moments. Hence, by the union bound and Markov's inequality, for any $t>0$ we get via Lemma \ref{conc}
\begin{align*}
P\del[3]{\max_{1\leq i\leq n}\max_{1\leq k,l\leq p}\envert[2]{X_{k,i} X_{l,i}}>t}
\leq np^2\frac{C}{t^{r/2}}.
\end{align*} 
Therefore, choosing $t=Mp^{4/r}n^{2/r}$ for $M>0$ sufficiently large reveals that 
\begin{align*}
\max_{1\leq i\leq n}\max_{1\leq k,l\leq p}\envert[1]{ X_{k,i} X_{l,i}}
=
O_p\del[2]{p^{4/r}n^{2/r}}.
\end{align*}
Next, note that by Theorem 1
\begin{align}
\frac{1}{n}\sum_{i=1}^n\del[1]{X_i'(\hat{\beta}-\beta_0)}^2
=
\enVert[1]{X(\hat{\beta}-\beta_0)}_n^2
=
O_p\del[3]{d_{n1} \frac{p^{4/r}s_0}{n} },\label{unif2}
\end{align}
such that, using (\ref{Bound2}), 
\begin{align}
\max_{1\leq k,l \leq p}\envert[3]{\frac{1}{n} \sum_{i=1}^n X_{k,i} X_{l,i} (\hat{\beta} - \beta_0)' X_i X_i' (\hat{\beta} - \beta_0)}
&=
O_p\del[2]{p^{4/r}n^{2/r}}O_p\del[3]{\frac{d_{n1} p^{4/r}s_0}{n}}\notag\\
&=
O_p\del[3]{\frac{d_{n1} p^{8/r}s_0}{n^{(r-2)/r}}}\label{b}.
\end{align}
Then, combining (\ref{a}) and (\ref{b}) implies that
\begin{align*}
\enVert[1]{\hat{\Sigma}_{xu} - \tilde{\Sigma}_{xu}}_\infty
=
O_p\del[3]{\frac{p^{8/r}\sqrt{s_0}}{n^{3/4}}}+O_p\del[3]{\frac{\sqrt{d_{n1}} p^{2/r}\sqrt{s_0}}{n^{1/2}}}
+O_p\del[3]{\frac{d_{n1} p^{8/r}s_0}{n^{(r-2)/r}}}.
\end{align*}
Therefore, combining with (\ref{t.6}) yields
\begin{align}
\envert[1]{\alpha'\hat{\Theta} \hat{\Sigma}_{xu} \hat{\Theta}'\alpha - \alpha'\hat{\Theta} \tilde{\Sigma}_{xu} \hat{\Theta}'\alpha}
=
O_p\del[3]{\frac{p^{8/r}\sqrt{s_0}h\bar{s}}{n^{3/4}}}+O_p\del[3]{\frac{\sqrt{d_{n1}}p^{2/r}\sqrt{s_0}h\bar{s}}{n^{1/2}}}
+O_p\del[3]{\frac{d_{n1} p^{8/r}s_0h\bar{s}}{n^{(r-2)/r}}}
=
o_p(1),\label{b.0}
\end{align}
by Assumption 3c) and since $d_{n1}$ is bounded by constants. This establishes (\ref{t.3}).

Next, we turn to (\ref{t.4}). First, note that
\begin{align}
| \alpha'\hat{\Theta}\tilde{\Sigma}_{xu} \hat{\Theta}'\alpha - \alpha\hat{\Theta} \Sigma_{xu} \hat{\Theta}'\alpha|
\le
 \|\tilde{\Sigma}_{xu} - \Sigma_{xu} \|_{\infty}
\|\hat{\Theta}'\alpha\|_1^2.\label{term2aux}
\end{align}
Furthermore, similarly to (\ref{holder}), three applications of Hölder's inequality reveal that $X_{k,i}X_{l,i}u_i^2$ have uniformly bounded $r/4$ moments. Hence, by Lemma \ref{conc}, for any $t>0$
\begin{align*}
P\del[2]{\|\tilde{\Sigma}_{xu} - \Sigma_{xu} \|_{\infty}>t}
=
P\del[3]{\envert[2]{\frac{1}{n}\sum_{i=1}^nX_{k,i}X_{l,i}u_i^2-E\del[1]{X_{k,i}X_{l,i}u_i^2}}>t}
\leq 
b_{r/4}\frac{p^2Cn^{r/8}}{(tn)^{r/4}}.
\end{align*}
Thus, choosing $t=M\frac{p^{8/r}}{n^{1/2}}$ for $M>0$ sufficiently large shows that
\begin{align*}
\|\tilde{\Sigma}_{xu} - \Sigma_{xu} \|_{\infty}
=
O_p\del[3]{\frac{p^{8/r}}{n^{1/2}}}.
\end{align*}
By (\ref{term2aux}) and (\ref{t.6})
\begin{align*}
| \alpha'\hat{\Theta}\tilde{\Sigma}_{xu} \hat{\Theta}'\alpha - \alpha\hat{\Theta} \Sigma_{xu} \hat{\Theta}'\alpha|
=
O_p\del[3]{\frac{p^{8/r}h\bar{s}}{n^{1/2}}}
=o_p(1),
\end{align*}
and Assumption 3b). 

Finally, we establish (\ref{t.5}) to conclude (\ref{t.2}). By Lemma 6.1 in \cite{van2014}
\begin{align*}
| \alpha'\hat{\Theta} \Sigma_{xu} \hat{\Theta}'\alpha - \alpha'\Theta \Sigma_{xu} \Theta'\alpha| 
&\le
 \|\Sigma_{xu}\|_{\infty} \|\hat{\Theta}'\alpha - \Theta'\alpha\|_1^2 
+ 2 \|\Sigma_{xu}  \Theta'\alpha \|_{2} \|\hat{\Theta}'\alpha - \Theta'\alpha \|_2 \\
& \leq
 \|\Sigma_{xu}\|_{\infty} \|(\hat{\Theta}' - \Theta')\alpha\|_1^2 
+ 2 \phi_{\max}({\Sigma_{xu}}) \enVert[1]{\Theta'\alpha}_2 \|(\hat{\Theta}' - \Theta')\alpha \|_2.
\end{align*}
Note that
\begin{align}
\|(\hat{\Theta}' - \Theta')\alpha\|_1
&=
\enVert[4]{\sum_{j\in H}\del[1]{\hat{\Theta}_j-\Theta_j}\alpha_j}_1
\leq
\sum_{j\in H}\enVert[1]{\hat{\Theta}_j-\Theta_j}_1|\alpha_j|
\leq
\max_{j\in H}\enVert[1]{\hat{\Theta}_j-\Theta_j}_1\sum_{j\in H}|\alpha_j|\notag\\
&=
O_p\del[3]{d_{n2} \bar{s} \frac{h^{2/r+1/2}p^{2/r}}{\sqrt{n}}},\label{l2auxx}
\end{align}
by (\ref{l1theta}) and $\enVert[0]{\alpha}_2=1$. Furthermore, using the symmetry of $\Theta$,
\begin{align*}
\enVert[1]{\Theta'\alpha}_2
\leq
\phi_{\max}(\Theta) \|\alpha\|_2
=
\frac{1}{\phi_{\min}(\Sigma)},
\end{align*}
which is bounded by Assumption 2a). Finally,
\begin{align*}
\|(\hat{\Theta}' - \Theta')\alpha\|_2
&=
\enVert[4]{\sum_{j\in H}\del[1]{\hat{\Theta}_j-\Theta_j}\alpha_j}_2
\leq
\sum_{j\in H}\enVert[1]{\hat{\Theta}_j-\Theta_j}_2|\alpha_j|
\leq
\max_{j\in H}\enVert[1]{\hat{\Theta}_j-\Theta_j}_2\sum_{j\in H}|\alpha_j|\\
&=
O_p\del[3]{\sqrt{d_{n1}} \sqrt{ \bar{s}} \frac{h^{2/r+1/2}p^{2/r}}{\sqrt{n}}},
\end{align*}
by (\ref{l2theta}) and $\enVert[0]{\alpha}_2=1$. Therefore, by $\enVert[0]{\Sigma_{xu}}_\infty\leq \phi_{\max}(\Sigma_{xu})$ with the latter assumed bounded from Assumption 3e),
\begin{align}
| \alpha'\hat{\Theta} \Sigma_{xu} \hat{\Theta}'\alpha - \alpha'\Theta \Sigma_{xu} \Theta'\alpha| 
=
O_p\del[3]{d_{n2}^2 \bar{s}^2 \frac{h^{4/r+1}p^{4/r}}{n}}+O_p\del[3]{\sqrt{d_{n1}} \sqrt{\bar{s}} \frac{h^{2/r+1/2}p^{2/r}}{\sqrt{n}}}
=
o_p(1),\label{l2auxxb}
\end{align}
where we used 
\begin{align*}
\frac{\bar{s}^2 h^{(4/r)+1} p^{4/r}}{n} 
\leq
\frac{\bar{s} (h \bar{s}) p^{8/r}}{n}
=
\frac{\bar{s}}{n^{1/2}}\cdot \frac{ (h \bar{s}) p^{8/r}}{n^{1/2}}\to 0,
\end{align*}
and Assumption 3b (which also implies $\bar{s}=o(n^{1/2})$), and $d_{n1}, d_{n2}$ being bounded by constants. The uniformity of (\ref{t.2}) over $\mathcal{B}_{\ell_0}(s_0)$ follows from simply observing that (\ref{unif1}) and (\ref{unif2}) above are actually valid uniformly over this set and that this is the only place in which $\beta_0$ enters in the above arguments.

We now turn to showing that the numerators of $t_1'$ and $t_1$ are asymptotically equivalent, i.e.
\[ |\alpha'\hat{\Theta} X'u/n^{1/2} - \alpha'\Theta X'u/n^{1/2}| = o_p (1).\]
By Lemma \ref{Noise} and (\ref{l2auxx}) above we get, using $h\leq p$, and Assumption 3b, $d_{n2}$ being bounded by constants
\begin{align}
n^{1/2}|\alpha'\hat{\Theta} X'u/n - \alpha'\Theta X'u/n|  
&\leq
  n^{1/2} \enVert[2]{\frac{X'u}{n}}_{\infty} \|\alpha'(\hat{\Theta} - \Theta) \|_1 \notag\\
 &=
  n^{1/2} O_p\del[3]{\frac{p^{2/r}}{\sqrt{n}}} O\del[3]{d_{n2} \bar{s} \frac{h^{2/r+1/2}p^{2/r}}{\sqrt{n}}}\notag\\
 &= 
O_p \del[2]{d_{n2} \bar{s}\frac{h^{2/r+1/2}p^{4/r}}{\sqrt{n}}}\notag\\
&=
O_p \del[2]{d_{n2} \bar{s}\frac{h^{1/2}p^{6/r}}{\sqrt{n}}}\notag\\
&=
o_p (1)\label{numerator}.
\end{align}

{\bf Step 2}. It remains to be shown that $t_2=o_p(1)$. The denominators of $t_1$ and $t_2$ are identical. Hence, the denominator of $t_2$ is asymptotically bounded away from zero with probability approaching one by (\ref{lowbound}) and (\ref{t.2}). Thus, it suffices to show that the numerator of $t_2$ vanishes in probability. Note that, by the definition of $\Delta$, and $\|\alpha \|_2=1$,
 \begin{align}
 |\alpha'\Delta|
&\leq
\max_{j\in H}\envert[0]{\Delta_j}\sum_{j\in H}|\alpha_j|
=
\max_{j\in H}\envert[2]{\del[1]{\hat{\Theta}_j'\hat{\Sigma}-e_j}\del[1]{\sqrt{n}(\hat{\beta}-\beta_0)}}\sum_{j\in H}|\alpha_j|\\
&\leq
\max_{j\in H}\enVert[2]{\del[1]{\hat{\Theta}_j'\hat{\Sigma}-e_j}}_\infty \enVert[1]{\sqrt{n}(\hat{\beta}-\beta_0)}_1 O\del[2]{\sqrt{h}}.\label{pt33c.1}
 \end{align}  
First, it follows from Theorem \ref{thm1} that $n^{1/2} \|\hat{\beta} - \beta_0 \|_1 =O_p\del[1]{ d_{n2} s_0p^{2/r}}$. Next, we consider  
\begin{align*}
\max_{j\in H}\enVert[2]{\del[1]{\hat{\Theta}_j'\hat{\Sigma}-e_j}}_\infty
\leq
\max_{j\in H}\frac{\lambda_{node,n}}{\hat{\tau}_j^2}
=
O_p\del[3]{\frac{h^{2/r}p^{2/r}}{n^{1/2}}},
\end{align*}
where we have used the definition of $\lambda_{node,n}$ and $\max_{j\in H}1/\hat{\tau}_j^2=O_p(1)$ by (\ref{pt32.3i}) and Assumption 3b). Thus, in total we have
\begin{align}
\envert[1]{\alpha'\Delta}
=
O_p\del[3]{\frac{h^{2/r}p^{2/r}}{n^{1/2}}}O_p\del[1]{d_{n2} s_0p^{2/r}}O\del[2]{\sqrt{h}}
=
O_p\del[3]{d_{n2} s_0\frac{h^{2/r+1/2}p^{4/r}}{n^{1/2}}}
=
o_p(1),\label{thm2aa}
\end{align}
by Assumption 3a), and $d_{n2}$ being bounded by constants. The fact that $\sup_{\beta_0\in\mathcal{B}_{\ell_0}(s_0)}\envert[1]{\alpha'\Delta}=o_p(1)$ follows from the observation that Theorem \ref{thm1} actually yields that $\sup_{\beta_0\in\mathcal{B}_{\ell_0}(s_0)}n^{1/2} \|\hat{\beta} - \beta_0 \|_1 =O_p\del[1]{d_{n2} s_0p^{2/r}}$ in the above argument and that this is the only place in which $\beta_0$ enters these arguments. Thus, for later reference,
\begin{align}
\sup_{\beta_0\in\mathcal{B}_{\ell_0}(s_0)}\envert[1]{\alpha'\Delta}=o_p(1).\label{Deltaunif}
\end{align}
\end{proof}

\begin{proof}[Proof of Theorem \ref{thm3}]
For $\epsilon>0$ define
\begin{align*}
A_{1,n}:=\cbr[3]{\sup_{\beta_0\in\mathcal{B}_{\ell_0}(s_0)}\envert[1]{\alpha'\Delta}< \epsilon},\ 
A_{2,n}:=\cbr[4]{\sup_{\beta_0\in\mathcal{B}_{\ell_0}(s_0)}\envert[3]{\frac{\sqrt{\alpha'\hat{\Theta} \hat{\Sigma}_{xu} \hat{\Theta}'\alpha}}{\sqrt{\alpha'\Theta \Sigma_{xu} \Theta'\alpha}}-1}<\epsilon},
\end{align*}
and
\begin{align*}
A_{3,n}:=\cbr[2]{\envert[1]{\alpha'\hat{\Theta} X'u/n^{1/2} - \alpha'\Theta X'u/n^{1/2}}<\epsilon}.
\end{align*}
By, (\ref{Deltaunif}), (\ref{asymcov}), (\ref{numerator}), and $\sqrt{\alpha'\Theta \Sigma_{xu} \Theta'\alpha}$ being bounded away from zero (by (\ref{lowbound})) the probabilities of these three sets all tend to one. Thus, for every $t\in\mathbb{R}$,
\begin{align*}
&\envert[4]{P\del[4]{\frac{n^{1/2} \alpha'(\hat{b} - \beta_{0})}{\sqrt{\alpha'\hat{\Theta} \hat{\Sigma}_{xu} \hat{\Theta}'\alpha}}\leq t}-\Phi(t)}\\
&=
\envert[4]{P\del[4]{\frac{\alpha'\hat{\Theta} X' u /n^{1/2}}{\sqrt{\alpha'\hat{\Theta} \hat{\Sigma}_{xu} \hat{\Theta}'\alpha}} -\frac{ \alpha'\Delta}{\sqrt{\alpha'\hat{\Theta} \hat{\Sigma}_{xu} \hat{\Theta}'\alpha}}\leq t}-\Phi(t)}\\
&\leq 
\envert[4]{P\del[4]{\frac{\alpha'\hat{\Theta} X' u /n^{1/2}}{\sqrt{\alpha'\hat{\Theta} \hat{\Sigma}_{xu} \hat{\Theta}'\alpha}} -\frac{ \alpha'\Delta}{\sqrt{\alpha'\hat{\Theta} \hat{\Sigma}_{xu} \hat{\Theta}'\alpha}}\leq t, A_{1,n}, A_{2,n}, A_{3,n}}-\Phi(t)}+P\del[1]{\cup_{i=1}^3A^c_{i,n}}.
\end{align*}

Using that $\sqrt{\alpha'\Theta \Sigma_{xu} \Theta'\alpha}$ does not depend on $\beta_0$ and is bounded away from zero by (\ref{lowbound}) there exists a positive constant $D$ such that 
\begin{align*}
&P\del[4]{\frac{\alpha'\hat{\Theta} X' u /n^{1/2}}{\sqrt{\alpha'\hat{\Theta} \hat{\Sigma}_{xu} \hat{\Theta}'\alpha}} -\frac{ \alpha'\Delta}{\sqrt{\alpha'\hat{\Theta} \hat{\Sigma}_{xu} \hat{\Theta}'\alpha}}\leq t, A_{1,n}, A_{2,n}, A_{3,n}}\\
&=
P\del[4]{\frac{\alpha'\hat{\Theta} X' u /n^{1/2}}{\sqrt{\alpha'\Theta \Sigma_{xu} \Theta'\alpha}} -\frac{ \alpha'\Delta}{\sqrt{\alpha'\Theta \Sigma_{xu} \Theta'\alpha}}\leq t\frac{\sqrt{\alpha'\hat{\Theta} \hat{\Sigma}_{xu} \hat{\Theta}'\alpha}}{\sqrt{\alpha'\Theta \Sigma_{xu} \Theta'\alpha}}, A_{1,n}, A_{2,n}, A_{3,n}}\\
&\leq
P\del[4]{\frac{\alpha'\Theta X' u /n^{1/2}}{\sqrt{\alpha'\Theta \Sigma_{xu} \Theta'\alpha}}\leq t(1+\epsilon)+\frac{\epsilon+\epsilon}{\sqrt{\alpha'\Theta \Sigma_{xu} \Theta'\alpha}}}\\
&\leq
P\del[4]{\frac{\alpha'\Theta X' u /n^{1/2}}{\sqrt{\alpha'\Theta \Sigma_{xu} \Theta'\alpha}}\leq t(1+\epsilon)+2D\epsilon}.
\end{align*}
Thus, as the right hand side in the above display does not depend on $\beta_0$
\begin{align*}
&\sup_{\beta_0\in\mathcal{B}_{\ell_0}(s_0)}P\del[4]{\frac{\alpha'\hat{\Theta} X' u /n^{1/2}}{\sqrt{\alpha'\hat{\Theta} \hat{\Sigma}_{xu} \hat{\Theta}'\alpha}} -\frac{ \alpha'\Delta}{\sqrt{\alpha'\hat{\Theta} \hat{\Sigma}_{xu} \hat{\Theta}'\alpha}}\leq t, A_{1,n}, A_{2,n}, A_{3,n}}\\
&\leq
P\del[4]{\frac{\alpha'\Theta X' u /n^{1/2}}{\sqrt{\alpha'\Theta \Sigma_{xu} \Theta'\alpha}}\leq t(1+\epsilon)+2D\epsilon}.
\end{align*}
In step 1a) of the proof of Theorem \ref{thm2} we established the asymptotic normality of $\frac{\alpha'\Theta X' u /n^{1/2}}{\sqrt{\alpha'\Theta \Sigma_{xu} \Theta'\alpha}}$. Therefore, for $n$ sufficiently large,
\begin{align*}
\sup_{\beta_0\in\mathcal{B}_{\ell_0}(s_0)}P\del[4]{\frac{\alpha'\hat{\Theta} X' u /n^{1/2}}{\sqrt{\alpha'\hat{\Theta} \hat{\Sigma}_{xu} \hat{\Theta}'\alpha}} -\frac{ \alpha'\Delta}{\sqrt{\alpha'\hat{\Theta} \hat{\Sigma}_{xu} \hat{\Theta}'\alpha}}\leq t,\ A_{1,n}, A_{2,n}, A_{3,n}}
\leq
\Phi\del[1]{t(1+\epsilon)+2D\epsilon}+\epsilon.
\end{align*}
As the above arguments are valid for all $\epsilon >0$ we can use the continuity of $q\mapsto \Phi(q)$ to conclude that for any $\delta>0$ we can choose $\epsilon$ sufficiently small to conclude that
\begin{align}
\sup_{\beta_0\in\mathcal{B}_{\ell_0}(s_0)}P\del[4]{\frac{\alpha'\hat{\Theta} X' u /n^{1/2}}{\sqrt{\alpha'\hat{\Theta} \hat{\Sigma}_{xu} \hat{\Theta}'\alpha}} -\frac{ \alpha'\Delta}{\sqrt{\alpha'\hat{\Theta} \hat{\Sigma}_{xu} \hat{\Theta}'\alpha}}\leq t,\ A_{1,n}, A_{2,n}, A_{3,n}}
\leq 
\Phi(t)+\delta+\epsilon\label{upper}.
\end{align}
Next, using that $\sqrt{\alpha'\Theta \Sigma_{xu} \Theta'\alpha}$ does not depend on $\beta_0$ and is bounded away from zero by (\ref{lowbound}) there exists a positive constant $D$ such that 
\begin{align*}
&P\del[4]{\frac{\alpha'\hat{\Theta} X' u /n^{1/2}}{\sqrt{\alpha'\hat{\Theta} \hat{\Sigma}_{xu} \hat{\Theta}'\alpha}} -\frac{ \alpha'\Delta}{\sqrt{\alpha'\hat{\Theta} \hat{\Sigma}_{xu} \hat{\Theta}'\alpha}}\leq t,\ A_{1,n}, A_{2,n}, A_{3,n}}\\
&=
P\del[4]{\frac{\alpha'\hat{\Theta} X' u /n^{1/2}}{\sqrt{\alpha'\Theta \Sigma_{xu} \Theta'\alpha}} -\frac{ \alpha'\Delta}{\sqrt{\alpha'\Theta \Sigma_{xu} \Theta'\alpha}}\leq t\frac{\sqrt{\alpha'\hat{\Theta} \hat{\Sigma}_{xu} \hat{\Theta}'\alpha}}{\sqrt{\alpha'\Theta \Sigma_{xu} \Theta'\alpha}},\ A_{1,n}, A_{2,n}, A_{3,n}}\\
&\geq
P\del[4]{\frac{\alpha'\Theta X' u /n^{1/2}}{\sqrt{\alpha'\Theta \Sigma_{xu} \Theta'\alpha}}\leq t(1-\epsilon)-\frac{ \epsilon+\epsilon}{\sqrt{\alpha'\Theta \Sigma_{xu} \Theta'\alpha}},\ A_{1,n}, A_{2,n}, A_{3,n}}\\
&\geq
P\del[4]{\frac{\alpha'\Theta X' u /n^{1/2}}{\sqrt{\alpha'\Theta \Sigma_{xu} \Theta'\alpha}}\leq t(1-\epsilon)-2D\epsilon,\ A_{1,n}, A_{2,n}, A_{3,n}}\\
&\geq
P\del[4]{\frac{\alpha'\Theta X' u /n^{1/2}}{\sqrt{\alpha'\Theta \Sigma_{xu} \Theta'\alpha}}\leq t(1-\epsilon)-2D\epsilon}+P\del[2]{\cap_{i=1}^3A_{i,n}}-1.
\end{align*}
Thus, as the right hand side in the above display does not depend on $\beta_0$ and since $P\del[2]{\cap_{i=1}^3A_{i,n}}$ can be made arbitrarily close to one by choosing $n$ sufficiently we conclude 
\begin{align*}
&\inf_{\beta_0\in\mathcal{B}_{\ell_0}(s_0)}P\del[4]{\frac{\alpha'\hat{\Theta} X' u /n^{1/2}}{\sqrt{\alpha'\hat{\Theta} \hat{\Sigma}_{xu} \hat{\Theta}'\alpha}} -\frac{ \alpha'\Delta}{\sqrt{\alpha'\hat{\Theta} \hat{\Sigma}_{xu} \hat{\Theta}'\alpha}}\leq t,\ A_{1,n}, A_{2,n}, A_{3,n}}\\
&\geq
P\del[4]{\frac{\alpha'\Theta X' u /n^{1/2}}{\sqrt{\alpha'\Theta \Sigma_{xu} \Theta'\alpha}}\leq t(1-\epsilon)-2D\epsilon}-\epsilon,
\end{align*}
for $n$ sufficiently large. In step 1a) of the proof of Theorem \ref{thm2} we established the asymptotic normality of $\frac{\alpha'\Theta X' u /n^{1/2}}{\sqrt{\alpha'\Theta \Sigma_{xu} \Theta'\alpha}}$. Thus, for $n$ sufficiently large,
\begin{align*}
\inf_{\beta_0\in\mathcal{B}_{\ell_0}(s_0)}P\del[4]{\frac{\alpha'\hat{\Theta} X' u /n^{1/2}}{\sqrt{\alpha'\hat{\Theta} \hat{\Sigma}_{xu} \hat{\Theta}'\alpha}} -\frac{ \alpha'\Delta}{\sqrt{\alpha'\hat{\Theta} \hat{\Sigma}_{xu} \hat{\Theta}'\alpha}}\leq t,\ A_{1,n}, A_{2,n}, A_{3,n}}
\geq
\Phi\del[1]{t(1-\epsilon)-2D\epsilon}-2\epsilon.
\end{align*}
As the above arguments are valid for all $\epsilon >0$ we can use the continuity of $q\mapsto \Phi(q)$ to conclude that for any $\delta>0$ we can choose $\epsilon$ sufficiently small to conclude that
\begin{align}
\inf_{\beta_0\in\mathcal{B}_{\ell_0}(s_0)}P\del[4]{\frac{\alpha'\hat{\Theta} X' u /n^{1/2}}{\sqrt{\alpha'\hat{\Theta} \hat{\Sigma}_{xu} \hat{\Theta}'\alpha}} -\frac{ \alpha'\Delta}{\sqrt{\alpha'\hat{\Theta} \hat{\Sigma}_{xu} \hat{\Theta}'\alpha}}\leq t,\ A_{1,n}, A_{2,n}, A_{3,n}}
\geq
\Phi(t)-2\epsilon-\delta.\label{lower}
\end{align}
By (\ref{upper}) and (\ref{lower}) and $\sup_{\beta_0\in\mathcal{B}_{\ell_0}(s_0)}P\del[1]{\cup_{i=1}^3A^c_{i,n}}=P\del[1]{\cup_{i=1}^3A^c_{i,n}}\to 0$ (here we used that none of the sets $A_1,A_2,$ or $A_3$ depend on $\beta_0$) we conclude that
\begin{align*}
\sup_{\beta_0\in\mathcal{B}_{\ell_0}(s_0)}\envert[4]{P\del[4]{\frac{n^{1/2} \alpha'(\hat{b} - \beta_{0})}{\sqrt{\alpha'\hat{\Theta} \hat{\Sigma}_{xu} \hat{\Theta}'\alpha}}\leq t}-\Phi(t)}\to 0.
\end{align*}

To see (\ref{t3p2}) note that
\begin{align*}
&P\del[3]{\beta_{0,j}\notin\sbr[2]{\hat{b}_j-z_{1-\alpha/2}\frac{\hat{\sigma}_j}{\sqrt{n}},\hat{b}_j+z_{1-\alpha/2}\frac{\hat{\sigma}_j}{\sqrt{n}}}}\\
&=
P\del[4]{\envert[3]{\frac{\sqrt{n}\del[1]{\hat{b}_j-\beta_{0,j}}}{\hat{\sigma}_j}}>z_{1-\alpha/2}}\\
&=
P\del[4]{\frac{\sqrt{n}\del[1]{\hat{b}_j-\beta_{0,j}}}{\hat{\sigma}_j}>z_{1-\alpha/2}}+P\del[4]{\frac{\sqrt{n}\del[1]{\hat{b}_j-\beta_{0,j}}}{\hat{\sigma}_j}<-z_{1-\alpha/2}}\\
&\le
1-P\del[4]{\frac{\sqrt{n}\del[1]{\hat{b}_j-\beta_{0,j}}}{\hat{\sigma}_j}\leq z_{1-\alpha/2}}+P\del[4]{\frac{\sqrt{n}\del[1]{\hat{b}_j-\beta_{0,j}}}{\hat{\sigma}_j}\leq-z_{1-\alpha/2}}.
\end{align*}
Thus, taking the supremum over $\beta_0\in\mathcal{B}_{\ell_0}(s_0)$ and letting $n$ tend to infinity yields an inequality in (\ref{t3p2}) via (\ref{t3p1}). The reverse inequality follows upon noting that\\ $P\del[1]{\beta_{0,j}\notin\sbr[0]{\hat{b}_j-z_{1-\alpha/2}\frac{\hat{\sigma}_j}{\sqrt{n}},\hat{b}_j+z_{1-\alpha/2}\frac{\hat{\sigma}_j}{\sqrt{n}}}}\geq 1-P\del[1]{\frac{\sqrt{n}\del[0]{\hat{b}_j-\beta_{0,j}}}{\hat{\sigma}_j}\leq z_{1-\alpha/2}}+P\del[1]{\frac{\sqrt{n}\del[0]{\hat{b}_j-\beta_{0,j}}}{\hat{\sigma}_j}\leq-z_{1-\alpha/2-\delta_1}}$ for any $\delta_1>0$. 

Finally, we turn to (\ref{t3p3}). By (\ref{asymcov}) we know $\sup_{\beta_0\in\mathcal{B}_{\ell_0}(s_0)}\envert[1]{\alpha'\hat{\Theta} \hat{\Sigma}_{xu} \hat{\Theta}'\alpha - \alpha'\Theta \Sigma_{xu} \Theta'\alpha}
=o_p(1)$. Hence, choosing $\alpha=e_j$ and $\phi_{\max}(\Theta)=1/\phi_{\min}(\Sigma)$, 
\begin{align*}
&\sqrt{n}\sup_{\beta_0\in\mathcal{B}_{\ell_0}(s_0)}\diam\del[3]{\sbr[2]{\hat{b}_j-z_{1-\alpha/2}\frac{\hat{\sigma}_j}{\sqrt{n}},\hat{b}_j+z_{1-\alpha/2}\frac{\hat{\sigma}_j}{\sqrt{n}}}}
=
\sup_{\beta_0\in\mathcal{B}_{\ell_0}(s_0)}2 \hat{\sigma}_j z_{1-\alpha/2}\\
&=
2\del[3]{\sup_{\beta_0\in\mathcal{B}_{\ell_0}(s_0)}\sqrt{e_j'\Theta \Sigma_{xu} \Theta'e_j}+o_p(1)}z_{1-\alpha/2}\\
&\leq 
2\del[3]{\sqrt{\phi_{\max}(\Sigma_{xu})}\frac{1}{\phi_{\min}(\Sigma)}+o_p(1)}  z_{1-\alpha/2}\\
&=
O_p(1),
\end{align*}
as $\phi_{\max}(\Sigma_{xu})$ is bounded from above and $\phi_{\min}(\Sigma)$ is bounded from below by Assumptions 2a) and 3e).
\end{proof}

{\bf Strong oracle optimality of the variant of the Conservative Lasso}\\
We provide a strong oracle optimality result for $\tilde{\beta}$; the variant of the conservative Lasso estimator. Recall that
\[ \tilde{\beta} = \argmin_{\beta \in \mathbb{R}^p} \{ \| Y - X \beta \|_n^2 + 2 \lambda_n \sum_{j=1}^p \tilde{w}_j  |\beta_j| \},\]
with $\tilde{w}_j = 1_{ \{ |\hat{\beta}_{L,j} | \le \lambda_{prec} \}}$. Define the oracle estimator as
\begin{equation}
 \hat{\beta}^{oracle} = (\hat{\beta}_{S_0}^{oracle},0) = \argmin_{\beta, \beta_{S_0^c} = 0}[ \|Y- X \beta\|_n^2].\label{orc}
 \end{equation}
which we assume to be unique as in (\cite{fanx2014}). Strong oracle optimality of $\tilde{\beta}$ means it is equal to the oracle estimator with probability approaching one (\cite{fanx2014}).

%

%
%


Introduce the events
\begin{equation}
{\cal C}_1 = \{ \|\hat{\beta}_L - \beta_0 \|_{\infty}  \le \lambda_{prec}\}\label{c1}
\end{equation}
and
\begin{equation}
{\cal C}_2 = \{ \| (\bigtriangledown_{S_0^c} \| Y - X \hat{\beta}^{oracle}\|_n^2 )\|_{\infty} <  2 \lambda_n \}.\label{c2}
\end{equation}
where $\bigtriangledown_{S_0^c}$ denotes the gradient with respect to the entries of $\beta$ that are indexed by $S_0^c$. Next, we introduce the $n \times (p-s_0)$ matrix
\[ \tilde{X} = M_{S_0} X_{S_o^c} ,\]
with $M_{S_0} = I_n - X_{S_0} (X_{S_0}' X_{S_0})^{-1} X_{S_0}'$, and $X_{S_0^c}, X_{S_0}$ are ($n \times (p-s_0), n \times s_0$ matrices). 

\begin{thm}\label{thm4}
Impose Assumptions 1-2 and

(i). With probability approaching one  
\[ \min_{j\in S_0^c}\tilde{w}_j =1,\] and with added  $\min_{j \in S_0} | \beta_{0,j} | > 2 \lambda_{prec}$,
\[ \max_{j\in S_0}\tilde{w}_j = 0.\]

(ii). If, furthermore, $E | \tilde{X}_{j,i}|^r < C$ for a universal constant $C$ then for all $\epsilon>0$ there exists an $n$ sufficiently large such that
\[ P (\tilde{\beta} = \hat{\beta}^{oracle}) \geq 1-\epsilon.\]

\end{thm}

\textbf{Remarks.}
 
1.The first part of Theorem \ref{thm4} is similar to Lemma \ref{lt1} (ii)-(iii). However, the important difference is that the new variant of the conservative Lasso ensures that the weights pertaining to the non-zero coefficients will be exactly \textit{equal} to zero with probability approaching one. Lemma \ref{lt1} only guarantees that these weights \textit{converge} to zero for the conservative Lasso. The same caveat before Lemma 1 applies regarding the restrictiveness of the result since we use $\beta-\min$ condition.


2. Note that $\lambda_{prec} \to 0$ under Assumptions 1-2 also for the variant of the conservative Lasso.


3. Part (ii) of Theorem \ref{thm4} is the strong oracle optimality of $\tilde{\beta}$.

\begin{proof}
Throughout we assume that $ \Xi = {\cal C}_1 \cap {\cal C}_2$ occurs and show at the end of the proof that this is indeed the case with probability approaching one. First, on ${\cal C}_1$
\[ \max_{j\in S_0^c}| \hat{\beta}_{L,j} | =  \max_{j\in S_0^c}| \hat{\beta}_{L,j} - \beta_{0,j} | \le \lambda_{prec}.\]
This shows that 
\begin{equation}
\min_{j\in S_0^c}\tilde{w}_j = 1_{ \{ \max_{j\in S_0^c}| \hat{\beta}_{L,j}| \le \lambda_{prec} \}} =1,\label{w0} 
\end{equation}
Next we consider $j \in S_0$. 
\[ \min_{j\in S_0}| \hat{\beta}_{L,j} | \ge \min_{j \in S_0} |\beta_{0,j}| - \max_{j\in S_0}| \hat{\beta}_{L,j} - \beta_{0,j}| > 2 \lambda_{prec} - \lambda_{prec} = \lambda_{prec}.\]
Thus,  
\begin{equation}
\max_{j\in S_0}\tilde{w}_j = 1_{ \{ \min_{j\in S_0}| \hat{\beta}_{L,j} |\le \lambda_{prec} \}} =0.\label{w1}
\end{equation}

\noindent Now we show that $\tilde{\beta}=\hat{\beta}^{oracle}$ on $\Xi$. Note that
\begin{eqnarray}
\tilde{\beta}  =  \argmin_{\beta}  \{ \|Y - X \beta \|_n^2 + 2 \lambda_n \sum_{j=1}^p \tilde{w}_j | \beta_j| \}
 =  \argmin_{\beta}  \{ \|Y - X \beta \|_n^2 + 2 \lambda_n \sum_{j \in S_0^c} \tilde{w}_j | \beta_j| \},\label{pa1.1}
\end{eqnarray}
since $\tilde{w}_j = 0$ for $j \in S_0$ on ${\cal C}_1$. By convexity of $\| Y - X \beta \|_n^2$ in $\beta$
\begin{eqnarray}
\| Y - X \beta \|_n^2  & \ge & \|Y- X \hat{\beta}^{oracle} \|_n^2 + \sum_{j=1}^p \bigtriangledown_j \|Y - X \hat{\beta}^{oracle} \|_n^2 (\beta_j - \hat{\beta}_j^{oracle}) \nonumber \\
& = &   \|Y- X \hat{\beta}^{oracle} \|_n^2 + \sum_{j \in S_0^c} \bigtriangledown_j \|Y - X \hat{\beta}^{oracle} \|_n^2 (\beta_j - \hat{\beta}_j^{oracle}),\label{pa1.2}
\end{eqnarray}
where $\sum_{j \in S_0} (\bigtriangledown_j \|Y - X \hat{\beta}^{oracle} \|_n^2) = 0$ by the first order conditions for a minimum.
Add $2 \lambda_n \sum_{j \in S_0^c} \tilde{w}_j |\beta_j|$ to both sides of (\ref{pa1.2}) and note that $\hat{\beta}_j^{oracle} = 0 $ for $j \in S_0^c$ from oracle estimator definition,
\begin{eqnarray}
\| Y - X \beta \|_n^2  +   2 \lambda_n \sum_{j \in S_0^c} \tilde{w}_j |\beta_j| 
 \ge 
\| Y - X \hat{\beta}^{oracle} \|_n^2 +  2 \lambda_n \sum_{j \in S_0^c} \tilde{w}_j |\beta_j| 
 +  \sum_{j \in S_0^c} \bigtriangledown_j \|Y - X \hat{\beta}^{oracle} \|_n^2 \beta_j.\label{pa1.2a} 
\end{eqnarray}
Now subtract $\| Y - X \hat{\beta}^{oracle} \|_n^2$ from both sides of (\ref{pa1.2a}) and add $2 \lambda_n \sum_{j \in S_0^c}
\tilde{w}_j \hat{\beta}_j^{oracle}=0$ (which is zero since $\hat{\beta}_j^{oracle} = 0,$ for $ j \in S_0^c$ by the definition of the oracle estimator) to the left side of (\ref{pa1.2a}) to get 
\begin{eqnarray}
\| Y - X \beta \|_n^2 + 2 \lambda_n \sum_{j \in S_0^c} \tilde{w}_j |\beta_j| & - & 
\{ \| Y - X \hat{\beta}^{oracle} \|_n^2 + 2 \lambda_n \sum_{j \in S_0^c} \tilde{w}_j |\hat{\beta}_j^{oracle}| \} \nonumber \\
& \ge & 
[ 2 \lambda_n \sum_{j \in S_0^c} \tilde{w}_j | \beta_j | + \sum_{j \in S_0^c} \bigtriangledown_j \|Y - X \hat{\beta}^{oracle} \|_n^2 \beta_j].\label{pa1.2b}
\end{eqnarray}
Note that $\tilde{w}_j =0$ for all $j \in S_0$ by  (\ref{w1}).  Using this  fact, add  $2 \lambda_n \sum_{j \in S_0} \tilde{w}_j | \beta_j | =0$ and subtract $2 \lambda_n \sum_{j \in S_0} \tilde{w}_j  |\hat{\beta}_j^{oracle}|=0$ from the left side of (\ref{pa1.2b}).
\begin{eqnarray}
\| Y - X \beta \|_n^2 + 2 \lambda_n \sum_{j=1}^p \tilde{w}_j |\beta_j| & - & 
\{ \| Y - X \hat{\beta}^{oracle} \|_n^2 + 2 \lambda_n \sum_{j=1}^p \tilde{w}_j |\hat{\beta}_j^{oracle}| \} \nonumber \\
& \ge & 
[ 2 \lambda_n \sum_{j \in S_0^c} \tilde{w}_j | \beta_j | + \sum_{j \in S_0^c} \bigtriangledown_j \|Y - X \hat{\beta}^{oracle} \|_n^2 \beta_j] \nonumber \\
&=& \sum_{j \in S_0^c}
[ 2 \lambda_n  +  \bigtriangledown_j \|Y - X \hat{\beta}^{oracle} \|_n^2 sgn(\beta_j)] |\beta_j|,\label{pa1.2c}
\end{eqnarray}
where we use (\ref{w0}) in the last equality and $sgn (\beta_j) |\beta_j| =  \beta_j$.   Next, if $sgn(\beta_j) =1$, then 
\[ \sum_{j \in S_0^c}
[ 2 \lambda_n  +  \bigtriangledown_j \|Y - X \hat{\beta}^{oracle} \|_n^2 ] |\beta_j| > 0
\]
while if $sgn(\beta_j)=-1$, since ${\cal C}_2$ is assumed to occur,
\[ \sum_{j \in S_0^c}
[ 2 \lambda_n  - \bigtriangledown_j \|Y - X \hat{\beta}^{oracle} \|_n^2 ] |\beta_j| > 0.
\]
By these inequalities and (\ref{pa1.2c}) we conclude
\begin{eqnarray}
\| Y - X \beta \|_n^2 + 2 \lambda_n \sum_{j=1}^p \tilde{w}_j |\beta_j|  -  
\{ \| Y - X \hat{\beta}^{oracle} \|_n^2 + 2 \lambda_n \sum_{j=1}^p \tilde{w}_j |\hat{\beta}_j^{oracle}| \} \ge 0.
\label{pa1.2d}
\end{eqnarray}
Strict inequality in (\ref{pa1.2d}) is true, unless $\beta_j =0$, for all $j \in S_0^c$. We now turn to verifying that the probability of $\Xi$ tends to one. By the above display $\tilde{\beta} = \hat{\beta}^{oracle}$ on $\Xi= {\cal C}_1 \cap {\cal C}_2 $ since $\beta \mapsto \| Y - X \beta \|_n^2$ is assumed to be uniquely minimized at $ \hat{\beta}^{oracle}$.

Lemma \ref{lemma9} proves $P({\cal C}_1^c ) \to 0$ under Assumptions 1-2, which also establishes part (i) of the theorem since the desired properties of the weights have been established on $\mathcal{C}_1$

To establish (ii) of the theorem it remains to show that $ P ({\cal  C}_2^c ) \geq 1-\epsilon$ for any $\epsilon>0$. As in the proof of Theorem 3 in \cite{fanx2014} by definition of the oracle estimator in (\ref{orc}) via simple matrix algebra 
 \[ \sum_{j \in S_0^c} (\bigtriangledown_j \|Y - X \hat{\beta}^{oracle} \|_n^2) = \frac{2}{n} X_{S_0^c}' M_{S_0} u = \frac{2}{n} \tilde{X}'u.\]
Next, $E | \tilde{X}_{ij} u_i |^{r/2} \le \sqrt{E |\tilde{X}_{i,j}|^r E |u_i |^r} \le C$ such that Lemma \ref{conc} yields
 \begin{eqnarray}
 P [ \| \tilde{X}' u \|_{\infty} \ge (n \lambda_n)] & \le & \frac{ b_{r/2} (p-s_0) n^{r/4} \max_{j \in S_0^c} \max_{1 \le i \le n} E |\tilde{X}_{i,j} u_i |^{r/2}}
 {(n \lambda_n)^{r/2}} \nonumber \\
 & \le & \frac{b_{r/2} p n^{r/4} C }{(n \lambda_n)^{r/2}} = \frac{C}{M^{r/2}},\label{ptxu}
 \end{eqnarray}
where we used $\lambda_n = M p^{2/r}/n^{1/2}$, and combined the constants $b_{r/2}$ and $C$ into $C$. Choosing $M$ sufficiently large we can make the right hand side of (\ref{ptxu}) less than $\epsilon$. 
 
\end{proof}

{\bf Choice of Tuning Parameter $\lambda_n$}
In this part we state a theorem for tuning parameter choice that guarantees variable selection consistency of the variant of the conservative Lasso. We discuss the assumptions needed in detail. Basically, we show that the variant of the conservative Lasso in (\ref{vcl}) fits into Corollary 1 of \cite{fantang13}. For this we assume deterministic regressors and gaussian error terms which simplifies the conditions of the following theorem a bit. The case of non-gaussianity can be handled as in Condition 3, p.544 of \cite{fantang13} but brings more notation.

Denote the set of $\lambda_n$ that result in an underfit by
\[ \Omega_{-} = \{ \lambda_n \in [\lambda_l, \lambda_u]: S_{\lambda_n} \not \supset S_{\lambda_0} \},\]
where $\lambda_0$ represents an ideal tuning parameter that provides the correct model. Thus, $S_{\lambda_0}=S_0$. $\lambda_l$ and $\lambda_u$ can be chosen as described on p.540 in \cite{fantang13}.
Denote the set of $\lambda_n$ that result in an overfit by
\[ \Omega_{+} = \{ \lambda_n \in [\lambda_l, \lambda_u]: S_{\lambda_n} \supset S_{\lambda_0}, S_{\lambda_n} \neq S_{\lambda_0} \}.\]

The following theorem shows that the $\lambda_n$ choice that minimizes $GIC$ will ensure that the variant of the conservative Lasso detects the correct model with probability approaching one. The conditions for the theorem are discussed in detail after the theorem statement.
\begin{thm}\label{thm5}
Under Conditions 1-7 below
\[ P \{ \inf_{\lambda_n \in \Omega_{-} \cup \Omega_{+}} GIC (\lambda_n) > GIC (\lambda_0) \} \to 1.\]
\end{thm}
Theorem \ref{thm5} yields that the $\lambda_n$ chosen by GIC will neither result in an underfit nor an overfit. Hence, consistent model selection is achieved.

The penalty function for each parameter is defined as $\rho_{\lambda_n} (|\beta_{j}|) = \lambda_n \tilde{w}_j |\beta_j|$ for the variant of the conservative Lasso. The partial derivative of the penalty function with respect to $\beta_j,\ j\in S_0$ evaluated at $\beta_{0,j}$ is
\begin{equation}
 sgn (\beta_{0,j}) \rho_{\lambda_n} ' (|\beta_{0,j}|).\label{pd1}
\end{equation}

{\bf Condition 1}. For each $\lambda_n$, $\rho_{\lambda_n}' (t)$ is non-increasing over $t \in (0,\infty)$.

{\bf Condition 2}. There is a  $\lambda_0 \in [ \lambda_l, \lambda_u]$ such that $S_{\lambda_0} = S_0$, and 
\[ \|\tilde{\beta}_{\lambda_0} - \beta_0 \|_2 = O_p ( n^{-\pi}),\]
with $0< \pi < 1/2$.

{\bf Condition 3}. $n^{\pi} \min_{j \in S_0} | \beta_{0,j}| \to \infty$, as $n \to \infty$.

{\bf Condition 4}. $\rho_{\lambda_0}' (\frac{1}{2} \min_{j \in S_0} | \beta_{0,j}|) = o(s_0^{-1/2} n^{-1/2} [\log \log(n) \log(p)]^{1/2})$.

{\bf Condition 5}. For any $S \subset \{1,2,..., p\}$ such that $|S| \le K_1$, $K_1 >s_0$, $K_1 = o(n)$ the 
minimum eigenvalue of $n^{-1} X_S' X_S$ is bounded from below by $c_1>0$, and the maximum eigenvalue is bounded from above by $1/c_1$. 

{\bf Condition 6}. The design matrix satisfies $\| X\|_{\infty} = O (n^{1/2 - \tau_1})$ with $\tau_1 \in (1/3, 1/2]$ and $\log(p)=O(n^{\kappa_1})$, for some $0< \kappa_1 < 1$.

{\bf Condition 7}. Let $\delta_n$ be as in (3.2) of \cite{fantang13}.We assume  
$\delta_n K_1^{-1} \sqrt{n/\log(p)} \to \infty$, and \[ n \delta_n/(s_0 \log\log(n) \log(p)) \to \infty.\]

Conditions 1-3 are Condition 4 in p.544 of \cite{fantang13}. Our Condition 4 is in the statement of Proposition 1 on p.535 of \cite{fantang13}. Condition 5 here is Condition 2 on p.544 of \cite{fantang13}. Condition 6 is a condition on p.537 of Theorem 2 of \cite{fantang13}. Condition 7 is in p.539, Corollary 1 of  \cite{fantang13}. $\delta_n$ is a measure of the smallest signal strength of the truly relevant covariates. Conditions 5-7 are related to the linear model and have already been verified in \cite{fantang13}.

Conditions 1-7 here replace Assumptions 1-2, and the beta-min type condition in Lemma \ref{lt1}, and Theorem \ref{thm4}. Conditions 1-7 are more restrictive than Assumptions 1-2.

\textbf{Further discussion of Conditions 1-7}
We now discuss Conditions 1-7 in more detail in our setting to better understand when  Corollary 1 in \cite{fantang13} applies. 

Let us start by verifying Condition 1. For all $t \in (0, \infty)$, the variant of the conservative lasso
\[ \rho_{\lambda_n}' (t) = \lambda_n 1_{ \{ |\hat{\beta}_{L,j}| \le \lambda_{prec} \}}.\]
which is constant in $t$. 

Regarding Condition 2, as Theorem \ref{thm1} applies to the variant of conservative Lasso as well, we get that
\[ \| \tilde{\beta} - \beta_0 \|_2 \le \| \hat{\beta} - \beta_0 \|_1 = O_p ( \lambda_n s_0).\]
In the case of deterministic regressors, and Gaussian random errors, $\lambda_n = O ( \sqrt{\log(p)/n})$, so Condition 2 will be fulfilled if $\sqrt{\log(p)/n} s_0 = O(1/n^{\pi})$ for $0 < \pi < 1/2$.  

Condition 3 is a refinement of a beta-min type condition and restricts the size of the smallest absolute value of the non-zero coefficients.  

Condition 4  is the following in case of  the variant of conservative lasso, 
\[ \rho_{\lambda_0}' (\frac{1}{2} \min_{j \in S_0} |\beta_{0,j}|) = \lambda_0 1_{ \{ \frac{1}{2} \min_{j \in S_0} |\beta_{0,j}| \le \lambda_{prec} \}}.\]
With the beta-min condition in Theorem \ref{thm4}, $\min_{j \in S_0} | \beta_{0,j}| > 2 \lambda_{prec}$, we have 
$\frac{1}{2} \min_{j \in S_0} | \beta_{0,j}| >  \lambda_{prec}$, so the indicator is always zero such that $
\rho_{\lambda_0}' (\frac{1}{2} \min_{j \in S_0} |\beta_{0,j}|)=0$ implying that Condition 4 is trivially satisfied.

Conditions 5-6 are about design of the regression and are used by \cite{fantang13} in the least squares case. They are more restrictive than our Assumption 1. Condition 7 is related to underfit of a model in least squares.

\subsection*{Appendix C}\label{AppC}
We first show why $\hat{\Theta}$ constructed by nodewise regressions is an approximate inverse of $\hat{\Sigma}$. Then we link the inverse of the population covariance matrix $\Theta$ to linear regression.

We show that 
\begin{equation*}
\|  \hat{\Theta}_j'\hat{\Sigma} - e_j' \|_{\infty} 
\leq
\frac{\lambda_{node,n}}{\hat{\tau}_j^2}.
\end{equation*}
for $j=1,...,p$ as claimed in (\ref{2.8}). First, note that
\begin{equation}
sgn(\hat{\gamma}_j)' \hat{\Gamma}_j \hat{\gamma}_j = \enVert[1]{\hat{\Gamma}_j \hat{\gamma}_j}_1,\label{2.3}
\end{equation}
where $sgn(\hat{\gamma}_j)=\del[1]{sgn(\hat{\gamma}_{j,k}),\ k=1,...,p, k\neq j}$. Therefore, postmultiplying the Karush-Kuhn-Tucker conditions (written as a row vector) of the problem (\ref{NodeCLObj}) by $\hat{\gamma}_j$ and adding  $(X_j - X_{-j} \hat{\gamma}_j)' X_j/n$ to both sides yields
\begin{equation}
\frac{(X_j - X_{-j} \hat{\gamma}_j)'(X_j - X_{-j} \hat{\gamma}_j)}{n} + \lambda_{node,n} \enVert[1]{\hat{\Gamma}_j \hat{\gamma}_j}_1 
=
 \frac{(X_j - X_{-j} \hat{\gamma}_j)' X_j}{n}.\label{2.4}
\end{equation}
Next, we recognize the left hand side of (\ref{2.4}) as $\hat{\tau}_j^2$ such that
\begin{equation}
\hat{\tau}_j^2
=
\frac{(X_j - X_{-j} \hat{\gamma}_j)' X_j}{n}\label{2.5}.
\end{equation}
Dividing each side of the above display by $\hat{\tau}_j^2$ (we shall later rigorously argue that $\hat{\tau}_j^2$ is bounded away from zero with high probability) and using the definition of $\hat{\Theta}_j$ implies that 
\begin{equation}
1
=
 \frac{(X_j - X_{-j} \hat{\gamma}_j)' X_j}{\hat{\tau}_j^2n}
=
 \frac{(X\hat{\Theta}_j)'X_j}{n}
=
\frac{\hat{\Theta}_j'X'X_j}{n},\label{2.6}
\end{equation}
which shows that the $j$'th diagonal element of $\hat{\Theta}\hat{\Sigma}$ equals exactly one. It remains to consider the off-diagonal elements of $\hat{\Theta}\hat{\Sigma}$. To this end, note that the Karush-Kuhn-Tucker conditions for the problem (\ref{NodeCLObj}) can be written as
\[ \hat{\kappa}_j = \frac{\hat{\Gamma}_j^{-1} X_{-j}' (X_j - X_{-j} \hat{\gamma}_j)}{n \lambda_{node,n}}.\]
Using $\| \hat{\kappa}_j \|_{\infty} \le 1$ yields
\begin{align*}
\enVert[3]{\frac{\hat{\Gamma}_j^{-1} X_{-j}' (X_j - X_{-j} \hat{\gamma}_j)}{n \lambda_{node,n}}}_{\infty} 
= 
\enVert[0]{\hat{\kappa}_j}
\le 1,
\end{align*}
which is equivalent to 
\[  \frac{\| \hat{\Gamma}_j^{-1} X_{-j}' X \hat{C}_j \|_{\infty} }{n} \le \lambda_{node,n},\]
since $ (X_j - X_{-j} \hat{\gamma}_j) = X \hat{C}_j$. Then, dividing both sides of the above display by $\hat{\tau}_j^2$ and using that $\hat{\Theta}_j=\frac{\hat{C}_j}{\hat{\tau}_j^2}$ implies that
\[ \frac{\| \hat{\Gamma}_j^{-1} X_{-j}' X \hat{\Theta}_j \|_{\infty} }{n} \le \frac{\lambda_{node,n}}{\hat{\tau}_j^2}.\]
Thus,  
\begin{equation}
\frac{ \| X_{-j}' X \hat{\Theta}_j\|_{\infty}}{n} 
=
\frac{ \| \hat{\Gamma}_j \hat{\Gamma}_j^{-1} X_{-j}' X \hat{\Theta}_j\|_\infty}{n}
\le   
 \| \hat{\Gamma}_j \|_{\ell_\infty}  \frac{\| \hat{\Gamma}_j^{-1}   X_{-j}' X \hat{\Theta}_j\|_{\infty}}{n}   
\leq
\frac{\lambda_{node,n}}{\hat{\tau}_j^2},\label{2.7a}
\end{equation}
where we have used that $ \| \hat{\Gamma}_j \|_{\ell_\infty}$ equals the largest diagonal element of $\hat{\Gamma}_j$ since $\hat{\Gamma}_j$ is diagonal and that all diagonal elements are less than one by observation 2 after (\ref{ConsLassoObj}). Of course (\ref{2.7a}) is equivalent to
\begin{align}
\frac{\enVert[1]{\hat{\Theta}_j'X'X_{-j}}_\infty}{n}
\leq
\frac{\lambda_{node,n}}{\hat{\tau}_j^2}.\label{2.7}
\end{align}
In total, denoting by $e_j$ the $j$'th $p\times 1$ unit vector, (\ref{2.6}) and (\ref{2.7}) yield
\begin{equation*}
\|  \hat{\Theta}_j'\hat{\Sigma} - e_j' \|_{\infty} 
\leq
\frac{\lambda_{node,n}}{\hat{\tau}_j^2}.
\end{equation*}

\bibliographystyle{chicagoa}	
\bibliography{references-2-2-2-4-4}		


\end{document}